\documentclass[12pt,a4paper]{article}
\usepackage{amssymb}
\usepackage{amsmath, amscd}
\usepackage{amsthm}
\usepackage{mathtools}
\usepackage{mathrsfs}
\usepackage{booktabs}
\usepackage{hyperref}
\usepackage{perpage}
\usepackage[all,cmtip]{xy}
\usepackage{setspace}
\usepackage{amsbsy}
\usepackage[T1]{fontenc}
\usepackage{verbatim}
\usepackage{mathdots}
\usepackage{mathrsfs}
\usepackage{ifthen}
\usepackage{blkarray}
\usepackage{multirow}
\usepackage{enumerate}
\usepackage[dvipsnames]{xcolor}
\usepackage{tikz}
\usepackage{fullpage}
\usepackage{colonequals}
\usepackage{braket}
\usepackage{lipsum}

\mathtoolsset{showonlyrefs=true}

\numberwithin{equation}{subsection}

\newtheorem{theorem}{Theorem}[subsection]
\newtheorem{lemma}[theorem]{Lemma}

\newtheorem{example}[theorem]{Example}
\newtheorem{corollary}[theorem]{Corollary}
\newtheorem{definition}[theorem]{Definition}

\newtheorem{proposition}[theorem]{Proposition}
\newtheorem{condition}[theorem]{Condition}

\newtheorem*{thm1}{Theorem 1}
\newtheorem*{thm2}{Theorem 2}

\theoremstyle{remark}

\newtheorem{rmk}[theorem]{Remark}

\newcommand{\GZip}{\mathop{\text{$G$-{\tt Zip}}}\nolimits}

\newcommand{\GF}{\mathop{\text{$G$-{\tt ZipFlag}}}\nolimits}

\newcommand{\VB}{\mathfrak{VB}}

\newskip\procskipamount
\newskip\interskipamount
\newskip\refskipamount
\newcommand{\procskip}{\vskip\procskipamount}
\newcommand{\interskip}{\vskip\interskipamount}
\newcommand{\refskip}{\vskip\refskipamount}

\newcommand{\procbreak}{\par
   \ifdim\lastskip<\procskipamount\removelastskip
   \penalty-100
   \procskip\fi
   \noindent\ignorespaces}

\newcommand{\titlebreak}{\par%
\ifdim\lastskip<\interskipamount\removelastskip%
\penalty10000%
\interskip\fi%
\noindent}%

\newcommand{\interbreak}{\par%
\ifdim\lastskip<\interskipamount\removelastskip%
\penalty-100%
\interskip\fi%
\noindent\ignorespaces}%

\newcommand{\refbreak}{\par%
\ifdim\lastskip<\refskipamount\removelastskip%
\penalty-100%
\refskip\fi%
\noindent\ignorespaces}%

\newcounter{listcounter}
\newcounter{deflistcounter}
\newcounter{equivcounter}

\newskip{\itemsepamount}
\newskip{\topsepamount}
\newenvironment{assertionlist}{%
  \begin{list}
    {\upshape (\arabic{listcounter})}
    {\setlength{\leftmargin}{18pt}
     \setlength{\rightmargin}{0pt}
     \setlength{\itemindent}{0pt}
     \setlength{\labelsep}{5pt}
     \setlength{\labelwidth}{13pt}
     \setlength{\listparindent}{\parindent}
     \setlength{\parsep}{0pt}
     \setlength{\itemsep}{\itemsepamount}
     \setlength{\topsep}{\topsepamount}
     \usecounter{listcounter}}}
  {\end{list}}

\newenvironment{definitionlist}{%
  \begin{list}
    {\upshape (\alph{deflistcounter})}
    {\setlength{\leftmargin}{18pt}
     \setlength{\rightmargin}{0pt}
     \setlength{\itemindent}{0pt}
     \setlength{\labelsep}{5pt}
     \setlength{\labelwidth}{13pt}
     \setlength{\listparindent}{\parindent}
     \setlength{\parsep}{0pt}
     \setlength{\itemsep}{\itemsepamount}
     \setlength{\topsep}{\topsepamount}
     \usecounter{deflistcounter}}}
  {\end{list}}

\newenvironment{equivlist}{%
  \begin{list}
    {\upshape (\roman{equivcounter})}
    {\setlength{\leftmargin}{18pt}
     \setlength{\rightmargin}{0pt}
     \setlength{\itemindent}{0pt}
     \setlength{\labelsep}{5pt}
     \setlength{\labelwidth}{13pt}
     \setlength{\listparindent}{\parindent}
     \setlength{\parsep}{0pt}
     \setlength{\itemsep}{\itemsepamount}
     \setlength{\topsep}{\topsepamount}
     \usecounter{equivcounter}}}
  {\end{list}}

\newenvironment{bulletlist}{%
  \begin{list}
    {\upshape \textbullet}
    {\setlength{\leftmargin}{18pt}
     \setlength{\rightmargin}{0pt}
     \setlength{\itemindent}{0pt}
     \setlength{\labelsep}{6pt}
     \setlength{\labelwidth}{12pt}
     \setlength{\listparindent}{\parindent}
     \setlength{\parsep}{0pt}
     \setlength{\itemsep}{\itemsepamount}
     \setlength{\topsep}{\topsepamount}}}
  {\end{list}}

% ...

\newcommand{\Acal}{{\mathcal A}}
\newcommand{\Bcal}{{\mathcal B}}
\newcommand{\Ccal}{{\mathcal C}}

\newcommand{\Fcal}{{\mathcal F}}
\newcommand{\Gcal}{{\mathcal G}}
\newcommand{\Hcal}{{\mathcal H}}

\newcommand{\Lcal}{{\mathcal L}}
\newcommand{\Mcal}{{\mathcal M}}

\newcommand{\Ocal}{{\mathcal O}}

\newcommand{\Scal}{{\mathcal S}}

\newcommand{\Ucal}{{\mathcal U}}
\newcommand{\Vcal}{{\mathcal V}}

\newcommand{\Xcal}{{\mathcal X}}

\newcommand{\Zcal}{{\mathcal Z}}

\newcommand{\pfr}{{\mathfrak p}}

\newcommand{\vfr}{{\mathfrak v}}

\newcommand{\Sfr}{{\mathfrak S}}

\renewcommand{\AA}{\mathbb{A}}

\newcommand{\CC}{\mathbb{C}}

\newcommand{\FF}{\mathbb{F}}
\newcommand{\GG}{\mathbb{G}}

\newcommand{\NN}{\mathbb{N}}

\newcommand{\QQ}{\mathbb{Q}}
\newcommand{\RR}{\mathbb{R}}

\newcommand{\ZZ}{\mathbb{Z}}

\DeclareMathOperator{\Pic}{Pic}

\newcommand{\Ascr}{{\mathscr A}}

\newcommand{\Lscr}{{\mathscr L}}

\newcommand{\Pscr}{{\mathscr P}}

\newcommand{\Sscr}{{\mathscr S}}

\DeclareMathOperator{\ad}{ad}

\newcommand{\dR}{{\rm dR}}

\DeclareMathOperator{\Gal}{Gal}

\DeclareMathOperator{\Hom}{Hom}

\DeclareMathOperator{\rank}{rank}
\DeclareMathOperator{\Span}{Span}

\DeclareMathOperator{\pr}{pr}

\DeclareMathOperator{\Ker}{Ker}

\DeclareMathOperator{\Rep}{Rep}

\DeclareMathOperator{\Sbt}{Sbt}
\DeclareMathOperator{\Sh}{Sh}

\DeclareMathOperator{\Sym}{Sym}

\DeclareMathOperator{\GS}{GS}

\DeclareMathOperator{\SL}{SL}
\DeclareMathOperator{\GL}{GL}

\DeclareMathOperator{\GSp}{GSp}

\DeclareMathOperator{\Sp}{Sp}
\DeclareMathOperator{\U}{U}
\DeclareMathOperator{\GU}{GU}

\newcommand{\shgx}{\Sh(\mathbf G, \mathbf X)}

\newcommand{\egx}{E(\mathbf G, \mathbf X)}
\newcommand{\gx}{(\mathbf G, \mathbf X)}

\newcommand{\gofaf}{\mathbf G(\mathbf A_f)}

\newcommand{\End}{{\rm End}}

\newcommand{\id}{{\rm Id}}

\newcommand{\cf}{{\em cf. }}

\newcommand{\diag}{{\rm diag}}

\newcommand{\crys}{{\rm crys}}

\renewcommand{\Im}{{\rm Im}}

\DeclareMathOperator{\Std}{Std}

\renewcommand{\div}{{\rm div}}

\DeclareMathOperator{\flag}{flag}

\DeclareMathOperator{\Ind}{Ind}
\DeclareMathOperator{\cara}{Char}

\DeclareMathOperator{\Hasse}{Hasse}

\DeclareMathOperator{\Flag}{Flag}

\DeclareMathOperator{\ev}{ev}

\DeclareMathOperator{\Ha}{Ha}

\begin{document}

\author{Naoki Imai and Jean-Stefan Koskivirta}

\title{Partial Hasse invariants for Shimura varieties\\ of Hodge-type} 

\date{}

\maketitle

\begin{abstract}
For a connected reductive group $G$ over a finite field, we define partial Hasse invariants on the stack of $G$-zip flags. We obtain similar sections on the flag space of Shimura varieties of Hodge-type. They are mod $p$ automorphic forms which cut out a single codimension one stratum. We study their properties and show that such invariants admit a natural factorization through higher rank automorphic vector bundles. We define the socle of an automorphic vector bundle, and show that partial Hasse invariants lie in this socle.
\end{abstract}

\section{Introduction}

Originally, partial Hasse invariants were defined for Hilbert--Blumenthal Shimura varieties by Goren and Andreatta--Goren in \cite{Goren-partial-hasse}, \cite{Andreatta-Goren-book} as certain Hilbert modular forms in characteristic $p$. The vanishing locus of each partial Hasse invariant coincides with the closure of a codimension one Ekedahl--Oort stratum. The product of all these sections is the classical Hasse invariant, which vanishes exactly outside the ordinary locus. Similar constructions can be found in \cite{Hernandez-invariants-de-Hasse}, \cite{bijakowski-partial-Hasse}, \cite{Bijakowski-Hernandez-Groupes-p-divisibles}. In the case of Hilbert modular varieties, Diamond--Kassaei have shown in \cite{Diamond-Kassaei, Diamond-Kassaei-cone-minimal} that partial Hasse invariants play an important role in the theory of mod $p$ automorphic forms. View the weight of a Hilbert modular form as a tuple $(a_1, \dots , a_n)\in \ZZ^n$ (where $n$ is the degree of the totally real extension defining the corresponding Hilbert--Blumenthal Shimura variety). Diamond--Kassaei define the minimal cone $C_{\min}\subset \ZZ^n$ and show the following: 
\begin{assertionlist}
\item Any automorphic form whose weight lies outside of $C_{\min}$ is divisible by (a specific) partial Hasse invariant.
\item The cone $C_{\Hasse}$ generated by the weights of all partial Hasse invariants spans (over $\QQ_{\geq 0}$) all  possible weights of Hilbert modular forms in characteristic $p$.
\end{assertionlist}
In \cite{Goldring-Koskivirta-global-sections-compositio}, Goldring and the second named author showed that (2) also holds for Picard modular surfaces and Siegel threefolds (conditionally on a Koecher principle for strata that should follow from \cite{Lan-Stroh-stratifications-compactifications}). The present paper provides a general theory of partial Hasse invariants. It serves as a preliminary to a follow-up paper (\cite{Goldring-Imai-Koskivirta-weights}) in which we discuss generalizations of Diamond--Kassaei's result to other Shimura varieties. In another joint paper by Goldring and the second author, we show that both (1) and (2) generalize naturally to Picard Shimura varieties (at split or inert primes), Siegel Shimura varieties of rank $2$ and $3$.

Let $(\mathbf{G},\mathbf{X})$ be a Shimura datum of Hodge-type and $Sh_K(\mathbf{G},\mathbf{X})$ the corresponding Shimura variety with level $K$ over a number field $\mathbf{E}$. Let $\mu\colon \GG_{\mathrm{m},\CC}\to \mathbf{G}_{\CC}$ be the cocharacter deduced from $\mathbf{X}$. Write $\mathbf{L}\subset \mathbf{G}_\CC$ for the Levi subgroup centralizing $\mu$. Choose a Borel pair $(\mathbf{B},\mathbf{T})$ such that $\mu$ factors through $\mathbf{T}$. Denote by $I\colonequals \Delta_{\mathbf{L}}$ the simple roots of $\mathbf{L}$ with respect to the opposite Borel of $\mathbf{B}$. 
For any $\lambda\in X^*(\mathbf{T})$, there is a vector bundle $\Vcal_I(\lambda)$ on $Sh_K(\mathbf{G},\mathbf{X})$, modeled on the $\mathbf{L}$-representation $\Ind_{\mathbf{B_L}}^{\mathbf{L}}(\lambda)$, where $\mathbf{B_L}\colonequals \mathbf{B}\cap \mathbf{L}$. Let $p$ be a prime of good reduction, and let $\Scal_K$ be the smooth canonical model over $\Ocal_{\mathbf{E}_\pfr}$ (where $\pfr|p$) constructed by Kisin (\cite{Kisin-Hodge-Type-Shimura}) and Vasiu (\cite{Vasiu-Preabelian-integral-canonical-models}). The vector bundle $\Vcal_I(\lambda)$ extends naturally to $\Scal_K$. We define partial Hasse invariants as certain sections of $\Vcal_I(\lambda)$ over $S_K\colonequals \Scal_K\otimes_{\Ocal_{\mathbf{E}_\pfr}}\overline{\FF}_p$. More precisely, let $G$ be the special fiber of a reductive $\ZZ_p$-model of $\mathbf{G}_{\QQ_p}$, and write $B,T,L$ for the subgroups of $G$ naturally induced by $\mathbf{B}, \mathbf{T}, \mathbf{L}$. Moonen--Wedhorn and Pink--Wedhorn--Ziegler define the stack of $G$-zips of type $\mu$ denoted by $\GZip^\mu$. By Zhang (\cite{Zhang-EO-Hodge}), there is a smooth, surjective map $\zeta\colon S_K\to \GZip^\mu$. We construct partial Hasse invariants on this stack.

More generally, one can start with an arbitrary connected, reductive group $G$ over $\FF_q$, endowed with a cocharacter $\mu\colon \GG_{\textrm{m},\overline{\FF}_q}\to G_{\overline{\FF}_q}$ (in the context of Shimura varieties, we take $q=p$). Goldring and the second named author defined in \cite{Goldring-Koskivirta-Strata-Hasse} the stack of $G$-zip flags $\GF^\mu$. One can define similarly the flag space $\Flag(S_K)$ of $S_K$. There is a natural map $\pi\colon \GF^\mu \to \GZip^\mu$ whose fibers are flag varieties. For $\lambda\in X^*(T)$, there is a line bundle $\Vcal_{\flag}(\lambda)$ on $\GF^\mu$ and one has $\pi_*(\Vcal_{\flag}(\lambda))=\Vcal_I(\lambda)$. In particular, we may identify global sections of $\Vcal_I(\lambda)$ and $\Vcal_{\flag}(\lambda)$. Furthermore, $\GF^\mu$ admits a natural stratification $(\Ccal_{w})_{w\in W}$ (where $W=W(G,T)$ is the Weyl group), similar to the Bruhat decomposition of $G$. The codimension one strata are $(\Ccal_{s_\alpha w_0})_{\alpha \in \Delta}$, where $\Delta$ is the set of simple roots, $s_\alpha$ is the reflection along $\alpha$, and $w_0\in W$ is the longest element. Then, a (flag) partial Hasse invariant for $\alpha$ is a section $\Ha_\alpha\in H^0(\GF^\mu,\Vcal_{\flag}(\lambda_\alpha))$ whose vanishing locus is the Zariski closure of $\Ccal_{w_0s_\alpha}$. A simple argument (Proposition \ref{prop-set-of-flagPHI}) shows that such sections always exist. More generally, for any dominant $\chi\in X^*(T)$, we define a section $\Ha_\chi$ of $\Vcal_{\flag}(\lambda_\chi)$ where $\lambda_\chi = \chi -qw_{0,I}(\sigma^{-1}\chi)$ 
(here $\sigma$ denotes the action of Frobenius and $w_{0,I}$ is the longest element in $W_L=W(L,T)$). 
We call $\Ha_\chi$ a Schubert section. If $\chi_\alpha$ is orthogonal to $\Delta\setminus \{\alpha\}$ and $\langle \chi_\alpha,\alpha^\vee\rangle >0$, then $\Ha_\alpha \colonequals \Ha_{\chi_\alpha}$ is a partial Hasse invariant for $\alpha$.

We now explain the main results of the paper. First, we construct group-theoretical Verschiebung homomorphisms for vector bundles on $\GZip^\mu$. We first explain the analogy with abelian schemes. Let $A$ be an abelian scheme over an $\FF_q$-scheme $S$, and let $\Omega=e^*(\Omega_{A/S})$ denote the Hodge bundle (where $e \colon S\to A$ is the unit section). It is a subbundle of $\Mcal\colonequals H^1_{\dR}(A/S)$. Then the Verschiebung $V\colon A^{(q)}\to A$ induces maps $\vfr_{\Omega}\colon \Omega\to \Omega^{(q)}$ and $\vfr_\Mcal\colon\Mcal\to \Mcal^{(q)}$. Furthermore, the image of $\vfr_\Mcal$ is $\Omega^{(q)}$ and $\vfr_\Mcal$ extends the map $\vfr_\Omega$. We construct such maps in a quite general setting. Let $(G,\mu)$ be an $\FF_q$-reductive group $G$ endowed with a cocharacter $\mu\colon \GG_{\textrm{m},\overline{\FF}_q}\to G_{\overline{\FF}_q}$. Let $L$ be the centralizer of $\mu$, and assume that it is defined over $\FF_q$. For an $L$-representation $(V,\rho)$, we define the Griffiths--Schmid subspace $V_{\GS}\subset V$ (see \S\ref{subsec-verschiebung}). When $(V,\rho)$ satisfies $V=V_{\GS}$, we show that the vector bundle $\Vcal(\rho)$ (attached to $\rho$) on $\GZip^\mu$ admits a natural Verschiebung map
\begin{equation}
 \vfr_\rho\colon \Vcal(\rho)\to \Vcal(\rho^{[1]})   
\end{equation}
where $\rho^{[1]}$ is the composition of $\rho$ with the Frobenius homomorphism $\varphi\colon L\to L$. This map can be viewed as a generalization of the map $\vfr_\Omega$. 
This generalization will be useful to extend and study objects related to usual Verschiebung maps. Actually, we will explain below that this generalized Verschiebung gives factorizations of partial Hasse invariants.

In the particular case when $V=V_I(w_{0,I}w_0\chi)$ with $\chi\in X^*(T)$ dominant, $V$ satisfies $V_{\GS}=V$, hence we obtain a Verschiebung map $\vfr_\chi\colon \Vcal_I(w_{0,I}w_0\chi)\to \Vcal_I(w_{0,I}w_0\chi^{[1]})$. By Proposition \ref{prop-PiX-isom}, $V_I(w_{0,I}w_0\chi)$ is naturally a sub-$L$-representation of the $G$-representation $V_\Delta(\chi)=\Ind_{B}^G(\lambda)$. 
Denote by $\Vcal_\Delta(\chi)$ the attached vector bundle on $\GZip^\mu$. Similarly to $\vfr_\Mcal$ extending $\vfr_\Omega$ for abelian schemes, we prove the following:

\begin{thm1}[Theorem \ref{ver-thmG}]\label{ver-thmG-intro}
Let $\chi\in X^*(T)$ be a dominant character. There exists a map of vector bundles $\vfr_\chi \colon  \Vcal_{\Delta}(\chi) \to \Vcal_{\Delta}(\chi^{[1]})$ over $\GZip^{\mu}$ with image $\Vcal_I(w_{0,I}w_0\chi^{[1]})$ and such that $\vfr_\chi$ extends the Verschiebung map of $\Vcal_I(w_{0,I}w_0\chi)$. 
$$\xymatrix@1@M=5pt@H=0pt{
\Vcal_{\Delta}(\chi) \ar[r]^-{\vfr_\chi} & \Vcal_I(w_{0,I}w_0\chi^{[1]})  
\ar@{^{(}->}[r] 
& \Vcal_{\Delta}(\chi^{[1]}).\\
\Vcal_I(w_{0,I}w_0\chi)  
\ar@{^{(}->}[u] 
\ar[ru]_-{\vfr_{\chi}} & &
}$$
\end{thm1}

We call $\vfr_\chi$ the Verschiebung homomorphism of $\Vcal_{\Delta}(\chi)$. The motivation for constructing such maps is that Schubert sections (in particular, partial Hasse invariants) admit a factorization in terms of Verschiebung homomorphisms. Specifically, let $\chi$ be dominant and set $\Vcal=\Vcal_I(-w_{0,I}\chi)$. Define $\Vcal^{[1]}\colonequals \Vcal_I(-w_{0,I}\chi^{[1]})$ and write $\vfr\colon \Vcal\to \Vcal^{[1]}$ for the Verschiebung map provided by Theorem 1. We show (Corollary \ref{cor-decomp-Schubert-Fq}) that there are natural maps of vector bundles over $\GF^\mu$:
\begin{equation}\label{eq-cor-decompos-Schubert-sec-Fq-intro}
\Vcal_{\flag}(-\chi) \to \pi^*\Vcal \xrightarrow{\pi^*\vfr} \pi^*\Vcal^{[1]} \to  \Vcal_{\flag}(-w_{0,I}(\sigma^{-1}\chi))^{\otimes q}
\end{equation}
such that the composition coincides with $\Ha_{\chi}$ as a section of $\Vcal_{\flag}(\lambda_{\chi})$. In particular, this factorization applies to partial Hasse invariants.

Lastly, we prove that partial Hasse invariants are primitive automorphic forms. Here, we do not assume that $P$ is defined over $\FF_q$. For an $L$-dominant $\lambda$, denote by $L_I(\lambda)\subset V_I(\lambda)$ the socle of $V_I(\lambda)$. Write $\Vcal_I^L(\lambda)$ for the vector bundle attached to $L_I(\lambda)$. We call sections of $\Vcal_I^L(\lambda)$ primitive automorphic forms of weight $\lambda$. We explain the reason for studying such forms. We would like to understand the ring of automorphic forms on $\GZip^\mu$, which is defined as
\begin{equation}
   R_{I}(G,\mu) \colonequals \bigoplus_{\lambda\in X_{+,I}^*(T)}H^0(\GZip^\mu,\Vcal_I(\lambda)).
\end{equation}
This $k$-algebra was defined in \cite[\S5]{Koskivirta-automforms-GZip} (see also \cite[\S6.1]{Imai-Koskivirta-vector-bundles}). It is a functorial invariant attached to the pair $(G,\mu)$. In general, it is very difficult to determine $R_I(G,\mu)$, but we conjectured that it is of finite-type. We can define a subgroup $R_I^L(G,\mu)$ as the direct sum of $H^0(\GZip^\mu,\Vcal^L_I(\lambda))$ for $\lambda\in X_{+,I}^*(T)$. The subgroup $R_I^L(G,\mu)$ is a more tractable object, because $L_I(\lambda)$ is a simple object and we can use Steinberg's tensor product theorem to study it. Furthermore, each vector bundle $\Vcal_I(\lambda)$ admits a filtration whose graded pieces are of the form $\Vcal^L_I(\lambda')$, which gives a relation between $R_I^L(G,\mu)$ and $R_I(G,\mu)$. Our main result is that partial Hasse invariants lie in $R_I^L(G,\mu)$ (under a certain assumption, which is satisfied in most cases). Specifically, assume that for each $\alpha\in \Delta$, there exists a character $\chi_{\alpha}\in X^*(T)$ satisfying
\begin{enumerate}[(a)]
\setlength{\leftmargin}{18pt}
     \setlength{\rightmargin}{0pt}
     \setlength{\itemindent}{0pt}
     \setlength{\labelsep}{5pt}
     \setlength{\labelwidth}{13pt}
     \setlength{\listparindent}{\parindent}
     \setlength{\parsep}{0pt}
     \setlength{\itemsep}{\itemsepamount}
     \setlength{\topsep}{\topsepamount}
\item $0 < \langle \chi_{\alpha},\alpha^\vee \rangle<q$ and $\langle \chi_{\alpha}, \beta^{\vee}\rangle =0$ for all $\beta\in \Delta \setminus \{\alpha\}$,
\item one has $L_I(\chi_{\alpha})=V_I(\chi_{\alpha})$ and $L_I(-w_{0,I}(\sigma^{-1}\chi_{\alpha}))=V_I(-w_{0,I}(\sigma^{-1}\chi_{\alpha}))$.
\end{enumerate}
For example, for $G=\Sp(2n)_{\FF_q}$, 
the fundamental weights satisfy (a) and (b) (even for $q=2$).

\begin{thm2}[Theorem \ref{thm-LIlambda}]\label{intro-thm-LIlambda}
Let $\alpha\in \Delta$ and let $\chi_{\alpha}\in X^*(T)$ satisfying \textnormal{(a)} and \textnormal{(b)} above. Set $\lambda_\alpha \colonequals \chi_{\alpha}-qw_{0,I}(\sigma^{-1}\chi_{\alpha})$. 
Then there exists a section $\Ha_\alpha$ over $\GF^\mu$ of the line bundle $\Vcal_{\flag}(\lambda_\alpha)$, such that
\begin{assertionlist}
\item $\Ha_\alpha$ is a flag partial Hasse invariant for $\alpha$, 
\item $\Ha_\alpha$ is a primitive automorphic form on $\GZip^\mu$.
\end{assertionlist}
\end{thm2}
Finally, we give in \S\ref{sec-examples} the modular interpretation of partial Hasse invariants in the case of Siegel-type and unitary Shimura varieties (at inert and split primes).

\subsection*{Acknowledgments}
We would like to thank the referee for their useful comments on the paper. This work was supported by JSPS KAKENHI Grant Number 21K13765 and 22H00093. 

\section{\texorpdfstring{Preliminaries and reminders on the stack of $G$-zips}{}}

\subsection{Notation}\label{subsec-notation}

Let $p$ be a prime number and $q$ a power of $p$. 
Let $\FF_q$ denote a finite field with $q$ elements. 
We write $k$ for an algebraic closure $\overline{\FF}_q$ of $\FF_q$. Let $G$ be a connected reductive group over $\FF_q$. For a $k$-scheme $X$, we denote by $X^{(q)}$ its $q$-th power Frobenius twist and by $\varphi \colon X\to X^{(q)}$ its relative Frobenius morphism. Write $\sigma \in \Gal(k/\FF_q)$ for the $q$-power Frobenius. We write $(B,T)$ for a Borel pair of $G$ defined over $\FF_q$, i.e.\ $T$ is a maximal torus, $B$ a Borel subgroup in $G$ and $T \subset B$. We do not assume that $T$ is split over $\FF_q$. Let $B^+$ be the Borel subgroup of $G_k$  opposite to $B$ with respect to $T$, which is the unique Borel subgroup such that $B^+\cap B=T$. We will use the following notations:
\begin{bulletlist}
\item $X^*(T)$ (resp.\ $X_*(T)$) denotes the group of characters (resp.\ cocharacters) of $T$. The group $\Gal(k/\FF_q)$ acts naturally on these groups. Let $W=W(G_k,T)$ be the Weyl group of $G_k$. Similarly, $\Gal(k/\FF_q)$ acts on $W$, compatibly with the action of $W$ on characters and cocharacters.
\item $\Phi\subset X^*(T)$ is the set of $T$-roots of $G$, and $\Phi_+\subset \Phi$ is the set of positive roots with respect to $B^+$ (i.e.\ $\alpha \in \Phi_+$ if the $\alpha$-root group $U_{\alpha}$ is contained in $B^+$). This convention differs from other authors. We use it to match conventions of \cite{Goldring-Koskivirta-Strata-Hasse}, \cite{Koskivirta-automforms-GZip}.
\item $\Delta\subset \Phi_+$ is the set of simple roots. 
\item For $\alpha \in \Phi$, let $s_\alpha \in W$ be the corresponding reflection. Then $(W,\{s_\alpha \mid \alpha \in \Delta\})$ is a Coxeter system. 
Write $\ell  \colon W\to \NN$ for the length function and $\leq$ for the Bruhat order on $W$. We have $\ell(s_\alpha)=1$ for all $\alpha\in \Delta$. Let $w_0$ denote the longest element of $W$.
\item Let $K$ be a subset of $\Delta$. We write $W_K$ for the subgroup of $W$ generated by $\{s_\alpha \mid \alpha \in K\}$. Let $w_{0,K}$ be the longest element in $W_K$.
Let ${}^KW$ (resp.\ $W^K$) denote the subset of elements $w\in W$ which have minimal length in the coset $W_K w$ (resp.\ $wW_K$). Then ${}^K W$ (resp.\ $W^K$) is a set of representatives of $W_K\backslash W$ (resp.\ $W/W_K$). The map $g\mapsto g^{-1}$ induces a bijection ${}^K W\to W^K$. The longest element in the set ${}^K W$ is $w_{0,K} w_0$.
\item $X_{+}^*(T)$ denotes the set of dominant characters, i.e.\ characters $\lambda\in X^*(T)$ such that $\langle \lambda,\alpha^\vee \rangle \geq 0$ for all $\alpha \in \Delta$.
\item For a subset $I\subset \Delta$, let $X_{+,I}^*(T)$ denote the set of characters $\lambda\in X^*(T)$ such that $\langle \lambda,\alpha^\vee \rangle \geq 0$ for all $\alpha \in I$. We call them $I$-dominant characters.
\item Let $P\subset G_k$ be a parabolic subgroup containing $B$ and let $L\subset P$ be the unique Levi subgroup of $P$ containing $T$. Define $I_P\subset \Delta$ as the unique subset such that $W(L,T)=W_{I_P}$. For an arbitrary parabolic subgroup $P\subset G_k$ containing $T$, define $I_P\subset \Delta$ by $I_P \colonequals I_{P'}$ where $P'$ is the unique conjugate of $P$ containing $B$. Moreover, set  $\Delta^P \colonequals \Delta \setminus I_P$.
\item For all $\alpha\in \Phi$, choose an isomorphism $u_\alpha\colon \GG_{\mathrm{a}}\to U_\alpha$ so that 
$(u_{\alpha})_{\alpha \in \Phi}$ is a realization in the sense of \cite[8.1.4]{Springer-Linear-Algebraic-Groups-book}. In particular, we have 
\begin{equation}\label{eq:phiconj}
 t u_{\alpha}(x)t^{-1}=u_{\alpha}(\alpha(t)x), \quad \forall x\in \GG_{\mathrm{a}},\  \forall t\in T.
\end{equation}
\item Let $\phi_{\alpha}\colon \SL_2\to G$ denote the map attached to $\alpha$, as in \cite[9.2.2]{Springer-Linear-Algebraic-Groups-book}. It satisfies
\[
 \phi_\alpha 
 \left( \begin{pmatrix}
 1 & x \\ 0 & 1 
 \end{pmatrix}\right) = u_{\alpha}(x), \quad 
 \phi_\alpha 
 \left( \begin{pmatrix}
 1 & 0 \\ x & 1 
 \end{pmatrix}\right) = u_{-\alpha}(x).
\]
\item Fix a $B$-representation $(V,\rho)$. For $j\in \ZZ$ and $\alpha\in \Phi$, we define a map $E_{\alpha}^{(j)} \colon V \to V$ as follows. Let $V=\bigoplus_{\nu \in X^*(T)}V_\nu$ be the weight decomposition of $V$. For $v\in V_\nu$, we can write uniquely
\[
 u_{\alpha}(x)v=\sum_{j \geq 0} x^j E_{\alpha}^{(j)}(v), \quad \forall x\in \GG_{\mathrm{a}},
\]
for elements $E_{\alpha}^{(j)}(v) \in V_{\nu+j\alpha}$ (\cite[Proposition 3.3.2]{Donkin-good-filtrations-LNM}). Extend $E_{\alpha}^{(j)}$ by additivity to a map $V\to V$. 
For $j<0$, we put $E_{\alpha}^{(j)}=0$.

\end{bulletlist}

\subsection{\texorpdfstring{The stack of $G$-zips}{}}

We recall some facts about the stack of $G$-zips of Pink--Wedhorn--Ziegler in \cite{Pink-Wedhorn-Ziegler-zip-data}.

\subsubsection{Definitions} \label{subsec-zipdatum}
Let $G$ be a connected, reductive group over $\FF_q$. A zip datum means a tuple $\Zcal \colonequals (G,P,L,Q,M,\varphi)$ consisting of the following objects:
\begin{assertionlist}
    \item $P, Q\subset G_k$ are parabolic subgroups of $G_k$.
    \item $L\subset P$ and $M\subset Q$ are Levi subgroups such that $L^{(q)}=M$.
\end{assertionlist}
 
For an algebraic group $H$, let $R_{\mathrm{u}}(H)$ denote the unipotent radical of $H$. If $P'\subset G_k$ is a parabolic subgroup with Levi subgroup $L'\subset P'$, any $x\in P'$ can be written uniquely as $x=\overline{x}u$ with $\overline{x}\in L'$ and $u\in R_{\mathrm{u}}(P')$. Denote by $\theta^{P'}_{L'} \colon P'\to L'$ the map $x\mapsto \overline{x}$. Since $M=L^{(q)}$, we have a Frobenius isogeny $\varphi \colon L\to M$. The zip group is the subgroup of $P\times Q$ defined by
\begin{equation}\label{zipgroup}
E \colonequals \{(x,y)\in P\times Q \mid  \varphi(\theta^P_L(x))=\theta^Q_M(y)\}.
\end{equation}
Then $E$ is the subgroup of $P\times Q$ generated by $R_{\mathrm{u}}(P)\times R_{\mathrm{u}}(Q)$ and elements of the form $(a,\varphi(a))$ with $a\in L$. 
Let $G\times G$ act on $G$ by $(a,b)\cdot g \colonequals agb^{-1}$, and let $E$ act on $G$ by restricting this action to $E$. 
We define the stack of $G$-zips of type $\Zcal$ (\cite{Pink-Wedhorn-Ziegler-zip-data},\cite{Pink-Wedhorn-Ziegler-F-Zips-additional-structure}) as the quotient stack
\[
\GZip^\Zcal = \left[E\backslash G_k \right].
\]
The stack $\GZip^\Zcal$ is the stack over $k$ such that for any $k$-scheme $S$, the groupoid $\GZip^{\Zcal}(S)$
is the category of tuples $\underline{I}=(I,I_P,I_Q,\iota)$, where $I$ is a $G$-torsor over $S$, $I_P\subset I$ and $I_Q\subset I$ are respectively a $P$-subtorsor and a $Q$-subtorsor of $I$, and $\iota \colon (I_P/R_{\mathrm{u}}(P))^{(p)}\to I_Q/R_{\mathrm{u}}(Q)$ is an isomorphism of $M$-torsors.

\subsubsection{Cocharacter datum} \label{subsec-cochar}
A cocharacter datum is a pair $(G,\mu)$ where $G$ is a connected, reductive group over $\FF_q$ and $\mu \colon \GG_{\mathrm{m},k}\to G_k$ is a cocharacter. One can attach to $(G,\mu)$ a zip datum $\Zcal_\mu$ as follows. Let $P_+(\mu)$ (resp.\ $P_-(\mu)$) denote the unique parabolic subgroup of $G_k$ such that $P_+(\mu)(k)$ (resp.\ $P_-(\mu)(k)$) consists of the elements $g\in G(k)$ satisfying the condition that the map 
\[
 \GG_{\mathrm{m},k} \to G_{k}; \  
 t\mapsto\mu(t)g\mu(t)^{-1} \quad (\textrm{resp.\ } t\mapsto\mu(t)^{-1}g\mu(t))
\]
extends to a morphism of varieties $\AA_{k}^1\to G_{k}$. We obtain a pair of opposite parabolics $(P_+(\mu),P_{-}(\mu))$ whose intersection $L(\mu)=P_+(\mu)\cap P_-(\mu)$ is the centralizer of $\mu$ (it is a common Levi subgroup of $P_+(\mu)$ and $P_-(\mu)$). We put $P \colonequals P_-(\mu)$, $Q \colonequals (P_+(\mu))^{(q)}$, $L \colonequals L(\mu)$ and $M \colonequals  L(\mu)^{(q)}$. The tuple $\Zcal_\mu \colonequals (G,P,L,Q,M,\varphi)$ is a zip datum, which we call the zip datum attached to the cocharacter datum $(G,\mu)$. We write simply $\GZip^\mu$ for $\GZip^{\Zcal_\mu}$. We will always consider zip data arising in this way from a cocharacter datum.

\subsubsection{Frames} \label{sec-frames}
Let $\Zcal=(G,P,Q,L,M)$ be a zip datum. In this paper, a frame for $\Zcal$ is a triple $(B,T,z)$ where $(B,T)$ is a Borel pair of $G_k$ defined over $\FF_q$ such that $B\subset P$ and 
$z\in W$ is an element such that 
\begin{equation}\label{eqBorel}
{}^z \! B \subset Q \quad \textrm{and} \quad
 B\cap M= {}^z \! B\cap M. 
\end{equation}
A frame (as defined here) may not always exist. However, if $(G,\mu)$ is a cocharacter datum and $\Zcal_\mu$ is the associated zip datum by \S\ref{subsec-cochar}, then we can find $g\in G(k)$ such that $\Zcal_{\mu'}$ for $\mu'=\ad(g)\circ \mu$ admits a frame. Hence, it is harmless to assume the existence of a frame, and we consider only zip data which admit frames. 
For a zip datum $(G,P,L,Q,M,\varphi)$, we define subsets $I,J,\Delta^P \subset \Delta$ as follows:
\begin{equation}\label{equ-IJDeltaP}
 I \colonequals I_P, \quad J \colonequals I_Q, \quad \Delta^P=\Delta\setminus I.  
\end{equation}

\begin{lemma}[{\cite[Lemma 2.3.4]{Goldring-Koskivirta-zip-flags}}]\label{lem-framemu}
Let $\mu \colon \GG_{\mathrm{m},k}\to G_k$ be a cocharacter, and let $\Zcal_\mu$ be the attached zip datum. Assume that $(B,T)$ is a Borel pair defined over $\FF_q$ such that $B\subset P$. Define the element
\[
z \colonequals w_0 w_{0,J}=\sigma(w_{0,I})w_0.
\]
Then $(B,T,z)$ is a frame for $\Zcal_\mu$.
\end{lemma}

\subsubsection{\texorpdfstring{Parametrization of the $E$-orbits in $G$}{}} \label{subsec-zipstrata}
Recall that the group $E$ defined in \eqref{zipgroup} acts on $G_k$. By \cite[Proposition 7.1]{Pink-Wedhorn-Ziegler-zip-data}, there are finitely many $E$-orbits in $G_k$. The $E$-orbits are smooth and locally closed in $G$. This gives a stratification of $G_k$, in the sense that the closure of an $E$-orbit is a union of $E$-orbits. We review below the parametrization of $E$-orbits following \cite{Pink-Wedhorn-Ziegler-zip-data}.

For $w\in W$, fix a representative $\dot{w}\in N_G(T)$, such that $(w_1w_2)^\cdot = \dot{w}_1\dot{w}_2$ whenever $\ell(w_1 w_2)=\ell(w_1)+\ell(w_2)$ (this is possible by choosing a Chevalley system, \cite[ XXIII, \S6]{SGA3}). For $w\in W$, define $G_w$ as the $E$-orbit of $\dot{w}\dot{z}^{-1}$. If no confusion occurs, we write $w$ instead of $\dot{w}$. For $w,w'\in {}^I W$, write $w'\preccurlyeq w$ if there exists $w_1\in W_L$ such that $w'\leq w_1 w \sigma(w_1)^{-1}$. This defines a partial order on ${}^I W$ (\cite[Corollary 6.3]{Pink-Wedhorn-Ziegler-zip-data}).

\begin{theorem}[{\cite[Theorem 7.5, Theorem 11.2]{Pink-Wedhorn-Ziegler-zip-data}}] \label{thm-E-orb-param}
We have two bijections:
\begin{align} \label{orbparam}
{}^I W &\to \{ \textrm{$E$-orbits in $G_k$} \}, \quad w\mapsto G_w \\  
\label{dualorbparam} W^J &\to \{ \textrm{$E$-orbits in $G_k$} \}, \quad w\mapsto G_w.
\end{align}
For $w\in {}^I W$, one has $\dim(G_w)= \ell(w)+\dim(P)$ and the Zariski closure of $G_w$ is 
\begin{equation}\label{equ-closure-rel}
\overline{G}_w=\bigsqcup_{w'\in {}^IW,\  w'\preccurlyeq w} G_{w'}.
\end{equation}
\end{theorem}
There is a unique open $E$-orbit
$U_\Zcal\subset G$ corresponding to the longest elements $w_{0,I}w_0\in {}^I W$ via \eqref{orbparam} and to $w_0w_{0,J}\in W^J$ via \eqref{dualorbparam}. If $\Zcal$ arises from a cocharacter datum (\S\ref{subsec-cochar}), we write $U_\mu$ for $U_{\Zcal_\mu}$. 
In this case we can choose $z=w_{0}w_{0,J}=\sigma(w_{0,I})w_0$ 
(Lemma \ref{lem-framemu}), hence \eqref{dualorbparam} shows that $1\in U_\mu$. Using the terminology pertaining to Shimura varieties, we call $U_\mu$ the $\mu$-ordinary stratum of $G$. The substack $\Ucal_\mu \colonequals [E\backslash U_\mu]$ will be called the $\mu$-ordinary locus. In the context of Shimura varieties of Hodge-type with good reduction, the $\mu$-ordinary locus was first defined to be the generic Newton stratum in \cite{Wedhorn-ordinariness-Shimura-varieties} and \cite{Wortmann-mu-ordinary}, but Wortmann proved in \cite{Wortmann-mu-ordinary} that it coincides with the generic Ekedahl--Oort stratum. This locus was studied for example in \cite{Moonen-Serre-Tate} and  \cite{He-Nie-mu-ordinary}.

For an $E$-orbit $G_w$ (with $w\in {}^I W$ or $w\in W^J$), we write $\Xcal_w \colonequals [E\backslash G_w]$ for the corresponding locally closed substack of $\GZip^\Zcal=[E\backslash G_k]$. We obtain similarly a stratification 
\begin{equation}\label{eq-GZip-stratification}
    \GZip^\Zcal = \bigsqcup_{w\in {}^I W} \Xcal_w
\end{equation}
and one has closure relations between strata similar to \eqref{equ-closure-rel}.

\subsection{Representation theory}\label{subsec-remind}

For an algebraic group $G$ over a field $K$, denote by $\Rep(G)$ the category of algebraic representations of $G$ on finite-dimensional $K$-vector spaces (we will mostly consider the case $K=k=\overline{\FF}_q$). We denote such a representation by $(V,\rho)$, or sometimes simply $\rho$ or $V$. 
For an algebraic group $G$ over $\FF_q$, 
a $G_k$-representation $(V,\rho)$ and a positive integer $m$, 
we denote by $(V^{[m]},\rho^{[m]})$ the representation such that $V^{[m]}=V$ and
\begin{equation}\label{equ-rhom}
  \rho^{[m]} \colon  G_k \xrightarrow{\varphi^m} G_k \xrightarrow{\rho} \GL(V).
\end{equation}

Let $H$ be a split connected, reductive $K$-group and choose a Borel pair $(B_H,T)$ defined over $K$. 
The isomorphism classes of irreducible representations of $H$ are in 1-to-1 correspondence with the dominant characters $X^*_+(T)$ of $T$, 
where the positivity is defined with respect to the Borel subgroup opposite to $B_H$ (\S\ref{subsec-notation}). 
This bijection is given by the highest weight of a representation. For a dominant character $\lambda$, we denote by $L_H(\lambda)$ the corresponding irreducible representation of highest weight $\lambda$. 
If there is no confusion, we simply write $L(\lambda)$ for 
$L_H(\lambda)$. 
If $K$ has characteristic zero, $\Rep(H)$ is semisimple. In characteristic $p$ however, this is no longer true in general. For $\lambda\in X_{+}^*(T)$, let $\Lcal_\lambda$ be the line bundle on the flag variety $H/B_H$ attached to $\lambda$  (\cite[\S5.8]{jantzen-representations}). We define an $H$-representation by 
\begin{equation}
    V_H(\lambda) \colonequals H^0(H/B_H,\Lcal_\lambda).
\end{equation}
Equivalently, $V_H(\lambda)$ is the induced representation $\Ind_{B_H}^{H} \lambda$. The representation $V_H(\lambda)$ is of highest weight $\lambda$. If $\cara(K)=0$, $V_H(\lambda)$ is irreducible, hence $V_H(\lambda)=L_H(\lambda)$. In positive characteristic, this is not true in general, but $L_H(\lambda)$ is always the socle of $V_H(\lambda)$. We view elements of $V_H(\lambda)$ as functions $f \colon H\to \AA^1$ satisfying
\begin{equation}\label{Vlambda-function}
f(hb)=\lambda(b^{-1})f(h), \quad h\in H, \  b\in B_H.    
\end{equation}
For dominant characters $\lambda,\lambda'$, we have a natural surjective map
\begin{equation}\label{Vlambda-natural-map}
V_H(\lambda)\otimes V_H(\lambda')\to V_H(\lambda+\lambda') 
\end{equation}
sending $f\otimes f'$ (where $f\in V_H(\lambda)$, $f'\in V_H(\lambda')$ as \eqref{Vlambda-function}) to the function $ff'\in V_H(\lambda+\lambda')$. 
Let $W_H \colonequals W(H,T)$ be the Weyl group 
and $w_{0,H}\in W_H$ the longest element. 
Then there is a unique $B_H$-stable subline of $V_H(\lambda)$, 
which is a weight space for the lowest weight $w_{0,H}\lambda$. 

Let $H$ be a connected, reductive group over $\FF_q$ with 
a Borel pair $(B,T)$ of $H_k$ defined over $\FF_q$. 
We recall the following well-known lemma: 
\begin{lemma}\label{lemma-proj}
Let $\lambda \in X_+^*(T)$. 
Let $v_{\rm high}\in V_{H}(\lambda)$ be a nonzero element in the highest weight line, and let $p_\lambda \colon V_{H}(\lambda)\to k v_{\rm high}$ be the projection onto $k v_{\rm high}$. Then $p_\lambda \in V_{H}(\lambda)^\vee$ is a $B$-eigenvector for the weight $-\lambda$. In other words, $p_\lambda$ is a $B$-equivariant map $V_{H}(\lambda) \to \lambda$.
\end{lemma}
\begin{proof}
This follows from \cite[Lemma 3.3.1]{Imai-Koskivirta-vector-bundles}, as
$B$ is generated by $T$ and $(U_{-\alpha})_{\alpha \in \Phi_+}$. 
\end{proof}

We put 
\[
 X_1^*(T)=\{ \lambda \in X^*(T) \mid \textrm{$0 \leq \langle \lambda, \alpha^\vee \rangle<q$ for all $\alpha \in \Delta$} \} . 
\]
Steinberg's tensor product theorem (\cf \cite{steinberg-reps-alg-gps-nagoya}) is usually stated in the split case. 
We give a statement of the theorem in the general case. 

\begin{theorem}\label{thm-L-decomp}
Let $\lambda =\sum_{i=0}^m q^i \sigma^{-i} (\lambda_i)$ with $\lambda_0, \lambda_1, \ldots, \lambda_m \in X_1^* (T)$. 
Then we have 
\begin{equation}
 L(\lambda) \simeq L(\lambda_0) \otimes L(\lambda_1^{[1]}) \otimes \cdots \otimes L(\lambda_m^{[m]}).  
\end{equation}
\end{theorem}

\begin{proof}
We take a split form $T' \subset B' \subset H'$ over $\FF_q$ of $T \subset B \subset H_k$. Put $\lambda'_i=\sigma^{-i}\lambda_i$. By \cite[II, 3.17]{jantzen-representations}, we have $L(\lambda) \simeq L(\lambda'_0) \otimes L(q \lambda'_1) \otimes \cdots \otimes L(q^m \lambda'_m)$. Since
$L(\lambda_i)^{[i]}=L(q^i \lambda'_i)$ for $0 \leq i \leq m$ the claim follows. 
\end{proof}

\subsection{\texorpdfstring{Vector bundles on the stack of $G$-zips}{}} \label{sec-vector-bundles-gzipz}

\subsubsection{General theory} \label{sec-genth-vb}
For an algebraic stack $\Xcal$, we write $\VB(\Xcal)$ for the category of vector bundles on $\Xcal$. Let $X$ be a $k$-scheme and $H$ an affine $k$-group scheme acting on $X$. 
We have a functor
\[
\Vcal_{H,X} \colon \Rep(H)\to \VB([H\backslash X]).
\]
sending 
$(V,\rho)$ to the vector bundle defined geometrically as $[H\backslash (X\times_k V)]$, where $H$ acts diagonally on $X\times_k V$. 
For $(V,\rho) \in \Rep(H)$, we have a natural identification 
\begin{equation}\label{globquot}
 H^0([H\backslash X],\Vcal_{H,X}(\rho))
 =\left\{f \colon X\to V \mid  f(h \cdot x)=\rho(h) f(x) \  \textrm{for} \  h\in H, \ x\in X\right\}.
\end{equation}

\subsubsection{\texorpdfstring{Vector bundles on $\GZip^\mu$}{}}\label{sec-VBGzip}
Fix a cocharacter datum $(G,\mu)$. Let $\Zcal=(G,P,L,Q,M,\varphi)$ be the attached zip datum. Fix a frame $(B,T)$ as in \S\ref{sec-frames}. By \S\ref{sec-genth-vb}, we obtain a functor $\Vcal_{E,G}\colon \Rep(E)\to \VB(\GZip^\mu)$, that we simply denote by $\Vcal$. For $(V,\rho)\in \Rep(E)$, the global sections of $\Vcal(\rho)$ are
\[
 H^0(\GZip^\mu,\Vcal(\rho))=\left\{f \colon G_k \to V \mid f(\epsilon \cdot g)=\rho(\epsilon) f(g) \  \textrm{for} \  \epsilon\in E, \  g\in G_k \right\}.
\]
Since $G$ admits an open dense $E$-orbit $U_\mu$ (see discussion below Theorem \ref{thm-E-orb-param}), the space $H^0(\GZip^\mu,\Vcal(\rho))$ is finite-dimensional by \cite[Lemma 1.2.1]{Koskivirta-automforms-GZip} (see also \eqref{injection-ev1} below). The first projection $p_1 \colon E\to P$ induces a functor $p_1^* \colon \Rep(P)\to \Rep(E)$. If $(V,\rho)\in \Rep(P)$, we write again $\Vcal(\rho)$ for $\Vcal(p_1^*(\rho))$. In this paper, we will only consider $E$-representations coming from $P$ in this way.

\subsection{\texorpdfstring{Global sections over $\GZip^\mu$}{}}\label{subsec-global-sections}
We review some results of \cite{Imai-Koskivirta-vector-bundles} regarding the space of global sections of $\Vcal(\rho)$ over $\GZip^\mu$ for a $P$-representation $\rho$. We view $H^0(\GZip^\mu,\Vcal(\rho))$ as a subspace of $V$, as follows. By \eqref{globquot}, such a global section is a map $h\colon G\to V$ satisfying $h(agb^{-1})=\rho(a)h(g)$ for all $(a,b)\in E$ and all $g\in G$. Since $1$ lies in the open dense $E$-orbit $U_\mu \subset G$ (see paragraph after Theorem \ref{thm-E-orb-param}), the map $h\mapsto h(1)$ is an injection
\begin{equation}\label{injection-ev1}
\ev_1 \colon H^0(\GZip^\mu,\Vcal(\rho))\to V.
\end{equation}

We describe the image of the map $\ev_1$. First, we recall the space of sections over the open subset $\Ucal_\mu\subset \GZip^\mu$. Recall that $\Ucal_\mu=[E\backslash U_\mu]$ and $1\in U_\mu$ (see \S\ref{subsec-zipstrata}). In particular, we can extend \eqref{injection-ev1} to an injection $\ev_1\colon H^0(\Ucal_\mu,\Vcal(\rho))\to V$. The scheme-theoretical stabilizer of $1$ in $E$ is given by
\begin{equation}\label{Lphi-equ}
L_{\varphi}=E\cap \{(x,x) \mid x\in G_k \}, 
\end{equation}
which is a $0$-dimensional algebraic group (in general non-smooth). The first projection $E\to P$ induces a closed immersion $L_{\varphi}\to P$. Hence we identify $L_\varphi$ with its image and view it as a subgroup of $P$. Denote by $L_0\subset L$ the largest algebraic subgroup defined over $\FF_q$. 
In other words,
\begin{equation}\label{eqL0}
    L_0=\bigcap_{n\geq 0}L^{(q^n)}.
\end{equation}

\begin{lemma}[{\cite[Lemma 3.2.1]{Koskivirta-Wedhorn-Hasse}}]\label{lemLphi} \ 
\begin{assertionlist}
\item \label{lemLphi-item1} One has $L_{\varphi}\subset L$.
\item \label{lemLphi-item2}  The group $L_{\varphi}$ can be written as a semidirect product $L_{\varphi}=L_{\varphi}^\circ\rtimes L_0(\FF_q)$ where $L_{\varphi}^\circ$ is the identity component of $L_{\varphi}$. Furthermore, $L_{\varphi}^\circ$ is a finite unipotent algebraic group.
\item \label{lemLphi-item3}  Assume that $P$ is defined over $\FF_q$. Then $L_0=L$ and $L_{\varphi}=L(\FF_q)$, viewed as a constant algebraic group.
\end{assertionlist}
\end{lemma}

\begin{lemma}[{\cite[Corollary 3.2.3]{Imai-Koskivirta-vector-bundles}}] \label{lem-Umu-sections}
One has an identification
\begin{equation}
 H^0(\Ucal_\mu,\Vcal(\rho))=V^{L_\varphi}.   
\end{equation}
\end{lemma}
More precisely, we mean that the image of $\ev_1\colon H^0(\Ucal_\mu,\Vcal(\rho))\to V$ coincides with $V^{L_\varphi}$. Here, the notation $V^{L_{\varphi}}$ denotes the space of scheme-theoretical invariants, i.e.\ the set of $v\in V$ such that for any $k$-algebra $R$, one has $\rho(x)v=v$ in $V\otimes_k R$ for all $x\in L_{\varphi}(R)$. We now move on to the space of global sections over $\GZip^\mu$. Note that restriction of sections via the open dense inclusion $\Ucal_\mu\subset \GZip^\mu$ induces an injective map $H^0(\GZip^{\mu},\Vcal(\rho))\to H^0(\Ucal_\mu,\Vcal(\rho))=V^{L_\varphi}$. 
We only explain a partial result, for the general result see \cite[Theorem 3.4.1]{Imai-Koskivirta-vector-bundles}. 
The Lang torsor map $\wp \colon T \to T$, $g\mapsto g\varphi(g)^{-1}$ induces an isomorphism:
\begin{equation}\label{equ-Pupstar}
\wp_* \colon X_*(T)_{\RR} \stackrel{\sim}{\longrightarrow} X_*(T)_{\RR};\ \delta \mapsto \wp \circ \delta  = \delta -q\sigma(\delta). 
\end{equation}
Set $\delta_{\alpha}\colonequals \wp_*^{-1}(\alpha^\vee)$. Define a subspace $V_{\geq 0}^{\Delta^P}\subset V$ as follows:

\begin{equation}\label{equ-VDeltaP}
V_{\geq 0}^{\Delta^P} = \bigoplus_{\substack{\langle \nu,\delta_{\alpha} \rangle \geq 0, \ 
 \forall \alpha\in \Delta^P}} V_\nu.
\end{equation}
If $T$ is split over $\FF_q$, then 
$\delta_{\alpha}=-\alpha^\vee /(q-1)$, 
and $V_{\geq 0}^{\Delta^P}$ is the direct sum of the weight spaces $V_\nu$ for those $\nu\in X^*(T)$ satisfying $\langle \nu,\alpha^\vee \rangle \leq 0$ for all $\alpha \in \Delta^P$. 

\begin{proposition}\label{cor-Fq-Levi}
Assume that $P$ is defined over $\FF_q$ and furthermore that $(V,\rho)\in \Rep(P)$ is trivial on the unipotent radical $R_{\mathrm{u}}(P)$. Then one has an equality
\begin{equation}
H^0(\GZip^\mu,\Vcal(\rho))=V^{L(\FF_q)}\cap V_{\geq 0}^{\Delta^P} .
\end{equation}
\end{proposition}

\subsection{\texorpdfstring{Category of $L$-vector bundles on $\GZip^\mu$}{}}\label{sec-Lreps}

Let $\theta^P_L \colon P\to L$ denote the natural projection modulo the unipotent radical $R_{\mathrm{u}}(P)$ as  \S\ref{subsec-zipdatum}. It induces a fully faithful functor 
\begin{equation}
(\theta^P_L)^* \colon \Rep(L)\to \Rep(P)  
\end{equation}
by composition, 
and its image is the full subcategory of $\Rep(P)$ of $P$-representations which are trivial on $R_{\mathrm{u}}(P)$. 
By this functor, we view $\Rep(L)$ as a full subcategory of $\Rep(P)$. If $(V,\rho)\in \Rep(L)$, we write again $\Vcal(\rho) \colonequals \Vcal((\theta^P_L)^*\rho)$. 
For $\lambda\in X^*(T)$, we write $B_L \colonequals B\cap L$ and define an $L$-representation $(V_I(\lambda),\rho_{I,\lambda})$ by
\begin{equation}\label{equ-VlambdaL}
V_I(\lambda)=\Ind_{B_L}^L (\lambda), \quad \rho_{I,\lambda}\colon L\to \GL(V_I(\lambda)). 
\end{equation}
We note that if $\lambda\in X^*(T)$ is not $I$-dominant, then we have  $V_I(\lambda)=0$. 
Let $\Vcal_I(\lambda)$ denote the vector bundle on $\GZip^\mu$ attached to $V_I(\lambda)$. We call it an automorphic vector bundle on $\GZip^\mu$ associated to $\lambda$. This terminology stems from its relation to Shimura varieties (\cf \S\ref{subsec-Shimura}). For $\lambda \in X^*(L)$, viewing it as an element of $X^*(T)$ by restriction to $T$, the vector bundle $\Vcal_I(\lambda)$ is a line bundle. 

Denote by $\VB_L(\GZip^\mu)$ the essential image of the functor $\Vcal\colon \Rep(L)\to \VB(\GZip^\mu)$. 
An object of $\VB_L(\GZip^\mu)$ is called an $L$-vector bundle. 
Assume that $P$ is defined over $\FF_q$, hence $L_{\varphi}=L(\FF_q)$. We describe the category $\VB_L(\GZip^\mu)$. Define the category $L_{\varphi}\mathchar`-\mathrm{MF}_{\Delta^P}$ of $\Delta^P$-filtered $L_{\varphi}$-modules over $k$ as follows. Its objects are pairs $((V,\tau),\Fcal)$ where $\tau \colon L_{\varphi} \to \GL_k(V)$ is a finite-dimensional representation of $L_{\varphi}$ and $\Fcal \colonequals \{V_{\geq \bullet}^\alpha\}_{\alpha\in \Delta^P}$ is a set of filtrations on $V$, one for each $\alpha \in \Delta^P$. Here, $V_{\geq \bullet}^\alpha$ denotes a descending filtration $(V_{\geq r}^\alpha)_{r\in \RR}$. Morphisms between two $\Delta^P$-filtered $L_{\varphi}$-modules $((V,\tau),\Fcal)$ and $((V',\tau'),\Fcal')$ over $k$ are $k$-linear, $L_{\varphi}$-equivariant maps $f \colon V\to V'$ which map $V_{\geq r}^\alpha$ to $V'^{\alpha}_{\geq r}$ for all $r\in \RR$ and all $\alpha\in \Delta^P$. 

If $(V,\rho)$ is an $L$-representation, we can attach to it a $\Delta^P$-filtered $L_{\varphi}$-module $F_{\mathrm{MF}}(V,\rho)$ as follows. First, define $\tau$ as the restriction of $\rho$ to $L_{\varphi}$. Then, for $\alpha\in \Delta^P$, define the filtration $V_{\geq \bullet}^\alpha$ as follows. Let $V \colonequals \bigoplus_{\nu}V_\nu$ be the $T$-weight decomposition of $V$. For $\alpha\in \Delta^P$ and $r\in \RR$, let $V^{\alpha}_{\geq r}$ be the direct sum of $V_\nu$ for all $\nu$ satisfying $\langle \nu,\delta_{\alpha} \rangle \geq r$. This construction gives rise to a functor
 \begin{equation}\label{equ-functor-modfil}
 F_{\mathrm{MF}} \colon \Rep(L)\to L_{\varphi}\mathchar`-\mathrm{MF}_{\Delta^P}, 
 \end{equation}
and we say that a $\Delta^P$-filtered $L_{\varphi}$-module is \emph{admissible} if it is in the essential image of $F_{\mathrm{MF}}$. 
Let $L_{\varphi}\mathchar`-\mathrm{MF}_{\Delta^P}^{\mathrm{adm}}$ 
denote the full subcategory of 
$L_{\varphi}\mathchar`-\mathrm{MF}_{\Delta^P}$ 
consisting of the admissible $\Delta^P$-filtered $L_{\varphi}$-modules. 

\begin{theorem}[{\cite[Theorem 5.2.3]{Imai-Koskivirta-vector-bundles}}]\label{fullfaith}
Assume that $P$ is defined over $\FF_q$. The functor $\Vcal\colon \Rep(L)\to \VB_L(\GZip^\mu)$ factors through the functor $F_{\mathrm{MF}} \colon \Rep(L)\to L_{\varphi}\mathchar`-\mathrm{MF}_{\Delta^P}^{\mathrm{adm}}$ and induces an equivalence of categories
\begin{equation}
 L_{\varphi}\mathchar`-\mathrm{MF}_{\Delta^P}^{\mathrm{adm}} \longrightarrow \VB_L(\GZip^\mu).
\end{equation}
\end{theorem}

\subsection{\texorpdfstring{Shimura varieties and $G$-zips}{}}\label{subsec-Shimura}
We explain the connection between the stack of $G$-zips and Shimura varieties. Let $\gx$ be a Shimura datum of Hodge-type (\cf \cite[2.1.1]{Deligne-Shimura-varieties}). 
In particular, $\mathbf{G}$ is a connected, reductive group over $\QQ$, and $\mathbf{X}$ is a $\mathbf{G}(\overline{\QQ})$-conjugacy class of cocharacters $\{\mu\}$ of $\mathbf{G}_{\overline{\QQ}}$. We write $\mathbf{E}=\egx$ for the reflex field of $\gx$ and $\Ocal_\mathbf{E}$ for its ring of integers. 
For an open compact subgroup $K \subset \gofaf$, we write $\shgx_{K}$ for Deligne's canonical model at level $K$ over $\mathbf{E}$ (see \cite{Deligne-Shimura-varieties}). For $K\subset \mathbf{G}(\AA_f)$ small enough, 
the canonical model $\shgx_K$ is a smooth, quasi-projective scheme over $\mathbf{E}$. 
We fix a prime number $p$ of good reduction. In particular, $\mathbf{G}_{\QQ_p}$ is unramified, so there exists a reductive $\ZZ_p$-model $\Gcal$, such that $G\colonequals \Gcal\otimes_{\ZZ_p}\mathbb{F}_p$ is connected. For any place $v$ above $p$ in $\mathbf{E}$, Kisin (\cite{Kisin-Hodge-Type-Shimura}) and Vasiu (\cite{Vasiu-Preabelian-integral-canonical-models})  constructed a smooth canonical model $\Sscr_K$ over $\Ocal_{\mathbf{E}_v}$-schemes. Write $S_K\colonequals \Sscr_K\otimes_{\Ocal_{\mathbf{E}_v}} \overline{\FF}_p$. 

For $\mu \in \{\mu\}$, let $\mathbf{P}=\mathbf{P}_-(\mu)$ be the parabolic of $\mathbf{G}_\CC$ defined as in \S\ref{subsec-cochar}. As explained in \cite[\S 2.5]{Imai-Koskivirta-vector-bundles}, we can find $\mu\in \{\mu\}$ which extends to a cocharacter of $\Gcal_{\Ocal_{\mathbf{E}_v}}$. Write again $\mu$ for its special fiber. Then $(G,\mu)$ is a cocharacter datum, and yields a zip datum $(G,P,L,Q,M,\varphi)$ (we always take $q=p$ in the context of Shimura varieties). Zhang (\cite[4.1]{Zhang-EO-Hodge}) constructed a smooth morphism
\begin{equation}\label{zeta-Shimura}
\zeta \colon S_K\to \GZip^\mu, 
\end{equation}
which is surjective by \cite[Corollary 3.5.3(1)]{Shen-Yu-Zhang-EKOR}. Let $\Pscr$ be the unique parabolic of $\Gcal_{\Ocal_{\mathbf{E}_v}}$ which extends $\mathbf{P}$. Then, we have a commutative diagram 
$$\xymatrix{
\Rep_{\overline{\ZZ}_p}(\mathscr{P}) \ar[r]^{\Vcal} \ar[d] & \VB(\Sscr_{K}) \ar[d] \\
\Rep_{\overline{\FF}_p}(P) \ar[r]^{\Vcal} & \VB(S_K).
}$$
The vector bundles of this form on $\Sscr_K$ and $S_K$ are called \emph{automorphic vector bundles} in \cite[III. Remark 2.3]{Milne-ann-arbor}. In particular, let $\lambda\in X^*(\mathbf{T})$ be an $\mathbf{L}$-dominant character and let $\mathbf{V}_{\mathbf{L}}(\lambda)=H^0(\mathbf{P}/\mathbf{B},\Lcal_\lambda)$ denote the unique irreducible representation of $\mathbf{P}$ over $\overline{\QQ}_p$ of highest weight $\lambda$. It admits a natural model over $\overline{\ZZ}_p$, namely $\mathbf{V}_{\mathbf{L}}(\lambda)_{\overline{\ZZ}_p} \colonequals H^0(\mathscr{P}/\mathscr{B},\Lcal_\lambda)$ where $\Lcal_\lambda$ is the line bundle attached to 
$\lambda$ viewed as a character of $\mathscr{T}$. 
Its reduction modulo $p$ is the $P$-representation $V_I (\lambda)=H^0(P/B,\Lcal_\lambda)$ over $k=\overline{\FF}_p$ as in \eqref{equ-VlambdaL}.

\section{\texorpdfstring{Vector bundles on the stack of $G$-zip flags}{}}

\subsection{\texorpdfstring{The stack of $G$-zip flags}{}} \label{subsec-zipflag}
Let $(G,\mu)$ be a cocharacter datum with attached zip datum $\Zcal=(G,P,L,Q,M,\varphi)$ (see \S\ref{subsec-cochar}). Fix a frame $(B,T,z)$ with $z=\sigma(w_{0,I})w_0=w_0w_{0,J}$ (Lemma \ref{lem-framemu}). The stack of zip flags (\cite[Definition 2.1.1]{Goldring-Koskivirta-Strata-Hasse}) is defined as
\begin{equation}\label{eq-Gzipflag-PmodB}
\GF^\mu=[E\backslash (G_k \times P/B)]    
\end{equation}
where the group $E$ acts on the variety $G_k \times (P/B)$ by the rule $(a,b)\cdot (g,hB) \colonequals (agb^{-1},ahB)$ for all $(a,b)\in E$ and all $(g,hB)\in G_k \times P/B$. The first projection $G_k \times P/B \to G_k$ is $E$-equivariant, and hence yields a natural morphism of stacks
 \begin{equation}\label{projmap-flag}
     \pi  \colon  \GF^\mu \to \GZip^\mu.
 \end{equation}
Similarly to the stack $\GZip^\mu$, there is an interpretation of this stack in terms of torsors, given in \cite[Definition 2.1.1]{Goldring-Koskivirta-Strata-Hasse}. For any $k$-scheme $S$, the $S$-points of $\GF^\mu$ are pairs $(\underline{I},J)$ where $\underline{I}=(I,I_P,I_Q,\iota)$ is a $G$-zip over $S$ (see the end of \S\ref{subsec-zipdatum}) and $J\subset I_P$ is a $B$-torsor. The map \eqref{projmap-flag} is given by forgetting the $B$-torsor $J$.

Set $E' \colonequals E\cap (B\times G_k)$. Then the injective map $G_k \to G_k \times P/B;\ g\mapsto (g,B)$ yields an isomorphism of stacks $[E' \backslash G_k]\simeq \GF^\mu$ (see \cite[(2.1.5)]{Goldring-Koskivirta-Strata-Hasse}). One of the main features of the stack $\GF^\mu$ is the existence of a stratification (indexed by $W$). First define the Schubert stack as the quotient stack
\begin{equation}\label{equ-def-Sbt}
\Sbt \colonequals [B\backslash G_k /B].
\end{equation}
This stack is finite and smooth. Its topological space is isomorphic to $W$, endowed with the topology induced by the Bruhat order on $W$. This follows easily from the Bruhat decomposition of $G$. For $w\in W$, put
\begin{equation}\label{Sbtw}
    \Sbt_w\colonequals [B\backslash BwB /B].
\end{equation}
It is a locally closed substack of $\Sbt$. $\Sbt_{w_0}$ is the unique open stratum of $\Sbt$. One has the inclusion $E'\subset B\times {}^z B$. In particular, there is a natural projection map $[E'\backslash G_k]\to [B\backslash G_k/{}^z\! B]$. To obtain a map to $\Sbt$, we compose with the isomorphism of stacks $[B\backslash G_k/{}^z B]\to [B\backslash G_k/B]$ induced by $G_k\to G_k;\ g\mapsto gz$. In the end, we obtain a smooth, surjective morphism of stacks
 \begin{equation}\label{eq-GF-to-Sbt}
     \psi  \colon  \GF^\mu \to \Sbt.
 \end{equation}
The \emph{flag strata} of the stack $\GF^\mu$ are defined as the fibers of $\psi$. They are locally closed substacks (endowed with the reduced structure). This gives a stratification of $\GF^\mu$ indexed by $W$. For $w\in W$, put
\begin{equation}
C_w \colonequals B(wz^{-1}){}^zB=BwBz^{-1},
\end{equation} 
which is the $B\times {}^zB$-orbit of $wz^{-1}$. The set $C_w$ is locally closed in $G_k$, and one has the dimension formula $\dim(C_w)=\ell(w)+\dim(B)$. The Zariski closure $\overline{C}_w$ is normal (\cite[Theorem 3]{Ramanan-Ramanathan-projective-normality}) and coincides with the union $\bigcup_{w'\leq w}C_{w'}$, where $\leq$ is the Bruhat order of $W$. Via the isomorphism $\GF^\mu\simeq [E'\backslash G_k]$, the flag strata of $\GF^\mu$ are given by
\begin{equation}\label{zipflag-Cw}
\Ccal_w \colonequals [E'\backslash C_w], \quad w\in W.    
\end{equation}
The set $C_{w_0}\subset G_k$ is the unique open $B\times {}^zB$-orbit in $G_k$ and similarly the flag stratum $\Ccal_{w_0}$ is open in $\GF^\mu$. The $B\times {}^zB$-orbits of codimension 1 are given by $C_{s_\alpha w_0}$ for $\alpha\in \Delta$.

\subsection{\texorpdfstring{Vector bundles on $\GF^\mu$}{}}
Let $\rho \colon B\to \GL(V)$ be an algebraic representation. Since $\GF^\mu$ classifies pairs $(\underline{I},J)$ where $J$ is a $B$-torsor, applying $\rho$ to $J$ yields a vector bundle $\Vcal_{\flag}(\rho)$ on $\GF^\mu$. We can also construct $\Vcal_{\flag}(\rho)$ as follows. View $\rho$ as a representation of $E'$ via the first projection $E'\to B$. Then $\rho$ induces a vector bundle on $[E'\backslash G_k]\simeq \GF^\mu$ by \S\ref{sec-genth-vb}. In this description, we see that the functor $\Vcal_{\flag}$ extends to a functor $\Rep(E')\to \VB(\GF^\mu)$. This fact will briefly be used in equation \eqref{pull-back-psi}. However, we will mainly consider vector bundles on $\GF^\mu$ coming from $\Rep(B)$. 

Note that the rank of $\Vcal_{\flag}(\rho)$ is the dimension of $\rho$. In particular, if $\lambda\in X^*(B)=X^*(T)$, then $\Vcal_{\flag}(\lambda)$ is a line bundle. Let $(V,\rho)\in \Rep(P)$ and let $\Vcal(\rho)$ be the attached vector bundle on $\GZip^\mu$. One has
\begin{equation}\label{pullbackV}
\pi^*(\Vcal(\rho))=\Vcal_{\flag}(\rho|_B).
\end{equation}
For $(V,\rho)\in \Rep(B)$, write $\Vcal_{P/B}(\rho)$ for the vector bundle on $P/B$ attached to $\rho$ by \S\ref{sec-genth-vb}. Then $H^0(P/B,\Vcal_{P/B}(\rho))=\Ind_B^P(\rho)$ is a $P$-representation, where $\Ind$ denotes induction. Concretely, we have:
\begin{equation}\label{eq-Indrho}
\Ind_B^P(\rho)=\{f\colon P\to V \mid f(xb) = \rho(b^{-1})f(x), \ \forall b\in B,\ \forall x\in P \}.
\end{equation}
Furthermore, for $y\in P$ and $f\in \Ind_B^P(\rho)$, the element $y\cdot f$ is the function $x\mapsto f(y^{-1} x)$.

\begin{proposition}\label{prop-push-flag-zip}
For $(V,\rho)\in \Rep(B)$, we have the identification
\begin{equation}
    \pi_{*}(\Vcal_{\flag}(\rho))=\Vcal(\Ind_B^P(\rho)).
\end{equation}
In particular $\pi_{*}(\Vcal_{\flag}(\rho))$ is a vector bundle on $\GZip^\mu$.
\end{proposition}

\begin{proof}
Consider the natural map $\Ind_B^P(\rho)\to \rho$ defined by $f\mapsto f(1)$ (where $f\in \Ind_B^P(\rho)$ is viewed as a regular function $f \colon P\to V$). This defines a map of $B$-representations. Hence, it induces a morphism $\pi^*(\Vcal(\Ind_B^P(\rho)))=\Vcal_{\flag}(\Ind_B^P(\rho)|_B)\to \Vcal_{\flag}(\rho)$ of vector bundles on $\GF^\mu$. By adjunction, we obtain a morphism of vector bundles $\Vcal(\Ind_B^P(\rho))\to \pi_{*}(\Vcal_{\flag}(\rho))$ on $\GZip^\mu$. Since this map induces an isomorphism on the stalks, it is itself an isomorphism.
\end{proof}

As a special case of Proposition \ref{prop-push-flag-zip}, consider the case when $\rho$ is a character $\lambda\in X^*(T)$. Then $\Vcal_{\flag}(\lambda)$ is a line bundle and one has the formula $\pi_*(\Vcal_{\flag}(\lambda))=\Vcal_I(\lambda)$ where the vector bundle $\Vcal_I(\lambda)$ was defined in \S\ref{sec-Lreps}. This recovers the formula (1.3.3) of \cite{Koskivirta-automforms-GZip} (where $\Vcal_{\flag}(\lambda)$ was denoted by $\Lscr(\lambda)$ and $\Vcal_I(\lambda)$ by $\Vcal(\lambda)$). In particular, we have 
\begin{equation}\label{ident-H0-GF}
    H^0(\GZip^\mu,\Vcal_I(\lambda))=H^0(\GF^\mu,\Vcal_{\flag}(\lambda)).
\end{equation}
If $f \colon G_k \to k$ is a section of the right hand side of \eqref{ident-H0-GF}, 
then the corresponding function 
$f_I \colon G_k \to V_I (\lambda)$ 
on the left hand side of 
\eqref{ident-H0-GF} is given by 
\begin{equation}\label{eq:H0expid}
 (f_I (g))(x)=f ((x^{-1},\varphi(x)^{-1}) \cdot g)= 
 f(x^{-1} g \varphi (x)) 
\end{equation}
for $g \in G_k$ and $x \in L$ 
by the construction of the identification. 
Note also that the line bundles $\Vcal_{\flag}(\lambda)$ satisfy the identity
\begin{equation}\label{Vflag-additivity}
\Vcal_{\flag}(\lambda+\lambda')=\Vcal_{\flag}(\lambda)\otimes \Vcal_{\flag}(\lambda'), \quad \forall \lambda, \lambda'\in X^*(T).
\end{equation}

\section{\texorpdfstring{Vector bundles on $\Sbt$}{}}

\subsection{Definition}\label{sec-VBSbt}
Any algebraic representation $\rho \colon B\times B\to \GL(V)$ is given by two commuting $B$-representations $\rho_1,\rho_2 \colon B\to \GL(V)$. We write $\rho=(\rho_1,\rho_2)$. By \S\ref{sec-genth-vb}, it induces a vector bundle $\Vcal_{\Sbt}(\rho_1,\rho_2)$ on $\Sbt = [B \backslash G_k/B]$. 
Let $m_z \colon B\times {}^zB\to B\times B;\ (x,y)\mapsto (x,z^{-1}yz)$. Recall that we have a map $\psi\colon \GF^\mu \to \Sbt$, as in \eqref{eq-GF-to-Sbt}. We have
\begin{equation}\label{pull-back-psi}
  \psi^*\Vcal_{\Sbt}(\rho_1,\rho_2) = \Vcal_{\flag}((m_z^*(\rho_1,\rho_2))|_{E'})
\end{equation}
where $m_z^*(\rho_1,\rho_2)$ is the pullback via $m_z$ of $(\rho_1,\rho_2)$ to a $B\times {}^z B$-representation. Since we are mainly interested in vector bundles on $\GF^\mu$ attached to $B$-representations, we define the following condition:
\begin{condition}\label{condition-trivial}
The representation $\rho_2$ is trivial on $z^{-1}R_{\mathrm{u}}(Q)z$.
\end{condition}
Note that $z^{-1}R_{\mathrm{u}}(Q)z\subset B$ by \eqref{eqBorel}. If Condition \ref{condition-trivial} is satisfied, the $E'$-representation $m_z^*(\rho_1,\rho_2)|_{E'}$ arises from a $B$-representation via the first projection $E'\to B$. Specifically, we have in this case
\begin{equation}\label{Rrhorep}
\psi^*\Vcal_{\Sbt}(\rho_1,\rho_2)=\Vcal_{\flag}(R(\rho_1,\rho_2)),    
\end{equation}
where $R(\rho_1,\rho_2) \colon B\to \GL(V)$ is defined by
\begin{equation}
  R(\rho_1,\rho_2)(b)=\rho_1(b)\rho_2(z^{-1}\varphi(\theta^P_{L}(b))z), \quad \forall b\in B.
\end{equation}
Let $\Acal$ be the subcategory of $\Rep(B)$ of representations which are trivial on $z^{-1}R_{\mathrm{u}}(Q)z$. Put $B_M\colonequals B\cap M$. For $\rho\in \Rep(B)$, denote by $\rho[z^{-1}]$ the $B_M$-representation $b\mapsto \rho(z^{-1}bz)$ (note that $z^{-1} B_M z\subset B$ by \eqref{eqBorel}). 
This gives an equivalence
 \begin{equation}\label{eq-ztwist}
 [z^{-1}]\colon \Acal\simeq \Rep(B_M);\ \rho \mapsto \rho[z^{-1}]. 
 \end{equation}
Similarly, for $\rho\in \Rep(B_M)$, write $\rho[z]$ for the $B$-representation in $\Acal$ given by $b\mapsto \rho(zbz^{-1})$ for $b \in z^{-1}B_M z$. 
We view both sides as full subcategories of $\Rep(B)$. We will always assume that $\rho_2\in \Acal$ when considering $\Vcal_{\Sbt}(\rho_1,\rho_2)$.

\subsection{\texorpdfstring{Global sections of vector bundles on $\Sbt$}{}}\label{sec-glob-sec-Sbt}

Let $(V,\rho)\in \Rep(B\times B)$ with $\rho=(\rho_1,\rho_2)$, as in \S\ref{sec-VBSbt}, and let $\Vcal_{\Sbt}(\rho_1,\rho_2)$ be the attached vector bundle on $\Sbt$. We determine the space of global sections of the vector bundle $\Vcal_{\Sbt}(\rho_1,\rho_2)$ on $\Sbt$. As a first step, we study section over the open stratum $\Sbt_{w_0}\subset \Sbt$ (see \eqref{Sbtw}). Consider the action of $B\times B$ on $G_k$ given by $(a,b)\cdot g=agb^{-1}$. Define $S_{w_0}\subset B\times B$ as the stabilizer of $w_0$ for this action.

\begin{lemma} \ \label{lemma-sections-Sbtw0}
\begin{assertionlist}
\item \label{Sbtw0-item1} One has $S_{w_0}=\{(t,w_0 t w_0) \mid t\in T\}$.
\item \label{Sbtw0-item2} The open stratum $\Sbt_{w_0}$ is naturally isomorphic to $[1/S_{w_0}]$. 
\item \label{Sbtw0-item3} We have an identification $H^0(\Sbt_{w_0},\Vcal_{\Sbt}(\rho_1,\rho_2))=V^{S_{w_0}}$.
\end{assertionlist}
\end{lemma}

\begin{proof}
The first assertion is clear. For the second assertion, note that $B\times B\to G_k;\ (a,b)\mapsto aw_0 b^{-1}$ induces an isomorphism $(B\times B)/S_{w_0}\to B w_0 B$. Thus $\Sbt_{w_0}=[B\times B\backslash B w_0 B]\simeq [B\times B\backslash (B\times B)/S_{w_0}]\simeq [1/S_{w_0}]$. The last assertion is an immediate consequence.
\end{proof}

For $f\in V^{S_{w_0}}$, 
we may view $f$ as a section of $\Vcal_{\Sbt}(\rho_1,\rho_2)$ over $\Sbt_{w_0}$ (Lemma \ref{lemma-sections-Sbtw0}(3)). Consider the associated regular map $\widetilde{f}\colon Bw_0B\to V$. We will check under what condition $\widetilde{f}$ extends to a regular map $G_k\to V$. We briefly recall notations used in \cite[\S 3.1]{Imai-Koskivirta-vector-bundles}. For all $\alpha\in \Phi$, choose a realization $(u_{\alpha})_{\alpha \in \Phi}$ and let $\phi_{\alpha}\colon \SL_2\to G$ be the map attached to $\alpha$ (see \S \ref{subsec-notation}). Put $A(t)= 
 \begin{pmatrix}
 t & 1 \\ -1 & 0 
 \end{pmatrix} \in \SL_2$. Note that $\phi_{\alpha} (A(0))=s_\alpha$ and  $\phi_{\alpha}(\diag(t,t^{-1}))=\alpha^\vee(t)$. As in \cite[(3.1.3)]{Imai-Koskivirta-vector-bundles}, one has a decomposition
\begin{equation} \label{At}
A(t)=\left(\begin{matrix}
1&0\\-t^{-1}&1
\end{matrix} \right) \left(\begin{matrix}
t&1\\0&t^{-1}
\end{matrix} \right)=\left(\begin{matrix}
1&0\\-t^{-1}&1
\end{matrix} \right) \left(\begin{matrix}
t&0\\0&t^{-1}
\end{matrix} \right)\left(\begin{matrix}
1&t^{-1}\\0&1
\end{matrix} \right).
\end{equation}
Let $\alpha\in \Delta$ be a simple root. Set $Y=B\times B\times \AA^1$ and let $\psi_\alpha\colon Y\to G$ be the map
\begin{equation}
\psi_\alpha \colon ((b,b'),t)\mapsto b\phi_\alpha(A(t))w_0 b'^{-1}.
\end{equation}
We have $\phi_\alpha(A(t))\in B B^+ = Bw_0Bw_0$ for $t\neq 0$. We deduce that $\psi_\alpha(b,b',t)\in Bw_0B$ if $t\neq 0$ and $\psi_\alpha(b,b',0)\in Bs_\alpha w_0B$. Write $Y_0\subset Y$ for the open set where $t\neq 0$. Using a similar argument as \cite[Corollary 3.1.5]{Imai-Koskivirta-vector-bundles}, we deduce that $\widetilde{f}$ extends to $G$ if and only if $\widetilde{f}\circ \psi_\alpha \colon Y_0 \to V$ extends to a regular map $Y\to V$ for all $\alpha \in \Delta$. We have
\begin{align}
\widetilde{f}\circ \psi_\alpha (b,b',t) & = \rho_1\left( b \phi_\alpha \left(\begin{matrix}
1&0\\-t^{-1}&1
\end{matrix} \right) \alpha^\vee(t) \right) \rho_2 \left( b' \phi_{-w_0\alpha} \left(\begin{matrix}
1&0\\-t^{-1}&1
\end{matrix} \right) \right) \cdot f \\
&=\rho_1(b)\rho_2(b') \rho_1 \left( u_{-\alpha}(-t^{-1}) \alpha^\vee(t) \right) \rho_2 \left(  u_{w_0\alpha}(-t^{-1})\right) \cdot f . 
\end{align}
Write $F_\alpha(t)=\rho_1 \left( u_{-\alpha}(-t^{-1}) \alpha^\vee(t) \right) \rho_2 \left(  u_{w_0\alpha}(-t^{-1})\right) \cdot f$. It is an element of $V\otimes_k k[t,\frac{1}{t}]$. It follows that $\widetilde{f}\circ \psi_\alpha$ extends to $Y$ if and only if $F_\alpha(t)$ lies in $V\otimes_k k[t]$. Write $f=\sum_{\chi_1}f_{\chi_1}$ for the weight decomposition with respect to the action of $T\times \{1\}$ on $V$. We deduce
\begin{align}
    F_\alpha(t) &= \sum_{\chi_1\in X^*(T)} t^{\langle \chi_1,\alpha^\vee \rangle} \rho_1 \left( u_{-\alpha}(-t^{-1}) \right) \rho_2 \left(  u_{w_0\alpha}(-t^{-1})\right) \cdot f_{\chi_1} \\
    & = \sum_{\chi_1} \sum_{j_1,j_2 \geq 0} (-1)^{j_1+j_2} \  t^{\langle \chi_1,\alpha^\vee \rangle-(j_1+j_2)} E_{-\alpha,0}^{(j_1)}\circ E_{0, w_0\alpha}^{(j_2)} f_{\chi_1} . 
\end{align}
Here $E_{\bullet,\bullet}$ are the maps attached to the $B\times B$-representation $V$ (see \S\ref{subsec-notation}). The $\chi_1$-component $F_{\alpha,\chi_1}(t)$ of $F_\alpha(t)$ is
\begin{align}
 F_{\alpha,\chi_1}(t) &=   \sum_{j_1,j_2 \geq 0} (-1)^{j_1+j_2} \  t^{\langle \chi_1,\alpha^\vee \rangle+ j_1 - j_2} E_{-\alpha,0}^{(j_1)}\circ E_{0, w_0\alpha}^{(j_2)} f_{\chi_1+j_1\alpha}  \\
 & = \sum_{d\in \ZZ} t^{d+\langle \chi_1,\alpha^\vee \rangle}(-1)^d \sum_{j_1 \geq d}  E_{-\alpha,0}^{(j_1)} E_{0,w_0\alpha}^{(j_1-d)}f_{\chi_1+j_1\alpha} . 
\end{align}
Hence $F_\alpha(t)\in V\otimes k[t]$ if and only if for all $d\in \ZZ$ and all $\chi_1$ such that $d+\langle \chi_1,\alpha^\vee \rangle <0$, one has $\sum_{j_1 \geq d} E_{-\alpha,0}^{(j_1)} E_{0,w_0\alpha}^{(j_1-d)}f_{\chi_1+j_1\alpha}=0$. Therefore, we showed the following proposition.

\begin{proposition}\label{prop-Sbt-glob}
The space $ H^0(\Sbt, \Vcal_{\Sbt}(\rho_1,\rho_2))$ identifies with the subspace of $f\in V^{S_{w_0}}$ satisfying: For all $d\in \ZZ$ and all $\chi_1\in X^*(T)$ such that $d+\langle \chi_1,\alpha^\vee \rangle <0$, one has
\begin{equation}
    \sum_{j_1 \geq d} E_{-\alpha,0}^{(j_1)} E_{0,w_0\alpha}^{(j_1-d)}f_{\chi_1+j_1\alpha}=0.
\end{equation}
\end{proposition}
There is a similar result for maps of vector bundles on $\Sbt$, which we now explain. Let $(\rho_1,\rho_2)$ and $(\rho_1',\rho_2')$ be $B\times B$-representations with underlying vector spaces $V,V'$ respectively. Then a morphism $\Vcal_{\Sbt}(\rho_1,\rho_2)\to \Vcal_{\Sbt}(\rho'_1,\rho'_2)$ is equivalent to a $k$-linear map $f \colon V\to V'$ satisfying the following conditions.
\begin{assertionlist}
   \item $f$ is $S_{w_0}$-equivariant.
    \item For all $\alpha \in \Delta$, for all $d\in \ZZ$ and all character $\chi_1$ such that $d+\langle \chi_1, \alpha^\vee \rangle < 0$, the map
\begin{equation}
    \sum_{j_1 \geq d} \sum_{a=0}^{j_1}\sum_{b=0}^{j_1-d} (-1)^{a+b} E'^{(a)}_{-\alpha,0} E'^{(b)} _{0,w_0\alpha}f E_{0,w_0\alpha}^{(j_1-d-b)} E_{-\alpha,0}^{(j_1-a)}
\end{equation}
maps to zero under the projection
\begin{equation}\label{proj-map}
    \Hom(V,V) \simeq \bigoplus_{\nu_1,\nu'_1} \Hom(V_{\nu_1},V_{\nu'_1})\to \bigoplus_{\nu_1}\Hom(V_{\nu_1},V_{\nu_1+\chi_1}).
\end{equation}
\end{assertionlist}

\section{Schubert sections, partial Hasse invariants}

\subsection{Definition}\label{subsec-Schubert-sections}

Let $(\chi,\nu)\in X^*(T)\times X^*(T)$. Then $H^0(\Sbt_{w_0}, \Vcal_{\Sbt}(\chi,\nu))\neq 0$ if and only if $\nu=-w_0\chi$. This follows from Lemma \ref{lemma-sections-Sbtw0}\eqref{Sbtw0-item3} (see also \cite[Theorem 2.2.1(a)]{Goldring-Koskivirta-Strata-Hasse}). If this condition is satisfied, the space $H^0(\Sbt_{w_0}, \Vcal_{\Sbt}(\chi,\nu))$ is one-dimensional, and the divisor of any nonzero element
\begin{equation}\label{hchi}
h_\chi\in H^0(\Sbt_{w_0}, \Vcal_{\Sbt}(\chi,-w_0\chi)) 
\end{equation}
is given by Chevalley's formula (\cite[Theorem 2.2.1(c)]{Goldring-Koskivirta-Strata-Hasse}):
\begin{equation}\label{divflam}
\div(h_\chi)=\sum_{\alpha\in \Delta}\langle \chi, \alpha^\vee \rangle \overline{B w_0 s_\alpha   B}. 
\end{equation}
In particular, $h_\chi$ extends to $\Sbt$ if and only if $\chi\in X^*_+(T)$ (this also follows from Proposition \ref{prop-Sbt-glob}). For all $\chi,\nu \in X^*(T)$, one has $\psi^*\Vcal_{\Sbt}(\chi,\nu)=\Vcal_{\flag}(R(\chi,\nu))$ by \eqref{Rrhorep}, where $R(\chi,\nu)$ is the character $t\mapsto \chi(t)\nu(z^{-1}\varphi(t)z)$. Hence, we obtain
\begin{equation}\label{pullbackformula}
\psi^*\Vcal_{\Sbt}(\chi,\nu)=\Vcal_{\flag}(\chi+(z\nu)\circ \varphi)=\Vcal_{\flag}(\chi+q\sigma^{-1}(z\nu)).
\end{equation}
In particular, for the pair $(\chi,-w_0\chi)$, we obtain
\begin{equation}
\psi^*\Vcal_{\Sbt}(\chi,-w_0\chi)=\Vcal_{\flag}(\chi-qw_{0,I}(\sigma^{-1}\chi))    
\end{equation}
where we used that $z=\sigma(w_{0,I})w_0$.

Recall that we have a locally closed stratification $(\Ccal_w)_{w\in W}$ on $\GF^\mu$ defined as fibers of the map $\psi \colon \GF^\mu \to \Sbt$. In particular, $\Ccal_{w_0}$ is the unique open flag stratum (see \S \ref{subsec-zipflag}). By the previous discussion, it follows that for any $\chi\in X^*(T)$, $\Ha_\chi \colonequals \psi^*(h_\chi)$ is a section of $\Vcal_{\flag}( \chi-qw_{0,I}(\sigma^{-1}\chi))$ over $\Ccal_{w_0}$. Furthermore, for $\chi\in X^*_+(T)$, the section $\Ha_\chi$ extends to a global section over $\GF^\mu$.

\begin{definition}\label{def-schubert-section}
We say that $f\in H^0(\GF^{\mu},\Vcal_{\flag}(\lambda))$ is a Schubert section if there exists $\chi \in X_{+}^*(T)$ such that $\lambda = \chi-q w_{0,I} (\sigma^{-1}\chi)$ and $f=\Ha_\chi$ up to a nonzero scalar.
\end{definition}

Recall that $H^0(\GF^{\mu},\Vcal_{\flag}(\lambda))=H^0(\GZip^\mu,\Vcal_I(\lambda))$, and this space injects into the representation $V_I(\lambda)$ by $\ev_1$ (see \eqref{injection-ev1}). Write 
\begin{equation}
V_{I,\Sbt}(\lambda)\subset V_I(\lambda)
\end{equation}
for the subspace of Schubert sections. By unicity of $h_\chi$ up to scalar, we have $\dim(V_{I, \Sbt}(\lambda))\leq 1$, and $V_{I, \Sbt}(\lambda)\neq 0$ if and only if there exists $\chi \in X_{+}^*(T)$ such that $\lambda = \chi-q w_{0,I} (\sigma^{-1}\chi)$. We will call $V_{I, \Sbt}(\lambda)$ the \emph{Schubert line} of $V_I(\lambda)$.

\begin{lemma}\label{lemma-zero-eigenfunction}
Let $H$ be a connected, smooth algebraic group over $k$ acting on an irreducible normal $k$-variety $X$. Let $f\colon X\to \AA^1$ be a nonzero function such that the vanishing locus of $f$ is $H$-stable. Then there exists a character $\lambda\in X^*(H)$ such that $f$ satisfies $f(h\cdot x)=\lambda(h)f(x)$ for all $(h,x)\in H\times X$.
\end{lemma}

\begin{proof}
First, note that since $H$ is connected, each irreducible component of $\div(f)$ is stable under the action of $H$. 
Hence we have $\div(h\cdot f)=\div(f)$ for all $h\in H$. Consider the rational map $\psi\colon H\times X\dashrightarrow \AA^1 $, $(h,x)\mapsto \frac{f(h\cdot x)}{f(x)}$. Its divisor is zero, so $\psi$ extends to a non-vanishing map $H\times X\to \GG_{\mathrm{m}}$. 
By \cite[Theorem 2]{Rosenlicht-toroidal}, we can write $\psi(h,x)=A(h)B(x)$ for two non-vanishing functions $A\colon H\to \GG_{\mathrm{m}}$ and $B\colon X\to \GG_{\mathrm{m}}$. Then the claim follows from 
\cite[Theorem 3]{Rosenlicht-toroidal} (\cf arguments in the proof of \cite[Lemma 5.1.2]{Koskivirta-automforms-GZip}).
\end{proof}

\begin{proposition}\label{prop-equiv-Schubert-sec}
For a regular map $f\colon G_k \to \AA^1$, the following are equivalent:
\begin{equivlist}
\item There is $\lambda\in X^*(T)$ such that $f\in H^0(\GF^\mu,\Vcal_{\flag}(\lambda))$ and $f$ is a Schubert section.
\item $f$ is non-vanishing on $C_{w_0}$.
\end{equivlist}
\end{proposition}
\begin{proof}
(i) $\Rightarrow$ (ii) is easy since $f$ arises from a section over $\Sbt$, and all such sections are of the form $h_\chi$ as in \eqref{hchi}. Conversely, assume $f$ is non-vanishing on $C_{w_0}$. Then the vanishing locus of $f$ is necessarily a union of $B\times {}^zB$-orbits. By Lemma \ref{lemma-zero-eigenfunction}, $f$ is an $B\times {}^zB$-eigenfunction. Twisting by $z$, the function $g\mapsto f(gz)$ is a $B\times B$-eigenfunction, hence it is of the form $h_\chi$ for some character $\chi\in X^*(T)$. This shows the result.
\end{proof}

\subsection{Partial Hasse invariants}
We define the notion of partial Hasse invariant. There is an ambiguity between partial Hasse invariants for $\GZip^\mu$ and $\GF^\mu$, therefore we use a different terminology for each case.

\subsubsection{Zip partial Hasse invariants}\label{sec-zipPHI}  

Consider the stratification $\GZip^\mu = \bigsqcup_{w\in {}^I W} \Xcal_w$ defined at the end of \S \ref{subsec-zipstrata}. The codimension one strata are $\Xcal_{w_{0,I} s_\alpha w_0}$ for $\alpha\in \Delta^P$. Also recall that for each $\lambda\in X^*(L)$, the vector bundle $\Vcal_I(\lambda)$ 
on $\GZip^\mu$ is a line bundle.

\begin{definition}\label{def-zip-PHI}
A zip partial Hasse invariant $h$ is a section of a line bundle $\Vcal_I(\lambda)$ (for $\lambda \in X^*(L)$) 
over $\GZip^\mu$, whose vanishing locus is the Zariski closure of a codimension one stratum $\overline{\Xcal}_{w_{0,I} s_\alpha w_0}$. We say that $h$ is a strict zip partial Hasse invariant if furthermore it has multiplicity one.
\end{definition}

Viewing a zip partial Hasse invariant as a map $G_k \to \AA^1$ as in \eqref{globquot}, its vanishing locus is the closure of an $E$-orbit of codimension one. Conversely, 
we have the following: 
\begin{lemma}\label{lem:cd1zpH}
Assume that $h \colon G_k \to \AA^1$ is a regular map whose vanishing locus is the closure of an $E$-orbit of codimension one. 
Then, there exists a character $\lambda \in X^*(L)$ such that under the identification \eqref{globquot}, $h$ identifies with an element of $H^0(\GZip^\mu,\Vcal_I(\lambda))$, and it is thus a zip partial Hasse invariant. 
\end{lemma}
\begin{proof}
By Lemma \ref{lemma-zero-eigenfunction} and the fact that any character of $E$ factors through $L$, the function $h$ satisfies $h(agb^{-1})=\lambda(a)h(g)$ for some character $\lambda \in X^*(L)$. Hence the claim holds. 
\end{proof}

\begin{proposition}\label{prop-set-of-zipPHI}
Let $\Xcal_{w_{0,I}s_\alpha w_0}$ be a codimension one stratum. There exists a zip partial Hasse invariant $h_0$ for $\Xcal_{w_{0,I}s_\alpha w_0}$. Furthermore, if we choose $h_0$ such that $\deg(\div(h_0))$ is minimal, the set of all zip partial Hasse invariants for that stratum is given by
\begin{equation}
\{ s \chi h_0^m \mid s\in k^\times, \ \chi\in X^*(G), \ m\geq 1  
\}.
\end{equation}
\end{proposition}

\begin{proof}
Consider the Zariski closure $Z\colonequals \overline{G}_{w_{0,I}s_\alpha w_0}$. By \cite[Proposition 4.5]{Knop-Kraft-Luna-Vust-Local-properties}, the Picard group $\Pic(G)$ is finite, hence there exists a function $h_0\colon G_k \to \AA^1$ whose divisor is $d_0 [Z]$ for some $d_0\geq 1$. 
Hence there exists $\lambda_0\in X^*(L)$ such that $h_0\in H^0(\GZip^\mu,\Vcal_I (\lambda_0))$ by Lemma \ref{lem:cd1zpH}, 
and it is a zip partial Hasse invariant. Now, take $d_0\geq 1$ minimal for this property. It is easy to see that any other function $G_k \to \AA^1$ with vanishing locus $Z$ is of the form $f h^{m}_0$ where $f\colon G_k \to \GG_{\mathrm{m}}$ is non-vanishing. By \cite[Theorem 3]{Rosenlicht-toroidal}, the function $f$ is of the form $f=s\chi$ for $s\in k^\times$ and $\chi \in X^*(G)$.
\end{proof}

\begin{lemma}\label{lem:XGnvan}
Assume $h$ is a nonzero section of $\Vcal_I (\lambda)$ for some $\lambda\in X^*(G)$. 
Then $h$ is non-vanishing on $G$. 
\end{lemma}
\begin{proof}
We can see that $h^{q-1}\lambda$ is a section of $\Vcal_I (0)$. But since the action of $E$ is dense in $G$, any section of $\Vcal_I (0)$ must be constant. Hence we obtain the claim. 
\end{proof}

\begin{corollary}
For a given codimension one stratum and a given $\lambda\in X^*(L)$, there is up to scalar 
at most one zip partial Hasse invariant of weight $\lambda$ for that stratum. 
\end{corollary}
\begin{proof}
We use notations in Proposition \ref{prop-set-of-zipPHI}. 
Assume that 
$\chi h_0^m$ and $\chi' h_0^{m'}$ 
have the same weight 
for $\chi,\chi' \in X^*(G)$ and $m \geq m' \geq 1$. By Proposition \ref{prop-set-of-zipPHI}, it suffices to show that $\chi =\chi'$ and $m=m'$. 
Writing $\lambda_0$ for the weight of $h_0$, we have $(1-q)\chi + m\lambda_0 = (1-q)\chi'+m'\lambda_0$. Hence $(m-m')\lambda_0\in X^*(G)$. 
Then we have $m=m'$ by Lemma \ref{lem:XGnvan}. 
Hence we also have $\chi=\chi'$. 
\end{proof}

\subsubsection{Flag partial Hasse invariants}\label{sec-flagPHI}
Similarly, we define flag partial Hasse invariant. Recall that we have a stratification $(\Ccal_w)_{w\in W}$ on $\GF^\mu$. The codimension one strata are $\Ccal_{s_\alpha w_0}$ for $\alpha \in \Delta$. For each $\lambda\in X^*(T)$, we have a line bundle $\Vcal_{\flag}(\lambda)$ on $\GF^\mu$.

\begin{definition}\label{def-flag-PHI}
A flag partial Hasse invariant $h$ is a global section of a line bundle $\Vcal_{\flag}(\lambda)$ over $\GF^\mu$ for $\lambda \in X^*(T)$, whose vanishing locus is the Zariski closure of a codimension one flag stratum $\overline{\Ccal}_{s_\alpha w_0}$. We say that $h$ is a strict flag partial Hasse invariant if furthermore it has multiplicity one.
\end{definition}

A flag partial Hasse invariant identifies with a map $G_k \to \AA^1$ whose vanishing locus is the closure of a $B\times {}^zB$-orbit  of codimension one. Conversely, similarly to the discussion in \S\ref{sec-zipPHI}, any such map $G_k \to \AA^1$ is a flag partial Hasse invariant. Similarly to Proposition \ref{prop-set-of-zipPHI}, we have
\begin{proposition}\label{prop-set-of-flagPHI}
Let $\Ccal_{ s_\alpha w_{0}}$ be a codimension one stratum. There exists a flag partial Hasse invariant $h'_0$ for $\Ccal_{s_\alpha w_{0}}$. Furthermore, if we choose $h'_0$ such that $\deg(\div(h'_0))$ is minimal, the set of all flag partial Hasse invariants for that stratum is given by
\begin{equation}
\{ s \chi h'^{m}_{0} \mid s\in k^\times, \ \chi\in X^*(G), \ m\geq 1 \}.
\end{equation}
\end{proposition}

Again, a flag partial Hasse invariant of weight $\lambda$ is uniquely determined up to scalar by $\alpha\in \Delta$ and $\lambda$. Finally, we note that if $\Pic(G)=0$, then all zip (resp.\ flag) strata admit strict zip (resp.\ flag) Hasse invariants, by the proof of Proposition \ref{prop-set-of-zipPHI}.

\begin{lemma}\label{Lemma-zipflag}
Assume that $P$ is defined over $\FF_q$. If $h\in H^0(\GZip^\mu,\Vcal_I(\lambda))$ (for $\lambda\in X^*(L)$) is a zip partial Hasse invariant, then $\pi^*(h)\in H^0(\GF^\mu,\Vcal_{\flag}(\lambda))$ is a flag partial Hasse invariant.
\end{lemma}

\begin{proof}
When $P$ is defined over $\FF_q$, the unique open $E$-orbit in $G$ is stable by $B\times {}^zB$ by \cite[Corollary 2.15]{Wedhorn-bruhat}. Hence, the closure of any $E$-orbit of codimension $1$ coincides with the closure of a $B\times {}^zB$-orbit of codimension one. The result follows immediately.
\end{proof}

\section{Verschiebung homomorphism}
Let $\Mcal$ be a vector bundle on an algebraic  stack (or scheme) $S$ over $\FF_q$, and let $\Mcal^{(q)} \colonequals F_S^*\Mcal$ be the pullback by the absolute Frobenius $F_S \colon S\to S$. In general, there is no natural map $\Mcal\to \Mcal^{(q)}$. However, if $\Acal \to S$ is an abelian variety over a scheme of characteristic $p$, then $\Hcal\colonequals \Hcal^1_{\dR}(\Acal/S)$ admits a Verschiebung map $V\colon \Hcal \to \Hcal^{(q)}$. We will see that certain vector bundles $\Vcal(\rho)$ (for a $P$-representation $(V,\rho)$) on $\GZip^\mu$ are endowed with a Verschiebung homomorphism $\vfr_\Vcal  \colon \Vcal \to \Vcal^{[1]}$, where $\Vcal^{[1]}$ is an analogue of the Frobenius-twist. We avoid the notation "$V$" for the Verschiebung. We will later that partial Hasse invariants admit a factorization in terms of the Verschiebung of certain vector bundles $\Vcal(\rho)$.

\subsection{\texorpdfstring{Verschiebung homomorphism for $L$-representations}{}}\label{subsec-verschiebung}

Assume that $P$ is defined over $\FF_q$. For a $P$-representation $(V,\rho)$, denote by $(V^{[1]},\rho^{[1]})$ the representation such that $V^{[1]}=V$ and
\begin{equation}\label{equ-rho1}
  \rho^{[1]} \colon  P \xrightarrow{\varphi} P \xrightarrow{\rho}\GL(V)
\end{equation}
where $\varphi$ denotes the Frobenius homomorphism of $P$. If $\Vcal=\Vcal(\rho)$ denotes the vector bundle attached to $\rho$, we sometimes write $\Vcal^{[1]} \colonequals \Vcal(\rho^{[1]})$. 

Recall (Theorem \ref{fullfaith}) that for any two $L$-representations $(V_1,\rho_1)$, $(V_2,\rho_2)$, there is a bijection between morphisms of vector bundles $\Vcal(\rho_1)\to \Vcal(\rho_2)$ on $\GZip^\mu$ and maps $F_{\mathrm{MF}} (\rho_1)\to F_{\mathrm{MF}} (\rho_2)$ of $\Delta^P$-filtered $L_{\varphi}$-modules. If $V=\bigoplus_{\nu}V_\nu$ denotes the weight decomposition of $V$, then one has $V^{[1]}=\bigoplus_{\nu}V^{[1]}_{q\sigma^{-1}\nu}$ and $V^{[1]}_{q\sigma^{-1}\nu}=V_\nu$. Recall from \S\ref{sec-Lreps} that for $r\in \RR$ and $\alpha\in \Delta^P$, the filtration $V^{\alpha}_{\geq r}$ is defined as the direct sum of $V_\nu$ for all $\nu$ satisfying $\langle \nu,\delta_{\alpha} \rangle \geq r$. Therefore, we find for all $\alpha \in \Delta^P$ and all $r\in \RR$:
\begin{equation}\label{eqV1n}
 \left( V^{[1]} \right)^{\alpha}_{\geq r} = V^{\sigma\alpha}_{\geq \frac{r}{q}}.
\end{equation}

Let $\rho\in \Rep(L)$ be an arbitrary $L$-representation. Similarly to the subspace $V_{\geq 0}^{\Delta^P}$ (see \eqref{equ-VDeltaP}), define a subspace $V_{\GS}\subset V$ as follows. Let $V=\bigoplus_{\nu\in X^*(T)} V_\nu$ be the $T$-weight decomposition of $V$. Define $V_{\GS}$ as the direct sum of the weight spaces $V_\nu$ such that $\langle \nu,\alpha^\vee \rangle \leq 0$ for all $\alpha\in \Delta^P$.

\begin{lemma}\label{lem-GS}
If $\sigma \alpha =\alpha$ for all $\alpha\in \Delta^P$, then $V_{\GS} = V_{\geq 0}^{\Delta^P}$.
\end{lemma}

\begin{proof}
This was already pointed out in \cite{Imai-Koskivirta-vector-bundles} (sentence after equation (3.4.3)). It follows from the fact that if $\sigma \alpha = \alpha$, then $\delta_\alpha = -\frac{\alpha^\vee}{q-1}$.
\end{proof}

\begin{proposition}\label{propver}
Let $(V,\rho)$ be an $L$-representation such that $V_{\GS}=V$. Then the identity map $\id \colon V\to V=V^{[1]}$ is a morphism of  $\Delta^P$-filtered $L_{\varphi}$-modules $F_{\mathrm{MF}} (\rho)\to F_{\mathrm{MF}}(\rho^{[1]})$, hence induces a morphism of vector bundles
\begin{equation}\label{VerVB}
\vfr_\rho \colon \Vcal(\rho) \to \Vcal(\rho^{[1]}).
\end{equation}
We call this map the Verschiebung homomorphism of $\Vcal(\rho)$.
\end{proposition}

\begin{proof}
Since $\rho^{[1]}=\rho \circ \varphi$, it is clear that $\id$ is $L(\FF_q)$-equivariant. We show that it is compatible with $\alpha$-filtrations. Using \eqref{eqV1n}, we must prove $V ^{\alpha}_{\geq r} \subset  \left( V^{[1]} \right)^{\alpha}_{\geq r}=V^{\sigma\alpha}_{\geq \frac{r}{q}}$ for all $\alpha\in \Delta^P$. Since $\delta_\alpha=\wp_*^{-1}(\alpha^\vee)$, we have $\alpha^{\vee}=\wp_*(\delta_{\alpha})=\delta_{\alpha} -q \delta_{\sigma\alpha}$ (see \ref{subsec-global-sections}). 
Suppose that $\langle \nu,\delta_\alpha \rangle \geq r$ for a weight $\nu$. We have $q \langle \nu, \delta_{\sigma \alpha} \rangle =\langle \nu,\delta_\alpha \rangle -\langle \nu, \alpha^\vee \rangle$ and  $\langle \nu, \alpha^\vee \rangle\leq 0$ by assumption. Hence $\langle \nu, \delta_{\sigma \alpha} \rangle \geq \frac{r}{q}$. This shows that $V ^{\alpha}_{\geq r} \subset V^{\sigma\alpha}_{\geq \frac{r}{q}}$ and terminates the proof. 
\end{proof}

\begin{rmk}
On the $\mu$-ordinary locus, one has $\Vcal(\rho)|_{\Ucal_\mu}=\Vcal(\rho^{[1]})|_{\Ucal_\mu}$, since $\rho^{[1]}|_{L(\FF_q)}=\rho|_{L(\FF_q)}$. Proposition \ref{propver} shows that this isomorphism $\Vcal(\rho)|_{\Ucal_\mu} \to \Vcal(\rho^{[1]})|_{\Ucal_\mu}$ extends to a map of vector bundles $\vfr_\rho \colon \Vcal(\rho) \to \Vcal(\rho^{[1]})$ over $\GZip^\mu$ (which need not be an isomorphism). 
\end{rmk}

Recall the Griffiths--Schmid cone (\cite[Definition 1.8.1]{Koskivirta-automforms-GZip}) $C_{\GS}$, defined as follows:

\begin{definition}\label{GSdef}
Let $C_{\GS}$ denote the set of characters $\lambda\in X^*(T)$ satisfying
\begin{align}
\langle \lambda, \alpha^\vee \rangle &\geq 0 \ \textrm{ for }\alpha\in I, \\
\langle \lambda, \alpha^\vee \rangle &\leq 0 \ \textrm{ for }\alpha\in \Phi_+ \setminus \Phi_{L,+}. 
\end{align}
\end{definition}

\begin{lemma}\label{lemma-cone-GS}
A character $\lambda$ is in $C_{\GS}$ if and only if $-w_{0,I}\lambda$ is $G$-dominant.
\end{lemma}
\begin{proof}
This follows from the fact that $w_{0,I}$ preserves $\Phi_+ \setminus \Phi_{L,+}$ and swaps $I$ and $-I$.
\end{proof}

The conditions defining $C_{\GS}$ were first understood by Griffiths--Schmid (\cf \cite[p.275]{Griffiths-Schmid-homogeneous-complex-manifolds}) in a relation with Griffiths--Schmid manifolds, a generalization of Shimura varieties to non-minuscule cocharacters. The following is \cite[Proposition 3.7.5(1)]{Koskivirta-automforms-GZip} (where the notation $V(\lambda)_{\leq 0}$ was used instead of $V_I(\lambda)_{\GS}$):

\begin{lemma} \label{GS-Deltapos}
Assume $\lambda\in C_{\GS}$. Then one has $V_I(\lambda)_{\GS}=V_I(\lambda)$. In particular, for all $\lambda \in C_{\GS}$, the vector bundle $\Vcal_I(\lambda)$ admits a Verschiebung homomorphism.
\end{lemma}

\begin{example}\label{exabvar}
Set $G\colonequals\GL_{2n,\FF_q}$, endowed with the cocharacter $\mu\colon \GG_{\mathrm{m}}\to G;\ x\mapsto \diag(xI_n,I_n)$. Let $P,L,Q$ be the subgroups of $G$ attached to $\mu$. Let $(u_1,\dots ,u_{2n})$ be the canonical basis of $V_0=k^{2n}$ and set $V_{0,P}=\Span_k(u_{n+1}, \dots , u_{2n})$ (hence $P$ is the stabilizer of $V_{0,P}$). Let $T\subset L$ be the diagonal torus and $B\subset G$ the lower-triangular Borel subgroup. Identify $X^*(T)=\ZZ^{2g}$ such that $(a_1, \dots ,a_{2g})$ corresponds to the character $\diag(x_1, \dots ,x_{2g})\mapsto \prod_{i=1}^{2g}x_i^{a_i}$. Write also $(e_1, \dots ,e_{2g})$ for the canonical basis of $\ZZ^{2g}$. In this case, a $G$-zip of type $\mu$ over an $\FF_p$-scheme $S$ is equivalent a tuple $(\Mcal,\Omega, F,V)$ satisfying the following:
\begin{enumerate}
\setlength{\leftmargin}{18pt}
     \setlength{\rightmargin}{0pt}
     \setlength{\itemindent}{0pt}
     \setlength{\labelsep}{5pt}
     \setlength{\labelwidth}{13pt}
     \setlength{\listparindent}{\parindent}
     \setlength{\parsep}{0pt}
     \setlength{\itemsep}{\itemsepamount}
     \setlength{\topsep}{\topsepamount}
    \item[\textnormal{(i)}] $\Mcal$ is a locally free $\Ocal_S$-module of rank $2n$ and $\Omega\subset \Mcal$ is a locally free, locally direct factor $\Ocal_S$-submodule of rank $n$.
    \item[\textnormal{(ii)}] $F \colon \Mcal^{(q)}\to \Mcal$ and $V \colon \Mcal\to \Mcal^{(q)}$ are maps of vector bundles satisfying $\Im(F)=\Ker(V)$ and $\Im(V)=\Ker(F)=\Omega^{(q)}$.
\end{enumerate}
Let $f\colon A\to S$ be an abelian scheme of rank $g$ over an $\FF_q$-scheme $S$. We can attach to $A$ a $G$-zip by setting $\Mcal=H^1_{\dR}(A/S)$ and $\Omega$ as the Hodge filtration of $\Mcal$. The Frobenius and Verschiebung maps of $A$ give rise to similar maps on $\Mcal$. Thus, we get a map of stacks 
$$\zeta \colon S\to \GZip^\mu.$$
The $P$-representation $V_{0,P}$ has highest weight $\lambda_\Omega=-e_n$, which lies in $C_{\GS}$. Hence, $\Vcal_I(\lambda_\Omega)$ admits a Verschiebung $\vfr_\Omega \colon  \Vcal_I(\lambda_\Omega) \to \Vcal_I(\lambda_\Omega^{[1]})$ by Lemma \ref{GS-Deltapos}. The pullback by $\zeta$ of this map coincides with the Verschiebung $V \colon \Omega \to \Omega^{(p)}$.
\end{example}

\subsection{\texorpdfstring{Twisting homomorphism on $\Sbt$}{}}

We show that the Verschiebung map constructed in the previous section arises from a similar map the stack $\Sbt$.

We start with a connected, reductive group $G$ over $k$, a Borel pair $(B,T)$ such that $T\subset B \subset G$ and a parabolic subgroup $P$ containing $B$. Write $L$ for the unique Levi subgroup of $P$ containing $T$. Put $z_1=w_{0,I} w_0$. Let $V$ be an $L$-representation, which we view as a $P$-representation trivial on $R_{\mathrm{u}}(P)$. Note that $z_1^{-1}B_L z_1 \subset B$. Define a $B$-representation $\rho[z_1]$ as follows: First, for $b\in z_1^{-1}B_L z_1$, put $\rho[z_1](b)=\rho(z_1 b z_1^{-1})$. Then, extend $\rho[z_1]$ to a $B$-representation trivial on $R_{\mathrm{u}}(z_1^{-1} L z_1)$. We may consider the $B\times B$-representations $(\rho,\rho_0)$ and $(\rho_0,\rho[z_1])$, both with underlying vector space $V$. Put $u_\rho\colonequals \rho(w_{0,I})$, viewed as a map $u_\rho \colon V\to V$. It is easy to see that $u_\rho$ is $S_{w_0}$-equivariant (see \ref{sec-glob-sec-Sbt} for the definition of $S_{w_0}$). Indeed, for all $t\in T$, and $x\in V$ we have
\begin{align*}
   \left( (t,w_0 t w_0)\cdot u_\rho \right) (x) 
   &= \rho[z_1](w_0 t w_0) \rho(w_{0,I}) (\rho(t)^{-1}x) \\ 
   &= \rho(w_{0,I}tw_{0,I}^{-1})\rho(w_{0,I})\rho(t)^{-1}x = u_\rho (x).
\end{align*}
Hence, $u_\rho$ defines a morphism of vector bundles over the open substack $\Sbt_{w_0}\subset \Sbt$:
\begin{equation}\label{u-rho-map}
u_\rho \colon \Vcal_{\Sbt}(\rho, \rho_0)|_{\Sbt_{w_0}} \to \Vcal_{\Sbt}(\rho_0, \rho[z_1])|_{\Sbt_{w_0}}.
\end{equation}

\begin{theorem}\label{urho-theorem}
Assume that $V=V_{\GS}$. Then $u_\rho$ extends to a map of vector bundles $u_\rho  \colon \Vcal_{\Sbt}(\rho, \rho_0) \to \Vcal_{\Sbt}(\rho_0, \rho[z_1])$ over $\Sbt$.
\end{theorem}

Before we prove Theorem \ref{urho-theorem}, we discuss the following situation. We may choose $\FF_q\subset k$ such that $G$, $B$, $T$ and $P$ are defined over $\FF_q$. Let $Q \colonequals LB^+$ be the opposite parabolic subgroup of $P$ with respect to $L$. Hence $\Zcal=(G,P,L,Q,L)$ is a zip datum for the $q$-th power Frobenius, and we have $\Zcal=\Zcal_{\mu}$, for any dominant cocharacter $\mu\colon \GG_{\mathrm{m},k}\to G_k$ with centralizer $L$. Hence, if we define $z=z_1=w_{0,I}w_0$, then $(B,T,z)$ is a frame of $\Zcal$ (Lemma \ref{lem-framemu}). Furthermore, by \eqref{Rrhorep}, we have
\begin{align}
    \psi^*\Vcal_{\Sbt}(\rho,\rho_0)&=\Vcal_{\flag}(\rho),  \\
    \psi^*\Vcal_{\Sbt}(\rho_0,\rho[z_1])&=\Vcal_{\flag}(\rho^{[1]}).
\end{align}
Now, let $(V,\rho)$ be an $L$-representation such that $V_{\GS}=V$. By Proposition \ref{propver}, we have a Verschiebung map $\vfr_\rho \colon \Vcal(\rho) \to \Vcal(\rho^{[1]})$ on $\GZip^\mu$. Pulling back via the natural projection $\pi\colon \GF^\mu \to \GZip^\mu$ (see \eqref{projmap-flag}), we obtain a map $\pi^*\vfr_\rho \colon \Vcal_{\flag}(\rho) \to \Vcal_{\flag}(\rho^{[1]})$ using \eqref{pullbackV}. Then, Theorem \ref{urho-theorem} above is a consequence of the following, more precise result:
\begin{proposition}\label{prop-w0I-Ver}
Assume that $V_{\GS}=V$. There exists a unique morphism of vector bundles $u_{\rho}\colon \Vcal_{\Sbt}(\rho,\rho_0) \to \Vcal_{\Sbt}(\rho_0,\rho[z_1])$ on the stack $\Sbt$ such that $\psi^*u_\rho = \pi^*\vfr_\rho $. Furthermore, $u_\rho$ extends the map defined in equation \eqref{u-rho-map} over $\Sbt$.
\end{proposition}

\begin{proof}
The uniqueness is clear, since $\psi$ is surjective. Using \eqref{globquot}, view $\pi^*\vfr_\rho$ as a map $f\colon G_k \to \Hom(V,V^{[1]})$. By construction, $f$ is the only such map satisfying: for all $(a,b)\in E$, $f(ab^{-1})=a\cdot \id_V=\rho(\varphi(a)a^{-1})$. We obtain that $f$ satisfies
\begin{equation}\label{relation-ver-Sbt}
f(r_1 x r_2)=\rho(x^{-1}),
\end{equation}
for all $x\in L$, $r_1\in R_{\mathrm{u}}(P)$ and $r_2\in R_{\mathrm{u}}(Q)$. Define $h\colon G_k \to \Hom(V,V)$, by $h(g)\colonequals f(gz_1^{-1})$. 
By \eqref{relation-ver-Sbt}, we have 
\begin{align}
h(bxb'^{-1})&=\rho(\theta^Q_L (z_1b'z_1^{-1}))\circ h(x)\circ \rho(b^{-1})    \\
& = \rho[z_1](b')\circ h(x)\circ \rho(b^{-1}) 
\end{align}
for all $b,b'\in B$ and all $x\in L z_1$. 
The same equality is true for 
all $b,b'\in B$ and all $x\in G_k$, 
since $B Lz_1 B$ is dense in $G_k$. Thus, $h$ defines a morphism $u_{\rho}\colon \Vcal_{\Sbt}(\rho,\rho_0) \to \Vcal_{\Sbt}(\rho_0,\rho[z])$ as claimed. Finally, we show that $u_\rho$ coincides with \eqref{u-rho-map} over $\Sbt_{w_0}$. Viewed as a map $V\to V$, the map $\pi^*\vfr_\rho$ corresponds to $h(w_0)=f(w_{0,I}^{-1})=\rho(w_{0,I})$, hence the result.
\end{proof}

\begin{rmk}
The statement of Theorem \ref{urho-theorem} also makes sense and holds for groups over an arbitrary algebraically closed field $F$. 
To explain this, we first write explicitly conditions (i) and (ii) for the map $u_\rho=\rho(w_{0,I})$.  
Using \S\ref{sec-glob-sec-Sbt}, one sees that the map \eqref{u-rho-map} extends to $\Sbt$ if and only if for all $d\in \ZZ$ and $\chi_1\in X^*(T)$ such that $d+\langle \chi_1, \alpha^\vee \rangle < 0$, the map
\begin{equation}\label{mapstozero}
    \sum_{j_1 \geq d} (-1)^{j_1}  E_{w_{0,I}\alpha}^{(j_1-d)} \rho(w_{0,I})  E_{-\alpha}^{(j_1)} =  \rho(w_{0,I}) \sum_{j_1 \geq d} (-1)^{j_1}  E_{\alpha}^{(j_1-d)} E_{-\alpha}^{(j_1)}
\end{equation}
maps to zero under the projection \eqref{proj-map}. We deduce that $u_\rho$ extends to $\Sbt$ if and only if for all $\alpha\in \Delta$ and all $d\in \ZZ$ such that $d+\langle w_{0,I}\nu-\nu-dw_{0,I}\alpha, \alpha^\vee \rangle <0$ we have
\begin{equation}\label{u-rho-criterion}
    \sum_{j\geq d} (-1)^j E_{\alpha}^{(j-d)} E_{-\alpha}^{(j)}(V_{\nu}) =0.
\end{equation}
Now, assume that $F$ has characteristic zero. By the classification of reductive groups, there exists a model of $G$ over $\overline{\QQ}$, and hence over a number field $K$,
and even over $\Ocal_K[\frac{1}{N}]$ for some integer $N\geq 1$. After possibly changing $\Ocal_K$ and $N$, we may assume that all other objects ($P$, $B$, the representation $V$, etc.) admit a model over $\Ocal_K[\frac{1}{N}]$. Therefore, for infinitely many primes $p$, the equation \eqref{u-rho-criterion} holds in $V\otimes_{\Ocal_K[\frac{1}{N}]} \FF_p$, since we showed the result in characteristic $p$. Hence it also holds in $V$, and we deduce that $u_\rho$ extends to $\Sbt$. In particular, formula \eqref{u-rho-criterion} above holds in general.
\end{rmk}

\begin{example}
For $G=\Sp(4)$, $L=\GL_2$ and $V=\Sym^n(\Std)$, formula \eqref{u-rho-criterion} amounts to the following: For all $d\in \ZZ$ and $0\leq i \leq n$ such that $4i-2n+3d<0$, we have
\begin{equation}
    \sum_{j=0}^{n-i-d}(-1)^j\binom{n-i}{j+d} \binom{j+d+i}{d+i} =0.
\end{equation}
The above discussion shows that this formula holds (in $\ZZ$, not only in $\FF_q$). It is also possible to show this formula 
directly using the binomial transform. 
\end{example}

\subsection{\texorpdfstring{Verschiebung homomorphism for $G$-representations}{}}

We study the existence of Verschiebung maps on $G$-representations. We restrict to representations of the form $V_\Delta(\chi)=\Ind_B^G(\lambda)$ for a dominant $\chi\in X_{+}^*(T)$. There is a map of $L$-representations:
\begin{equation}
  \Pi_\chi \colon V_\Delta(\chi)|_L\to V_I(w_{0,I}w_0 \chi)
\end{equation}
defined as follows. Let $f\in V_{\Delta}(\chi)$, viewed as a function $f \colon G\to \AA^1$ satisfying $f(xb)=\chi^{-1}(b)f(x)$ for all $x\in G$ and $b\in B$. Define $\Pi_\chi(f)$ as the function $L\to \AA^1$ mapping $x\in L$ to $f(x w_{0,I}w_0)$. Clearly $\Pi_\chi(f)$ lies in $V_I(w_{0,I}w_0 \chi)$ and $\Pi_\chi$ is $L$-equivariant.

\begin{proposition} \label{prop-PiX-isom}
The map $\Pi_\chi$ induces an isomorphism of $L$-representations
\begin{equation}
V_{\Delta}(\chi)^{R_{\mathrm{u}}(P)} \to V_I(w_{0,I}w_0 \chi).    
\end{equation}
In particular, $V_{\Delta}(\chi)|_L$ decomposes as $V_{\Delta}(\chi)|_L=V_I(w_{0,I}w_0 \chi) \oplus \Ker(\Pi_\chi)$.
\end{proposition}

\begin{proof}
We first show that $\Pi_\chi$ is injective on $V_{\Delta}(\chi)^{R_{\mathrm{u}}(P)}$. Suppose $f\in V_{\Delta}(\chi)^{R_{\mathrm{u}}(P)}$ satisfies $\Pi_\chi(f)=0$. Then $f(x w_0b)=0$ for all $x\in P$ and all $b\in B$. Since $P w_0 B$ is Zariski dense in $G$, we deduce $f=0$. We now show that $\Pi_\chi|_{V_{\Delta}(\chi)^{R_{\mathrm{u}}(P)}}$ is surjective. Let $f'\in V_I(w_{0,I}w_0 \chi)$, viewed as a function $f' \colon L\to \AA^1$. We must define $f(x w_{0,I} w_0 b) = f'(\theta^P_L(x) ) \chi(b)^{-1}$ for all $(x,b)\in P\times B$. We claim that $f$ is well-defined on $P w_{0}B$. Indeed, suppose $x w_{0,I} w_0 b = x' w_{0,I} w_0 b'$ for $x,x'\in P$ and $b,b'\in B$. We obtain $x'^{-1}x = w_{0,I}w_0 b'b^{-1}w_{0}^{-1}w_{0,I}^{-1}\in B_L$, hence
\begin{align}
f'(\theta^P_L(x) ) \chi(b)^{-1} &= f'(\theta^P_L(x') x'^{-1}x)  \chi(b)^{-1} \\ 
&= f'(\theta^P_L(x'))(w_{0,I}w_0 \chi)^{-1}( w_{0,I}w_0 b'b^{-1}w_{0}^{-1}w_{0,I}^{-1}) \chi(b)^{-1} = f'(\theta^P_L(x') ) \chi(b')^{-1}.
\end{align}
Next, we show that $f$ extends to a regular function on $G$. For this, we consider the map $\psi_\alpha \colon P\times \AA^1 \to G$ for $\alpha \in \Delta \setminus I$ given by $\psi_\alpha \colon (x,t)\mapsto x \phi_\alpha(A(t)) w_0$ for $(x,t)\in P \times \AA^1$. It suffices to show that for all $\alpha \in \Delta \setminus I$, the map $f\circ \psi_\alpha \colon P\times \GG_{\mathrm{m}} \to \AA^1$ extends to a map $P\times \AA^1 \to \AA^1$. For this, we write $A(t)$ as in \eqref{At}, and we find for $(x,t)\in P\times \GG_{\mathrm{m}}$:
\begin{equation}
    f(\psi_\alpha(x,t))= f'(\theta^P_L(x) \alpha^\vee(t) w_{0,I}) = f'(\theta^P_L(x)w_{0,I}) t^{-\langle w_{0,I}w_0 \chi, w_{0,I}\alpha^\vee \rangle} = f'(\theta^P_L(x)w_{0,I}) t^{-\langle \chi, w_{0}\alpha^\vee \rangle}.
\end{equation}
Since $\chi$ is dominant, we have $-\langle \chi, w_{0}\alpha^\vee \rangle \geq 0$, which terminates the proof.
\end{proof}

Let $(G,\mu)$ be a cocharacter datum over $\FF_q$, with attached zip datum $\Zcal_\mu=(G,P,L,Q,M,\varphi)$. Assume in the rest of this section that $P$ is defined over $\FF_q$. Let again $\chi \in X^*_{+}(T)$. We identify $V_I(w_{0,I}w_0 \chi)$ with $V_{\Delta}(\chi)^{R_{\mathrm{u}}(P)}$ and view $\Pi_\chi$ as a map $V_{\Delta}(\chi)\to V_{\Delta}(\chi)$ with image $V_I(w_{0,I}w_0 \chi)$. It is clear that $\Pi_\chi \colon V_{\Delta}(\chi) \to V_{\Delta}(\chi^{[1]})$ is $L(\FF_q)$-equivariant. Therefore, by Lemma \ref{lem-Umu-sections}, the map $\Pi_\chi$ gives rise to a morphism of vector bundles
\begin{equation}
  \vfr_\chi \colon  \Vcal_{\Delta}(\chi)|_{\Ucal_\mu} \to \Vcal_{\Delta}(\chi^{[1]})|_{\Ucal_\mu}.
\end{equation}
In the following, we investigate whether this map extends to a morphism $\vfr_\chi \colon  \Vcal_{\Delta}(\chi) \to \Vcal_{\Delta}(\chi^{[1]})$ over the whole stack $\GZip^\mu$. When it does, we call the unique extension the Verschiebung homomorphism of $\Vcal_{\Delta}(\chi)$

\begin{lemma}\label{lemma-roots}
Let $\alpha\in \Delta^P$, $\beta\in \Delta_L$ and assume $i \alpha + j\beta \in \Phi_+$ for $i\geq 0, j\geq 1$. Then we have $i\leq - j\langle \beta , \alpha^\vee \rangle$.
\end{lemma}
\begin{proof}
The claim is equivalent to $(n-j\langle\beta,\alpha^{\vee}\rangle)\alpha + j\beta \notin \Phi_+$ 
for any $n \geq 1$.
This is equivalent to
$s_{\alpha}(j\beta-n\alpha) \notin \Phi_+$ 
for any $n \geq 1$. 
If $s_{\alpha}(j\beta-n\alpha) \in \Phi_+$ 
for some $n \geq 1$, 
then $j\beta-n\alpha$ is a root, 
but this contradicts that $\alpha$ and $\beta$ are different simple roots.
\end{proof}

\begin{lemma} \label{Gamma-extends}
Let $\alpha\in \Delta^P$. For $(x,t)\in R_{\mathrm{u}}(B_L)\times \GG_{\mathrm{m}}$, define
\begin{equation}
    \Gamma(x,t) =  \phi_{\alpha} \left(\begin{matrix}
t&0\\1&t^{-1}
\end{matrix} \right) x \phi_{\alpha} \left(\begin{matrix}
t&0\\1&t^{-1}
\end{matrix} \right)^{-1} \in G.
\end{equation}
Then $\Gamma$ extends to a regular map $R_{\mathrm{u}}(B_L)\times \AA^1 \to G$.
\end{lemma}

\begin{proof}
We can view $\Gamma$ as an element of $R[t,t^{-1}]$ where $R=k[R_{\mathrm{u}}(B_L)]$ is the ring of functions of $R_{\mathrm{u}}(B_L)$. We need to show that $\Gamma\in R[t]$. For this, it suffices to check that for any $x\in R_{\mathrm{u}}(B_L)$, the map $t\mapsto \Gamma(x,t)$ extends to $\AA^1$. Furthermore, we may assume that $x$ is of the form $x=u_{-\beta}(s)$ for $s\in \AA^1$ and $\beta\in \Delta_L$, as elements of this form generate $R_{\mathrm{u}}(B_L)$. We may write $\Gamma(x,t)= \alpha^{\vee}(t) F(t,s) \alpha^{\vee}(t)^{-1}$, where $F(t,s)=u_{-\alpha}(t)u_{-\beta}(s) u_{-\alpha}(t)^{-1}$. Write $[u_{-\alpha}(t),u_{-\beta}(s)]$ for the commutator of $u_{-\alpha}(t)$ and $u_{-\beta}(s)$. By \cite[Proposition 8.2.3]{Springer-Linear-Algebraic-Groups-book}, we can write
\begin{equation}\label{comm-eq}
    F(t,s)u_{-\beta}(s)^{-1} = [u_{-\alpha}(t),u_{-\beta}(s)] = \prod_{i,j> 0} \phi_{i\alpha+j\beta}\left(\begin{matrix}
1&0\\u_{i,j} t^i s^j&1
\end{matrix} \right)
\end{equation}
for some $u_{i,j}\in k$, and for an arbitrarily chosen order on the set of positive roots $\Phi_+$. 
Returning to $\Gamma(x,t)$, we deduce
\begin{equation}
    \Gamma(x,t) = \left(\prod_{i,j> 0} \phi_{i\alpha+j\beta}\left(\begin{matrix}
1&0\\u_{i,j} t^{i-\langle i\alpha+j\beta, \alpha^\vee \rangle}s^j&1
\end{matrix} \right) \right) \phi_{\beta}\left(\begin{matrix}
1&0\\st^{-\langle\beta,\alpha^\vee\rangle}&1
\end{matrix} \right). 
\end{equation}
Since $\alpha$ and $\beta$ are distinct simple roots, $-\langle\beta,\alpha^\vee\rangle\geq 0$. Moreover, $i-\langle i\alpha+j\beta, \alpha^\vee \rangle=-i-j\langle \beta,\alpha^\vee \rangle \geq 0$ by Lemma \ref{lemma-roots}. 
Hence the result follows. 
\end{proof}

\begin{lemma} \label{lemma-Ffalpha}
Let $\chi\in X^*_{+}(T)$ and $f\in V_{\Delta}(\chi)$, viewed as a function $G\to \AA^1$. For $\alpha\in \Delta^P$, consider the function $L\times \GG_{\mathrm{m}} \to \AA^1$
\begin{equation}
  F_{f,\alpha}  \colon  (x,t) \mapsto  f\left( \phi_{\alpha} \left(\begin{matrix}
t&0\\1&t^{-1}
\end{matrix} \right) x w_0 \right).
\end{equation}
Then $F_{f,\alpha}$ extends to a function $L\times \AA_1\to \AA^1$.
\end{lemma}

\begin{proof}
We may view $F_{f,\alpha}$ as an element of $R_1[t,t^{-1}]$ where $R_1=k[L]$ is the ring of functions of $L$ over $k$. We must show that $F_{f,\alpha}$ lies in $R_1[t]$. Furthermore, it suffices to show that this is true for $x$ in a dense subset of $L$ (because we can write $F_{f,\alpha}=\sum_{j\geq -N}F_j(x)t^j$ for some $F_j\in R_1$ and $N\in \NN$, and we deduce by density that $F_j=0$ for all $j<0$). The elements of the form $x=yz$ for $y\in R_{\mathrm{u}}(B_L)$ and $z\in B^+_L$ form an open dense subset of $L$. Note that
\begin{equation}
 F_{f,\alpha}(yz,t) = \chi(w_{0}^{-1} z w_0)^{-1} F_{f,\alpha}(y,t).
\end{equation}
Hence, it suffices to show that $t\mapsto F_{f,\alpha}(x,t)$ extends to $\AA^1$ for $x\in  R_{\mathrm{u}}(B_L)$. We have
\begin{equation}
     \phi_{\alpha} \left(\begin{matrix}
t&0\\1&t^{-1}
\end{matrix} \right) x =  \Gamma(x,t) \phi_{\alpha} \left(\begin{matrix}
t&0\\1&t^{-1}
\end{matrix} \right) = \Gamma(x,t) \phi_{\alpha} \left(\begin{matrix}
t&-1\\1&0
\end{matrix} \right) \phi_{\alpha} \left(\begin{matrix}
1&t^{-1}\\0&1
\end{matrix} \right).
\end{equation}
Using $w_0^{-1} \phi_{\alpha} \begin{pmatrix}
1&t^{-1}\\0&1
\end{pmatrix} w_0 \in R_{\mathrm{u}}(B)$, we get $F_{f,\alpha}(x,t)=f \Big( \Gamma(x,t) \phi_{\alpha} \begin{pmatrix}
t&-1\\1&0
\end{pmatrix} w_0  \Big)$. The result then follows from Lemma \ref{Gamma-extends}.
\end{proof}

Recall that we view $V_I(w_{0,I}w_0\chi)$ as a subspace of $V_\Delta(\chi)$ as in Proposition \ref{prop-PiX-isom}. Note that since $\chi$ is dominant, we have $w_{0,I}w_0\chi\in C_{\GS}$ (Lemma \ref{lemma-cone-GS}). Thus by Proposition \ref{propver}, we have a map $\vfr_\rho  \colon  \Vcal_I(w_{0,I}w_0\chi) \to \Vcal_I(w_{0,I}w_0\chi^{[1]})$ (where $\rho=\rho_{I,w_{0,I}w_0\chi}$).

\begin{theorem}\label{ver-thmG}
Let $\chi\in X^*_+(T)$ be a dominant character. The map $\vfr_\chi$ extends to a morphism of vector bundles $\vfr_\chi \colon  \Vcal_{\Delta}(\chi) \to \Vcal_{\Delta}(\chi^{[1]})$ over $\GZip^{\mu}$. Furthermore, the image of $\vfr_\chi$ is $\Vcal_I(w_{0,I}w_0\chi^{[1]})$ and $\vfr_\chi$ extends the Verschiebung map of $\rho = \rho_{I,w_{0,I}w_0\chi}$. In other words, there is a commutative diagram
$$\xymatrix@1@M=6pt{
\Vcal_{\Delta}(\chi) \ar[r]^-{\vfr_\chi} & \Vcal_I(w_{0,I}w_0\chi^{[1]})  
\ar@{^{(}->}[r] 
& \Vcal_{\Delta}(\chi^{[1]}). \\
\Vcal_I(w_{0,I}w_0\chi)  
\ar@{^{(}->}[u] 
\ar[ru]_-{\vfr_{\rho}} & &
}$$
We call $\vfr_\chi$ the Verschiebung homomorphism of $\Vcal_{\Delta}(\chi)$.
\end{theorem}

\begin{proof} 
We need only show that $\vfr_\chi$ extends to a morphism of vector bundles $\vfr_\chi \colon  \Vcal_{\Delta}(\chi) \to \Vcal_{\Delta}(\chi^{[1]})$. The other assertions follow immediately from the fact that $\Pi_\chi$ extends the identity map of $V_I(w_{0,I}w_0\chi)$. As explained in \cite[Remark 5.1.3]{Imai-Koskivirta-vector-bundles}, we need to show that for all $\alpha \in \Delta^P$, all $\eta \in X^*(T)$ and $j\in \NN$ such that $j ( 1-\langle \alpha,\delta_{\alpha}  \rangle) > \langle \eta,\delta_{\alpha} \rangle$, we have 
\begin{equation}\label{cond-eq}
 \pr_{\eta} \circ
 \sum_{0 \leq i \leq \frac{j}{q}} (-1)^{i} 
  E_{-\sigma\alpha}^{(i)}  \Pi_\chi E_{-\alpha}^{(j -qi)} =0     
\end{equation}
where $\pr_{\eta}$ denotes the projection 
\[
\Hom (V_{\Delta}(\chi),V_{\Delta}(\chi^{[1]})) \simeq 
 \bigoplus_{\nu,\nu' \in X^*(T)} \Hom (V_{\Delta}(\chi)_{\nu},V_{\Delta}(\chi^{[1]})_{\nu'}) 
 \to 
 \bigoplus_{\nu \in X^*(T)} \Hom (V_{\Delta}(\chi)_{\nu},V_{\Delta}(\chi^{[1]})_{\nu+\eta}). 
\] 
Since the image of $\Pi_{\chi}$ is $V_I(w_{0,I}w_0 \chi)= V_{\Delta}(\chi)^{R_{\mathrm{u}}(P)}$, we have $E_{-\sigma\alpha}^{(i)}  \Pi_\chi=0$ except for $i=0$. Hence, we can write condition \eqref{cond-eq} as  $\pr_\eta( \Pi_\chi E_{-\alpha}^{(j)})=0$ for all $j\in \ZZ$ such that $j ( 1-\langle \alpha,\delta_{\alpha}  \rangle) > \langle \eta,\delta_{\alpha} \rangle$. The map $\Pi_\chi E_{-\alpha}^{(j)}$ takes $V_{\Delta}(\chi)_\nu$ to $V_{\Delta}(\chi^{[1]})_{q\sigma^{-1}(\nu-j\alpha)}$, so it suffices to check that $\Pi_\chi E_{-\alpha}^{(j)}$ is zero on $V_{\Delta}(\chi)_\nu$ for $\nu$ such that $\eta =q\sigma^{-1}(\nu-j\alpha)-\nu$. We have $q \langle \sigma^{-1}\nu, \delta_{ \alpha} \rangle -\langle \nu,\delta_\alpha \rangle = -\langle \nu, \alpha^\vee \rangle$, hence $\langle \eta,\delta_{\alpha} \rangle = -\langle \nu, \alpha^\vee \rangle - qj \langle \sigma^{-1}\alpha, \delta_\alpha \rangle$. 
Therefore, the condition $j ( 1-\langle \alpha,\delta_{\alpha}  \rangle) > \langle \eta,\delta_{\alpha} \rangle$ becomes $j ( 1-\langle \alpha,\delta_{\alpha}\rangle +q \langle \alpha, \delta_{\sigma \alpha}   \rangle) + \langle \nu,\alpha^\vee \rangle >0$. Using again that $q \langle \alpha, \delta_{\sigma \alpha} \rangle =\langle \alpha,\delta_\alpha \rangle -\langle \alpha, \alpha^\vee \rangle$, it is also equivalent to $j(1-\langle \alpha, \alpha^\vee \rangle)=-j >-\langle \eta, \alpha^\vee \rangle$. Hence, it suffices to show that for $j<\langle \nu, \alpha^\vee \rangle$, one has $\Pi_\chi E_{-\alpha}^{(j)}(V_{\Delta}(\chi)_\nu)=0$. Furthermore, we can assume that $\langle \nu, \alpha^\vee \rangle>0$ (since $E_{-\alpha}^{(j)}=0$ for $j<0$). Now, let $f\in V_{\Delta}(\chi)_\nu$, viewed as a function $f \colon G\to \AA^1$. Recall that $\Ker(\Pi_\chi)$ consists of functions satisfying $f(xw_{0})=0$ for all $x\in L$. By the definition of $E_{-\alpha}^{(j)}$, we must show that for all $x\in L$, the expression
\begin{equation}
    f \left( \phi_{\alpha} \left(\begin{matrix}
1&0\\t&1
\end{matrix} \right) x w_0  \right) \in k[t,t^{-1}]
\end{equation}
has $t$-valuation $\geq \langle \nu, \alpha^\vee \rangle$. Since $f$ is a eigenvector for $\nu$, this amounts to saying that 
\begin{equation}
    f \left( \alpha(t) \phi_{\alpha} \left(\begin{matrix}
1&0\\t&1
\end{matrix} \right) x w_0  \right) = f \left( \phi_{\alpha} \left(\begin{matrix}
t&0\\1&t^{-1}
\end{matrix} \right) x w_0  \right)
\end{equation}
is a polynomial in $k[t]$. Hence, the result follows from Lemma \ref{lemma-Ffalpha}.
\end{proof}

\begin{example}
Continuing Example \ref{exabvar}, the vector bundle $\Mcal=H^1_{\dR}(A/S)$ is given by $\Mcal=\Vcal(\Std_{\GL_{2n}}^\vee)=\Vcal_\Delta(\lambda_\Mcal)$ with $\lambda_\Mcal = -e_{2g}$. The Verschiebung $V\colon \Mcal \to \Mcal^{(q)}$ corresponds to the map $\vfr_{\lambda_\Mcal}$ afforded by Theorem \ref{ver-thmG}. 
\end{example}

\section{Consequences}

\subsection{Decomposition of Schubert sections}

We return to the following setting: Let $G$ be a connected, reductive group over $\FF_q$ and $\mu \colon \GG_{\mathrm{m},k}\to G_k$ a cocharacter. Let $\Zcal_\mu =(G,P,L,Q,M)$ be the attached zip zatum. Fix a frame $(B,T,z)$ where $z=\sigma(w_{0,I})w_0$ (see \ref{sec-frames}). Furthermore, we put $z_1=w_{0,I} w_0$. Hence, if $L$ is defined over $\FF_q$, then $z=z_1$.

Let $\chi\in X_{+}^*(T)$ be a dominant character. Let $h_\chi\in H^0(\Sbt,\Vcal_{\Sbt}(\chi,-w_0\chi))$ be a non-zero section as in \eqref{hchi}. Multiplication by $h_\chi$ gives a map $\Vcal_{\Sbt}(-\chi,0)\to \Vcal_{\Sbt}(0,-w_0 \chi)$. 
By Lemma \ref{lemma-cone-GS}, the character $-w_{0,I}\chi$ lies in $C_{\GS}$, in particular it is $I$-dominant. To simplify, we write $V\colonequals V_I(-w_{0,I}\chi)$ and $\rho \colonequals \rho_{I,-w_{0,I}\chi}|_B$. 
Fix a nonzero element $v_{\rm low}\in V$ in the lowest weight line of $V$. 
This choice identifies the $B$-subrepresentation $k v_{\rm low}$ with $-\chi$. We obtain a natural map $\iota_\chi$ of $B$-representations $\iota_\chi \colon {-\chi} \to \rho$ given by the inclusion $k v_{\rm low}\subset V$. This induces a morphism
\[\Vcal_{\Sbt}(\iota_\chi)\colon \Vcal_{\Sbt}(-\chi,0) \to \Vcal_{\Sbt}(\rho,\rho_0).\]
Recall that $\rho_0$ denotes the trivial representation of $B$ on $V$. Similarly, there is a projection map of $B$-representations $p_\chi\colon \rho \to -w_{0,I}\chi$ given by projection onto the highest weight line of $V$. Twisting by $z_1$, we have $p_\chi[z_1]\colon \rho[z_1] \to -w_{0}\chi$, which gives a map: 
\[\Vcal_{\Sbt}(p_\chi[z_1])\colon \Vcal_{\Sbt}(\rho_0,\rho[z_1])\to \Vcal_{\Sbt}(0,-w_{0}\chi) .\]

\begin{proposition}\label{prop-maps-Vsbt}
Let $\rho = \rho_{I,-w_{0,I}\chi}|_B$. If we put $u_1\colonequals\Vcal_{\Sbt}(\iota_\chi)$, $u_2\colonequals \Vcal_{\Sbt}(p_\chi[z_1])$ and $u_\rho$ is the map afforded by Theorem \ref{urho-theorem}, then the composition \begin{equation}
\Vcal_{\Sbt}(-\chi,0)\xrightarrow{u_1} \Vcal_{\Sbt}(\rho,\rho_0)\xrightarrow{u_\rho} \Vcal_{\Sbt}(\rho_0,\rho[z_1])\xrightarrow{u_2} \Vcal_{\Sbt}(0,-w_0 \chi)
\end{equation}
coincides with the multiplication by $h_\chi$  up to a nonzero scalar.
\end{proposition}

\begin{proof}
It suffices to show that the composition $u\colonequals u_2\circ u_\rho \circ u_1$ is nonzero. Indeed, in this case $u$ gives rise to a section in $H^0(\Sbt,\Vcal_{\Sbt}(\chi,-w_0 \chi))$. Since this space is one-dimensional (see \S\ref{subsec-Schubert-sections}), it coincides with $h_\chi$ up to a nonzero scalar. If $(\rho_1,\rho_2)$ and $(\rho'_1,\rho'_2)$ are $B\times B$-representations with underlying vector spaces $V$, $V'$ respectively, we may view a map $\Vcal_{\Sbt}(\rho_1,\rho_2) \to \Vcal_{\Sbt}(\rho'_1,\rho'_2)$ as a $k$-linear map $V\to V'$ satisfying properties (i), (ii) of section \ref{sec-glob-sec-Sbt}. First, the map $u_1$ corresponds to the inclusion $kv_{\rm low} \subset V$. Choose a nonzero vector $v_{\rm high}$ in the highest weight line of $V$. Then, $u_2$ is the  projection $p_\chi \colon  V\to k v_{\rm high}$. Finally, by construction the map $u_\rho$ corresponds to $\rho(w_{0,I})\colon V\to V$. Since $\rho(w_{0,I})v_{\rm low}\in kv_{\rm high}$, it is clear that the composition is nonzero.
\end{proof}

By \S\ref{subsec-Schubert-sections}, there is an associated Schubert section $\Ha_\chi=\psi^*(h_\chi)$, which is a section of $\Vcal_{\flag}(\chi-qw_{0,I}(\sigma^{-1}\chi))$. Multiplication by $\psi^*(h_\chi)$ gives a map $\Vcal_{\flag}(-\chi) \to \Vcal_{\flag}(-qw_{0,I}(\sigma^{-1}\chi))$. We obtain a decomposition of $\Ha_\chi$, as follows.

\begin{corollary} \label{cor-decomp-Schubert-Fq}
Assume that $P$ is defined over $\FF_q$. Write $\Vcal\colonequals \Vcal_{I}(-w_{0,I}\chi)$. There exist natural maps of vector bundles on $\GF^\mu$:
\begin{equation}\label{eq-cor-decompos-Schubert-sec-Fq}
\Vcal_{\flag}(-\chi) \xrightarrow{\xi_1} \pi^*\Vcal \xrightarrow{\pi^*\vfr_\rho} \pi^*\Vcal^{[1]} \xrightarrow{\xi_3} \Vcal_{\flag}(-qw_{0,I}(\sigma^{-1}\chi))
\end{equation}
such that the composition coincides with multiplication with $\Ha_\chi$.
\end{corollary}
\begin{proof}
Since $P$ is defined over $\FF_q$, we may apply Proposition \ref{prop-w0I-Ver} to get $\pi^*\vfr_\rho= \psi^*u_\rho$.
\end{proof}

\subsection{Primitiveness of partial Hasse invariants}

For $\lambda\in X^*(T)$, we write $L_I(\lambda)\subset V_I(\lambda)$ for the unique irreducible $L$-representation of highest weight $\lambda$ (see \ref{subsec-remind}). Denote by $\Vcal_I^L(\lambda)\subset \Vcal_I(\lambda)$ the subbundle attached to the irreducible $L$-subrepresentation $L_I(\lambda)\subset V_I(\lambda)$.

\begin{definition}
We say that $f\in H^0(\GZip^\mu, \Vcal_I(\lambda))$ is primitive if it lies in the subspace $H^0(\GZip^\mu, \Vcal^L_I(\lambda))$.
\end{definition}
Recall that there is an injective map $\ev_1 \colon H^0(\GZip^\mu, \Vcal_I(\lambda))\to V_I(\lambda)$ defined in \eqref{injection-ev1}.

\begin{lemma}
A section $f\in H^0(\GZip^\mu, \Vcal_I(\lambda))$ is primitive if and only if $\ev_1(f)\in L_I(\lambda)$.
\end{lemma}

\begin{proof}
More generally, let $V_1\subset V_2$ two $P$-representations and let $\Vcal_1\subset \Vcal_2$ be the corresponding vector bundles on $\GZip^\mu$. We have a commutative diagram
$$\xymatrix@1@M=5pt{
H^0(\GZip^\mu,\Vcal_1) \ar@{^{(}->}[d] \ar@{^{(}->}[r]^-{\ev_1} & V_1 \ar@{^{(}->}[d] \\
H^0(\GZip^\mu,\Vcal_2) \ar@{^{(}->}[r]^-{\ev_1} & V_2.
}$$
We claim that $H^0(\GZip^\mu,\Vcal_1)$ is the intersection of $H^0(\GZip^\mu,\Vcal_2) $ and $V_1$ inside $V_2$. Assume that $f\in H^0(\GZip^\mu,\Vcal_2)$ satisfies $\ev_1(f)\in V_1$. Since $V_1\cap V_2^{L_\varphi}=V_1^{L_\varphi}$, we obtain by Lemma \ref{lem-Umu-sections} that $f\in H^0(\Ucal_\mu,\Vcal_1)$. The claim follows by density of $\Ucal_\mu$.
\end{proof}

We study the primitiveness of (flag) partial Hasse invariants. We do not assume that $P$ is defined over $\FF_q$. For an $L$-representation $(V,\rho)$, define $(V^{(1)},\rho^{(1)})$ as follows. Let $L_1$ be the split form of $L$ over $\FF_q$ and let $\varphi_1\colon L_k\to L_k$ be the Frobenius homomorphism of $L_1$. We set $V^{(1)}=V$ and $\rho^{(1)}=\rho \circ \varphi_1$. Let $\chi\in X_{+,I}^*(T)$. Then $V_I(\chi)^{(1)} \subset V_I(q\chi)$ is the sub-$L$-representation whose underlying space is the image of the (non-linear) map $V_I(\chi)\to V_I(q\chi)$, $f\mapsto f^q$. 
For $\lambda\colonequals \chi-qw_{0,I}(\sigma^{-1}\chi)$, we obtain maps
\begin{equation}\label{eq-nat-map}
V_I(\chi)\otimes V_I(-w_{0,I}(\sigma^{-1}\chi))^{(1)}\to V_I(\chi)\otimes V_I(-qw_{0,I}(\sigma^{-1}\chi)) \to V_I(\lambda)
\end{equation}
which are morphisms of $L$-representations.

\begin{proposition}\label{propinimage}
Let $\chi\in X_{+}^*(T)$ and set $\lambda=\chi-qw_{0,I}(\sigma^{-1}\chi)$. The Schubert line $V_{I,\Sbt}(\lambda)$ is contained in the image of $V_I(\chi)\otimes V_I(-w_{0,I}(\sigma^{-1}\chi))^{(1)}$ by the map \eqref{eq-nat-map}.
\end{proposition}

\begin{proof}
Let $\chi\in X^*_+(T)$ and write $\rho$ for a $G$-representation $V_{\Delta}(\chi)$. Let $p_\chi \in V_{\Delta}(\chi)^{\vee}$ be a $B$-eigenvector in $V_{\Delta}(\chi)^\vee$ for the weight $-\chi$ afforded by Lemma \ref{lemma-proj}. It satisfies $p_{\chi}(\rho(b) x) = \chi(b)p_{\chi}(x)$ for all $x\in G$ and all $b\in B$. Let $v_{\rm low}\in V_{\Delta}(\chi)$ be a nonzero vector in the lowest weight line. Consider the function $\Delta_\chi \colon  G\to \AA^1$ defined by 
\begin{equation}
\Delta_\chi(x)=p_{\chi}(\rho(x) v_{\rm low}), \quad x\in G.
\end{equation}
We have for all $x\in G$ and all $(b,b')\in B\times B$:
\begin{equation}
    \Delta_\chi(bxb'^{-1}) =  p_{\chi}(\rho(bxb'^{-1}) v_{\rm low}) = \chi(b)p_{\chi}(\rho(x) ((w_0\chi)(b')^{-1}v_{\rm low})) =\chi(b)(w_0\chi)(b')^{-1} \Delta_\chi(x).
\end{equation}
It follows that $\Delta_\chi$ is a $B\times B$-eigenfunction and it lies in the space $H^0(\Sbt, \Vcal_{\Sbt}(\chi,-w_0\chi))$. 
Put again $\lambda=\chi-qw_{0,I}\sigma^{-1}(\chi)$. Then $f_\lambda \colonequals \psi^*(\Delta_\chi)$ is a Schubert section of weight $\lambda$. Explicitly, $f_\lambda$ is a function $G_k \to k$ satisfying $f_\lambda(x)= \Delta_\chi(xz)$ for all $x\in G_k$. Let $f_{\lambda,I}$ be the element of 
$H^0(\GZip^\mu,\Vcal_I(\lambda))$ corresponding to $f_{\lambda}$ under \eqref{ident-H0-GF}. Viewing $f_{\lambda,I}$ in $V_I(\lambda)$ via $\ev_1$, it is given as follows: For all $x\in L$, one has
\begin{equation}
\ev_1(f_{\lambda,I})(x)=f_\lambda(x^{-1}\varphi(x)) = p_{\chi}(\rho(x^{-1} \varphi(x) z)v_{\rm low}).
\end{equation}
Let $\langle -,-\rangle \colon V_{\Delta}(\chi)\times V_{\Delta}(\chi)^{\vee} \to k$ be the natural pairing. We can write for all $x\in L$:
\begin{align}
\ev_1(f_{\lambda,I}) (x) & = \langle \rho(x^{-1}\varphi(x)z) v_{\rm low}, p_{\chi} \rangle \\
& = \langle \rho(\varphi(x)z) v_{\rm low}, \rho^{\vee}(x) p_{\chi} \rangle. 
\end{align}
Choose a basis $\Bcal \colonequals (e_1, \dots ,e_r)$ of $V_{\Delta}(\chi)$ and let $\Bcal^\vee \colonequals (e^\vee_1, \dots ,e^\vee_r)$ denote the dual basis. For all $x\in G$, we define $A(x) \colonequals \rho(xz)v_{\rm low}$ and write $A(x)=(A_1(x), \dots ,A_r(x))$ for the coordinates of $A(x)$ with respect to $\Bcal$. Similarly $C(x)\colonequals \rho^{\vee}(x) p_{\chi}$ and $C(x)=(C_1(x),\dots ,C_r(x))$ with respect to $\Bcal^\vee$. Hence we have for all $x\in L$, 
$\ev_1(f_{\lambda,I}) (x)=\sum_{i=1}^r A_i(\varphi(x)) C_i(x)$. Since $G$ is defined over $\FF_q$, we may also consider the functions ${}^{\sigma^{-1}} A_i  \colon  G\to \AA^1$. We obtain
\begin{equation}\label{decomp-Sbt-equ-V2}
    \ev_1(f_{\lambda,I}) (x)=\sum_{i=1}^r {}^{\sigma^{-1}}A_i(x)^q \  C_i(x).
\end{equation}
Note that for all $x\in L$ and all $b\in B_L$, we have:
\begin{align}
    A(xb) &= \chi(\sigma(w_{0,I})b \sigma(w_{0,I})^{-1}) A(x), \label{eq-A-V} \\
    C(xb) &= \chi(b)^{-1} C(x).  \label{eq-C-V}
\end{align}
Hence for $1\leq i \leq r$, the functions $x\mapsto C_i(x)$ lie in $V_I(\chi)$ and the functions $x\mapsto {}^{\sigma^{-1}} A_i(x)$ lie in $V_I(-w_{0,I}\sigma^{-1}\chi)$. This shows that $\ev_1(f_{\lambda,I})$ lies in the image of the map \eqref{eq-nat-map}. 
\end{proof}

Let $\alpha\in \Delta$. Consider the following condition on a character $\chi_{\alpha}$:
\begin{condition}\label{condition-funda} 
\ \begin{definitionlist}
\item One has $0 < \langle \chi_{\alpha},\alpha^\vee \rangle<q$ 
and $\chi_{\alpha}$ is orthogonal to $\beta^{\vee}$ for all $\beta\in \Delta \setminus \{\alpha\}$. 
\item One has $L_I(\chi_{\alpha})=V_I(\chi_{\alpha})$ and $L_I(-w_{0,I}(\sigma^{-1}\chi_{\alpha}))=V_I(-w_{0,I}(\sigma^{-1}\chi_{\alpha}))$.
\end{definitionlist}
\end{condition}

\begin{theorem}\label{thm-LIlambda}
Let $\alpha\in \Delta$ and assume that $\chi_{\alpha}\in X^*(T)$ satisfies Condition \ref{condition-funda}. Denote by $f_\alpha\colonequals \psi^*(h_{\chi_{\alpha}})$ the corresponding Schubert section, of weight $\lambda_\alpha \colonequals \chi_{\alpha}-qw_{0,I}(\sigma^{-1}\chi_{\alpha})$. We have the following properties:
\begin{assertionlist}
\item $f_\alpha$ is a flag partial Hasse invariant for the flag stratum $\Ccal_{s_\alpha w_0}$.
\item $f_\alpha$ is a primitive automorphic form on $\GZip^\mu$.
\end{assertionlist}
\end{theorem}
\begin{proof}
The first statement follows from (a) of Condition \ref{condition-funda}. We now show (2). View $f_\alpha$ as an element of $V_I(\lambda_\alpha)$. By part (b), we deduce from Proposition \ref{propinimage} that $f_\alpha$ lies in the image of the map $L_I(\chi_{\alpha})\otimes L_I(-w_{0,I}(\sigma^{-1}\chi_{\alpha}))^{(1)} \to V_I(\lambda_\alpha)$ given by \eqref{eq-nat-map}. We put 
\[
 X_{1,I}^*(T) = \{ \chi \in X^* (T) \mid \textrm{$0 \leq \langle \chi,\alpha^\vee \rangle<q$ for all $\alpha \in I$} \}
\]
(\cf \cite[II, \S3.15]{jantzen-representations}). 
By part (a) of Condition \ref{condition-funda}, $\chi_{\alpha}$ and $-w_{0,I}(\sigma^{-1}\chi_{\alpha})$ are in $X_{1,I}^*(T)$. By Theorem \ref{thm-L-decomp}, we have $L_I(\chi_{\alpha})\otimes L_I(-w_{0,I}(\sigma^{-1}\chi_{\alpha}))^{(1)}=L_I(\lambda_{\alpha})$. The result follows.
\end{proof}

\begin{rmk}
Let $\chi \in X^*(T)$. 
In \cite[Theorem 1.1 and \S2]{Garibaldi-Guralnick-nakano-globally-irreducible}, 
it is studied when $L_I(\chi)=V_I(\chi)$ holds for all $p$. 
On the other hand, 
$L_I(\chi)=V_I(\chi)$ holds for sufficiently large $p$ by 
\cite[II, 5.6 Corollary]{jantzen-representations}. 
\end{rmk}

\section{Examples}\label{sec-examples}

\subsection{The symplectic case}
The Siegel Shimura variety $\Ascr_n$ of rank $n$ is attached to the reductive group $\mathbf{G}=\GSp(\mathbf{V}_0,\mathbf{\psi})$, where $(\mathbf{V}_0,\mathbf{\psi})$ is a symplectic space of rank $n$ over $\QQ$. The special fiber $\Ascr_{n,\FF_p}$ at a place of good reduction admits a smooth surjective map to $\GZip^\mu$ where $G$ is the special fiber of a reductive $\ZZ_p$-model of $\mathbf{G}_{\QQ_p}$, and $\mu$ is the usual cocharacter attached to the Siegel-type Shimura datum. We may define the flag space of $\Ascr_{n,\FF_p}$ as the fiber product:
\begin{equation}\label{flag-eq}
\xymatrix@1@M=5pt{
\Flag(\Ascr_{n,\FF_p}) \ar[r] \ar[d] & \GF^\mu \ar[d] \\
\Ascr_{n,\FF_p} \ar[r]^-{\zeta} & \GZip^\mu.
}    
\end{equation}
We obtain by pullback a stratification on $\Flag(\Ascr_{n,\FF_p})$ as well as (flag) partial Hasse invariants for each of the codimension one strata of $\Flag(\Ascr_{n,\FF_p})$. 
In this section, we compute these partial Hasse invariants explicitly in the case $G=\Sp(V_0,\psi)$ (where $(V_0,\psi)$ is a symplectic space over $\FF_p$). The case of $G=\GSp(V_0,\psi)$ is completely similar. Furthermore, we give the modular interpretation of these sections in section \ref{moduli-sp}.

\subsubsection{\texorpdfstring{The group $G$}{}}\label{thegroup}
We consider the case when $G$ is the reductive $\FF_q$-group $\Sp(V_0,\psi)$, where $(V_0,\psi)$ is a non-degenerate symplectic space over $\FF_q$ of dimension $2n$, for some integer $n\geq 1$. After choosing an appropriate basis $\Bcal$ for $V_0$, we assume that $\psi$ is given by the matrix
\begin{equation}
\begin{pmatrix}
& -J \\
J&\end{pmatrix}
\quad \textrm{ where } \quad
J\colonequals \begin{psmallmatrix}
&&1 \\
&\iddots& \\
1&&
\end{psmallmatrix}. 
\end{equation} 
Define $G$ as follows:
\begin{equation}\label{group}
G(R) = \{f\in \GL_{\FF_{q}}(V_0\otimes_{\FF} R) \mid  \psi_R(f(x),f(y))=\psi_R(x,y), \ \forall x,y\in V_0\otimes_{\FF_q} R \}
\end{equation}
for all $\FF_q$-algebras $R$. Identify $V_0=k^{2n}$ and view $G$ as a subgroup of $\GL_{2n,\FF_q}$. Fix the $\FF_q$-split maximal torus $T$ given by diagonal matrices in $G$, i.e.
\begin{equation}
T(R)\colonequals \{ \diag_{2n}(x_1,\ldots ,x_n,x^{-1}_n,\ldots,x^{-1}_1) \mid x_1, \ldots ,x_n\in R^\times \}.
\end{equation}
Define $B$ as the Borel subgroup of $G$ consisting of the lower-triangular matrices in $G$. For a tuple $(a_1,\dots,a_n)\in \ZZ^n$, define a character of $T$ by mapping $\diag_{2n}(x_1, \dots ,x_n,x^{-1}_n, \dots ,x^{-1}_1)$ to $x_1^{a_1} \cdots x_n^{a_n}$. From this, we obtain an identification $X^*(T) = \ZZ^n$. Denoting by $(e_1,\dots ,e_n)$ the standard basis of $\ZZ^n$, the $T$-roots of $G$ and the $B$-positive roots are respectively
\begin{align}
\Phi&\colonequals \{\pm e_i \pm e_j \mid 1\leq i \neq j \leq n\} \cup \{ \pm 2e_i \mid 1\leq i \leq n \}, \\
\Phi_+&\colonequals \{e_i \pm e_j \mid 1\leq i< j \leq n\} \cup \{ 2e_i \mid 1\leq i \leq n\} 
\end{align}
and the $B$-simple roots are $\Delta\colonequals \{\alpha_1,\dots , \alpha_{n-1},\beta\}$ where 
\begin{align*}
\alpha_i&\colonequals e_{i+1}-e_i \textrm{ for } i=1,...,n-1 ,\\ \beta&\colonequals 2e_n.
\end{align*}
The Weyl group $W\colonequals W(G,T)$ can be identified with the group of permutations $\sigma \in \Sfr_{2n}$ satisfying $\sigma(i)+\sigma(2n+1-i)=2n+1$ for all $1\leq i \leq 2n$.

Define a cocharacter $\mu \colon \GG_{\mathrm{m},\FF_q}\to G$ by $z\mapsto \diag(zI_n,z^{-1}I_n)$. Write $\Zcal\colonequals (G,P,L,Q,M,\varphi)$ for the associated zip datum (since $\mu$ is defined over $\FF_q$, we have $M=L$). Concretely, if we denote by $(u_i)_{i=1}^{2n}$ the canonical basis of $k^{2n}$, then $P$ is the stabilizer of $V_{0,P}=\Span_k(u_{n+1},...,u_{2n})$ and $Q$ is the stabilizer of $V_{0,Q}=\Span_k(u_{1},...,u_{n})$. The intersection $L\colonequals P\cap Q$ is the 
common Levi subgroup, which is isomorphic to $\GL_{n,\FF_q}$.

\subsubsection{Partial Hasse invariants}

There is a unique stratum of codimension one in $\GZip^\mu$. It is the vanishing locus of the (ordinary) Hasse invariant $\Ha_\mu\in H^0(\GZip^\mu,\omega^{q-1})$. 
Concretely, this section is given by
\begin{equation}
\Ha_\mu \colon \left(\begin{matrix}
A&*\\
*&*
\end{matrix} \right) \mapsto \det(A).
\end{equation}
The (flag) partial Hasse invariants are given as follows. For $1\leq d\leq n$, define a function $\Delta_d \colon \GL_n\to \AA^1$ by
\begin{equation}\label{eq-hd}
\Delta_d(A)=\left| \begin{matrix}
a_{1,n+1-d} & a_{1,n+2-d} & \cdots & a_{1,n}\\
a_{2,n+1-d} & a_{2,n+2-d} & \cdots & a_{2,n} \\
\vdots & \vdots & & \vdots \\
a_{d,n+1-d} & a_{d,n+2-d} & \cdots & a_{d,n}
\end{matrix} \right| \qquad \textrm{for } A=(a_{i,j})_{1\leq i,j\leq n}.
\end{equation}
Define $\Ha_d \colon G\to \AA^1$ by 
\[
 \left(\begin{matrix}
A&*\\
*&*
\end{matrix} \right) \mapsto \Delta_d(A). 
\] 
Then $\Ha_d$ is a section in $H^0(\GF^\mu, \Vcal_{\flag}(\lambda_d))$ for the weight
\begin{equation}
\lambda_d = (\underbrace{1, \dots , 1}_{\text{$d$ times}}, \underbrace{0, \dots , 0}_{\text{$n-d$ times}} ) - (\underbrace{0, \dots , 0}_{\text{$n-d$ times}}, 
\underbrace{q, \dots , q}_{\text{$d$ times}} ).
\end{equation}
The family $\{\Ha_d\}_{1\leq d \leq n}$ is a complete set of flag partial Hasse invariants. Note that $\Ha_n$ coincides with the classical Hasse invariant $\Ha_\mu$. This illustrates Lemma \ref{Lemma-zipflag}, which states that any zip partial Hasse invariant becomes a flag partial Hasse invariant when pulled back to $\GF^\mu$ (since $P$ is defined over $\FF_q$ in this example).

\subsubsection{Primitiveness}\label{sec-prim-sp}

Write $\chi_d=(1,\dots ,1, 0 \dots ,0)$, where $1$ appears $d$ times and $0$ appears $n-d$ times. Then we have $\lambda=\chi_d-qw_{0,I}\chi_d$. 
One can show easily that $\chi_d$ satisfies Condition \ref{condition-funda}. Hence, by Theorem \ref{thm-LIlambda}, the section $\Ha_d$ is a primitive automorphic form on $\GZip^\mu$, in the sense that it corresponds to an element of $L_I(\lambda_d)\subset V_I(\lambda_d)$. We check this in the case $d=1$ (the computation for larger $d$ is slightly tedious). The element $\ev_1(\Ha_d) \in V_I(\lambda_d)$ corresponding to $\Ha_d$ via \eqref{injection-ev1} is the function
\begin{equation}
 \ev_1(\Ha_d)  \colon  L\to \AA^1 ,\quad A\mapsto \Delta_d(A^{-1} \varphi(A)).
\end{equation}
When $d=1$, we can write this function as
\begin{equation}\label{eq-evHa}
  \ev_1(\Ha_1)(A) = \frac{1}{\delta} \sum_{i=1}^n (-1)^{i+1} \delta_{i} a_{i,n}^q, \ \textrm{ where } \ \delta_{i}=\left| 
\begin{matrix}
a_{1,2} & \cdots & a_{1,n} 
\\
\vdots & & \vdots 
\\
a_{i-1,2} & \cdots & a_{i-1,n} 
\\
a_{i+1,2} & \cdots & a_{i+1,n} 
\\
\vdots & & \vdots 
\\
a_{n,2} & \cdots & a_{n,n} 
\end{matrix}
\right|
\end{equation}
and $\delta$ is the determinant $\GL_n\to \GG_{\mathrm{m}}$. One shows easily that the function $\delta_i$ lies in $L_I(\chi_1)$ and that $a_{i,n}$ lies in $L_I(-w_{0,I}\chi_1)$. Thus, $\ev_1(\Ha_1)$ lies in $L_I(\chi_1)\otimes L_I(-w_{0,I}\chi_1^{[1]}) = L_I(\lambda_1)$. For $d>1$, the function $\ev_1(\Ha_d)$ admits a similar expansion in terms of minors of the matrix $A$. Using the formula given in equation \eqref{decomp-Sbt-equ-V2} in the proof of Proposition \ref{propinimage}, there is an expansion with $\dim(V_{\Delta}(\chi))$ summands (for $\chi=\chi_d$). Actually, when $P$ is defined over $\FF_q$, one can modify the proof of Proposition \ref{propinimage} by using the representation $V_I(\chi)$ instead of $V_{\Delta}(\chi)$. Hence, in this case we expect an expansion with $\dim(V_I(\chi_d))=\binom{n+d-1}{d}$ summands (for $d=1$, we indeed have $n$ terms in \eqref{eq-evHa}).

We give a counter-example showing that in general, Schubert sections do not lie in $L_I(\lambda)$. Take $G=\Sp(4)$ and consider $\Ha_1^m$ with $m\colonequals q^2-q+1$. Set $f_1=\ev_1(\Ha_1)\in L_I(1,-q)$. We claim that the section $f_1^{m}\in V_I(m,-qm)$ does not lie in $L_I(m,-qm)$. We can write $m(1,-q)=m(1,1)+(0,-1)+q^3(0,-1)$. By Theorem \ref{thm-L-decomp}, we see that 
\[L_I(m,-qm)\simeq \delta^{-m}\otimes L_I(0,-1)\otimes L_I(0,-1)^{[3]}\]
where $\delta $ is the determinant $\GL_2\to \GG_{\mathrm{m}}$. By equation \eqref{eq-evHa}, we can write $f_1=\delta^{-1}(x^q y-xy^q)$ for a certain basis $(x,y)$ of $L_I(0,-1)$. Then $L_I(m,-qm)$ has dimension 4, with basis \[
 \delta^{-m}x^{q^3+1},\  \delta^{-m}x^{q^3}y,\   \delta^{-m}x y^{q^3},\  \delta^{-m}y^{q^3+1}. 
\] 
It is easy to see that the expansion of $f_1^{m}=\delta^{-m}(x^q y-xy^q)^m$ involves also other monomials, hence it does not lie in $L_I(m,-qm)$.

\subsubsection{Modular interpretation}\label{moduli-sp}

We now explain the modular interpretation of partial Hasse invariants. We refer to \cite[\S 8]{Pink-Wedhorn-Ziegler-F-Zips-additional-structure} for the modular interpretation of symplectic zips. Let $S$ be an $\FF_q$-scheme and $n\geq 1$ an integer. Define a symplectic zip of rank $n$ as a tuple $\underline{\Mcal}=(\Mcal, F, V, \Omega, \langle -, - \rangle, \iota)$ satisfying the following:
\begin{equivlist}
    \item $\Mcal$ is a locally free $\Ocal_S$-module of rank $2n$ and $\Omega\subset \Mcal$ is a locally free $\Ocal_S$-submodule of rank $n$ which is Zariski locally a direct factor of $\Mcal$. 
    \item $\langle - , - \rangle  \colon \Mcal\times \Mcal \to \Ocal_S$ is a perfect $\Ocal_S$-bilinear pairing.
    \item $F \colon \Mcal^{(q)}\to \Mcal$ and $V \colon \Mcal\to \Mcal^{(q)}$ are maps of vector bundles.
    \item We have $\Im(F)=\Ker(V)$ and $\Im(V)=\Ker(F)=\Omega^{(q)}$.
    \item We have $\langle Fx, y \rangle = \langle x,Vy \rangle$ for all $x,y\in \Mcal$.
\end{equivlist}
Note that \cite[\S 8]{Pink-Wedhorn-Ziegler-F-Zips-additional-structure} uses a slightly different formulation. Instead of maps $(F,V)$, the authors define an F-zip using two filtrations of $\Mcal$ and isomorphisms between the graded pieces. In the above description, the corresponding filtrations are given by $\Omega$ and $\Im(F)$. The isomorphisms between the graded pieces are precisely induced by $F$ and $V$.

\begin{example}\label{ex-var-ab}
Let $(A,\chi)$ be an abelian scheme over $S$ together with a principal polarization. Then we obtain a symplectic F-zip by defining $\Mcal=H^1_{\dR}(A/S)$ and $\Omega$ as the Hodge filtration of $\Mcal$. The polarization induces a perfect pairing on $\Mcal$, and the Frobenius and Verschiebung maps of $A$ give rise to similar maps on $\Mcal$.
\end{example}

There is an obvious notion of morphisms of symplectic $F$-zips over $S$. The symplectic zips form a stack over $\FF_q$, which is isomorphic to $\GZip^\mu$ for $G=\Sp(2n)_{\FF_q}$ and $\mu \colon z\mapsto \diag(z I_n,z^{-1}I_n)$. Via this isomorphism, the inclusion $\Omega\subset \Mcal$ between vector bundles corresponds to the inclusion of $P$-representations $V_{0,P} \subset V_0$. Note that the highest weights of $V_{0,P}$ and $V_0$ are respectively $-e_n \in C_{\GS}$ and $e_1\in X_{+}^*(T)$. This is consistent with the fact that $\Omega$ and $\Mcal$ admit Verschiebung maps $\Omega\subset \Mcal \xrightarrow{V} \Omega^{(q)}$ and $\Mcal \xrightarrow{V} \Mcal^{(q)}$, as predicted by Lemma \ref{GS-Deltapos} and Theorem \ref{ver-thmG}.

We define a flagged symplectic zip of rank $n$ over $S$ as a pair $(\underline{\Mcal},\Fcal_\bullet)$ where $\underline{\Mcal}=(\Mcal, F, V, \Omega, \langle -, - \rangle)$ is a symplectic zip of rank $n$ over $S$ and $0\subset \Fcal_1\subset \dots \subset \Fcal_n=\Omega$ is a filtration by locally free, locally direct factor $\Ocal_S$-submodules such that $\rank(\Fcal_i)=i$. This flag extends uniquely to a symplectic flag of $\Mcal$ by defining $\Fcal_{i}=\Fcal_{i-n}^{\perp}$ for all $n<i\leq 2n$. 
For all $1\leq i \leq n$, we define a line bundle $\Lcal_i=\Fcal_i/\Fcal_{i-1}$ (with the convention $\Fcal_0=0$). For a tuple $\underline{a}=(a_1,\dots , a_n)\in \ZZ^n$, we define a line bundle
\begin{equation}
    \Lcal(\underline{a})=\bigotimes_{i=1}^n \Lcal_i^{\otimes -a_i}.
\end{equation}
The minus signs in the above formula account 
for our choice of $B$ as the lower-triangular Borel subgroup. Indeed, $B$ is the stabilizer of the flag $0\subset \Span_k(e_{2n})\subset  \Span_k(e_{2n-1},e_{2n})\subset \dots \subset \Span_k(e_{n+1},\dots,e_{2n})\subset V_{0,P}$, and $T$ acts on the graded pieces of this flag by $t_1^{-1},\dots ,t_n^{-1}$ in this order.

The stack of flagged symplectic zips is isomorphic to the stack $\GF^\mu$. We now explain the modular interpretation of partial Hasse invariants. Let $1\leq d \leq n$. Consider the composition
\begin{equation}\label{fi-eq}
\xymatrix@M=5pt{
    \Fcal_d \ar[r] & \Omega \ar[r]^{V} & \Omega^{(q)} \ar[r] & \Omega^{(q)}/\Fcal^{(q)}_{n-d} . 
}
\end{equation}
The first map is the inclusion $\Fcal_d\subset \Omega$, and the third map is the natural projection. Denote by $f_d$ the composition of the maps \eqref{fi-eq}. We have 
\begin{equation}
    \det(\Fcal_d)=\Lcal_1\otimes \dots \otimes \Lcal_d, \qquad  \det \left(\Omega^{(q)}/\Fcal^{(q)}_{n-d}\right) = \Lcal^q_{n-d+1}\otimes \dots \otimes \Lcal^q_n.
\end{equation}
Therefore, $f_d$ induces a map $\wedge^d(f_d) \colon \Lcal_1\otimes \dots \otimes \Lcal_d \to \Lcal^q_{n-d+1}\otimes \dots \otimes \Lcal^q_n$. Hence, we obtain a section $\Ha_d=\wedge^d(f_d)\in H^0(S,\Lcal(\lambda_d))$ which corresponds to the flag partial Hasse invariant given by \eqref{eq-hd}.

\subsection{The unitary inert case}
This example arises in the theory of Shimura varieties attached to general unitary groups. More precisely, let $\mathbf{E}/\QQ$ be a quadratic totally imaginary extension and $(V_0,\psi)$ a hermitian space over $\mathbf{E}$ of dimension $n$. Then there is a Shimura variety of PEL-type attached to the group $\mathbf{G}=\GU(V_0,\psi)$. If we write $n=\dim_E(V_0)$, then the attached Shimura variety parametrizes abelian varieties of dimension $n$ with a polarization, an action of $\Ocal_{\mathbf{E}}$ and a level structure. Let $p$ be a prime of good reduction, and let $X$ be the special fiber of the Kisin--Vasiu (canonical) integral model of the Shimura variety. By \eqref{zeta-Shimura}, we have a smooth, surjective morphism $\zeta \colon X\to \GZip^\mu$, where $G$ is the special fiber of a reductive $\ZZ_p$-model of the group $\mathbf{G}_{\QQ_p}$. If $p$ is inert in $\mathbf{E}$, then $G$ is a unitary group over $\FF_p$. If $p$ is split in $\mathbf{E}$, then $G\simeq \GL_{n,\FF_p}\times \GG_{m,\FF_p}$. 
Define again the flag space $\Flag(X)$ of $X$ as the fiber product of $X$ and $\GF^\mu$ over $\GZip^\mu$, similarly to \eqref{flag-eq}. We explain the modular interpretation of partial Hasse invariants on $\Flag(X)$. 
We will first assume that $p$ is an inert prime in $\mathbf{E}$. To simplify, we consider the case of a unitary group $G=\U(V_0,\psi)$ (the case of $G=\GU(V_0,\psi)$ is very similar).

\subsubsection{Group theory}\label{sec-gp-inert}

Let $(V_0,\psi)$ be an $n$-dimensional $\FF_{q^2}$-vector space endowed with a non-degenerate hermitian form $\psi \colon V_0\times V_0\to \FF_{q^2}$ (in the context of Shimura varieties, we always take $q=p$). Choose a basis $\Bcal$ of $V_0$ where $\psi$ is given by the matrix 
\[
 J= \begin{pmatrix}
&&1\\& \iddots &\\1&&\end{pmatrix}. 
\]
Define a reductive group $G$ by
\begin{equation}
G(R) = \{f\in \GL_{\FF_{q^2}}(V_0\otimes_{\FF_q} R) \mid  \psi_R(f(x),f(y))=\psi_R(x,y), \ \forall x,y\in V_0\otimes_{\FF_q} R \}
\end{equation}
for any $\FF_q$-algebra $R$. There is an isomorphism $G_{\FF_{q^2}}\simeq \GL(V_0)$. It is induced by the $\FF_{q^2}$-algebra isomorphism $\FF_{q^2}\otimes_{\FF_q} R\to R\times R$, $a\otimes x\mapsto (ax,\sigma(a)x)$ (where $\Gal(\FF_{q^2}/\FF_q)=\{\id, \sigma\}$). The action of $\sigma$ on the set $\GL_n(k)$ is given by $\sigma\cdot A = J \sigma({}^t \!A)^{-1}J$. Let $T$ denote the maximal diagonal torus and $B$ the lower-triangular Borel subgroup of $G_k$. By our choice of the basis $\Bcal$, the groups $B$ and $T$ are defined over $\FF_q$. Identify $X^*(T)=\ZZ^n$ such that $(a_1, \dots ,a_n)\in \ZZ^n$ corresponds to the character $\diag(x_1,\dots ,x_n)\mapsto \prod_{i=1}^n x_i^{a_i}$. The simple roots are given by $\Delta=\{e_i-e_{i+1} \mid 1\leq i \leq n-1 \}$, where $(e_1, \dots ,e_n)$ is the canonical basis of $\ZZ^n$. 

Choose non-negative integers $(r,s)$ such that $n=r+s$. We will assume that $r\geq s$. Define a cocharacter $\mu  \colon  \GG_{\mathrm{m},k}\to G_{k}$ by $x\mapsto \diag(xI_r,I_s)$ via the identification $G_{k}\simeq \GL_{n,k}$. Let $\Zcal_{\mu}=(G,P,L,Q,M,\varphi)$ be the associated zip datum. Concretely, if we denote by $(u_i)_{i=1}^{n}$ the canonical basis of $k^{n}$, then $P$ is the stabilizer of $V_{0,P}\colonequals \Span_k(u_{r+1},...,u_{n})$. Note that $P$ is not defined over $\FF_q$ unless $r=s$. We may also identify $L=\GL_{r}\times \GL_{s}$.  One has $\Delta^P=\{\alpha\}$ with $\alpha=e_r-e_{r+1}$. The element $z=\sigma(w_{0,I})w_0$ is given by the matrix
\begin{equation}
    z=\left( 
    \begin{matrix}
    0 & I_s \\ I_r & 0    
    \end{matrix} \right).
\end{equation}
Consider again the function $\Delta_d \colon G_k\to \AA^1$ defined in \eqref{eq-hd}. It is clear that $\Delta_d$ is a $B\times B$-eigenfunction on $G_k$. It is a section of $\Vcal_{\Sbt}(\chi_d,-w_0\chi_d)$ for the character $\chi_d=(1, \dots, 1, 0, \dots,0)$, where $1$ appears $d$ times and $0$ appears $n-d$ times. Specifically, $\Delta_d$ is the function $h_{\chi}$ for $\chi=\chi_d$, as defined in \eqref{hchi}. The function $G_k\to \AA^1_k$, $g\mapsto \Delta_d(gz)$ corresponds to the pull-back $\Ha_d=\psi^*(\Delta_d)$. The sections $\{\Ha_d\}_{1\leq d \leq n-1}$ form a complete set of partial Hasse invariants for $\GF^\mu$. If $d\leq s$, the section $\Ha_d$ lies in $H^0(\GF^\mu,\Vcal_{\flag}(\lambda_d))$, where 
\begin{equation}
    \lambda_d =\chi_d-q w_{0,I}(\sigma^{-1} \chi_d) = (\underbrace{1, \dots ,1}_{\textrm{$d$ times}},\underbrace{0, \dots ,0}_{\textrm{$n-d$ times}}) + (\underbrace{0, \dots , 0}_{\textrm{$r$ times}}, \underbrace{q, \dots ,q}_{\textrm{$d$ times}}, \underbrace{0, \dots ,0}_{\textrm{$s-d$ times}}). 
\end{equation}
Similarly, in the case $s<d\leq n-1$, the weight of $\Ha_d$ is 
\begin{equation}
    \lambda_d =\chi_d-q w_{0,I}(\sigma^{-1} \chi_d) = (\underbrace{1, \dots ,1}_{\textrm{$d$ times}},\underbrace{0, \dots ,0}_{\textrm{$n-d$ times}}) + (\underbrace{q, \dots , q}_{\textrm{$d-s$ times}}, \underbrace{0, \dots ,0}_{\textrm{$n-d$ times}}, \underbrace{q, \dots ,q}_{\textrm{$s$ times}}). 
\end{equation}

Note also that the determinant function $\det  \colon  G_k\to \GG_{\mathrm{m}}$ defines a non-vanishing section of $H^0(\GZip^\mu,\Vcal_I(\lambda_{\det}))$ where $\lambda_{\det}=(q+1, \dots , q+1)$. 
In particular, the line bundles $\Vcal_I(\lambda_{\det})$ and $\Vcal_{\flag}(\lambda_{\det})$ are trivial. Hence, the line bundles $\{\Vcal_{\flag}(\lambda)\}_\lambda$ 
on $\GF^\mu$ are parametrized up to isomorphism by the group $\ZZ^n/\ZZ \lambda_{\det}$.

\subsubsection{Modular interpretation}\label{sec-mod-inert}

We explain the modular interpretation of the sections $\Ha_d$. We first define the notion of unitary zip of type $(r,s)$. Let $S$ be an $\FF_q$-scheme and let $\underline{\Mcal}=(\Mcal, F, V, \Omega, \langle -, - \rangle)$ be a symplectic zip over $S$ of rank $n$. Consider a ring homomorphism $\iota  \colon  \FF_{q^2}\to \End_{\Ocal_S}(\underline{\Mcal})$. 
(Here, $\End_{\Ocal_S}(\underline{\Mcal})$ denotes the ring of endomorphisms of the symplectic zip $\underline{\Mcal}$). Let $k$ be an algebraic closure of $\FF_q$ and write $\iota_0, \iota_1$ for the two embeddings $\FF_{q^2}\to k$. Then $\Mcal_k$ splits naturally as
\begin{equation}
    \Mcal_k = \Mcal_0 \oplus \Mcal_1
\end{equation}
where $\alpha \in \FF_{q^2}$ acts on $\Mcal_i$ (for $i=0,1$) by multiplication with $\iota_i(\alpha)$. One sees easily that $\Mcal_0$, $\Mcal_1$ are totally isotropic. The Frobenius and Verschiebung interchange the components. More precisely, they induce maps $V \colon \Mcal_0\to \Mcal^{(q)}_1$, $V \colon \Mcal_1\to \Mcal^{(q)}_0$, $F \colon \Mcal^{(q)}_0\to \Mcal_1$ and $F \colon \Mcal^{(q)}_1\to \Mcal_0$. Similarly, we have a decomposition $\Omega_k = \Omega_0 \oplus \Omega_1$ since by definition $\Omega$ is stable by the action of $\FF_{q^2}$. 

\begin{definition}\label{def-unit-zip}
A unitary zip of type $(r,s)$ is a tuple $(\Mcal, F, V, \Omega, \langle -, - \rangle, \iota)$ satisfying the following: 
\begin{assertionlist}
\item $\underline{\Mcal}=(\Mcal, F, V, \Omega, \langle -, - \rangle)$ is a symplectic zip over $S$ of rank $n=r+s$.
    \item $\iota  \colon  \FF_{q^2}\to \End_{\Ocal_S}(\underline{\Mcal})$ is a ring homomorphism.
    \item $\rank_{\Ocal_{S_k}}(\Omega_0) = r$, $\rank_{\Ocal_{S_k}}(\Omega_1) = s$.
\end{assertionlist}

\end{definition}

Unitary zips of type $(r,s)$ form an algebraic stack over $\FF_q$ which is isomorphic to $\GZip^\mu$. By our choice of convention, the standard representation $\Std \colon G\to \GL_{n,\FF_{q}}$ corresponds to the vector bundle $\Mcal_1$. Furthermore, the sub-$P$-representation $V_{0,P}\subset V_0=k^n$ corresponds to the vector bundle $\Omega_1\subset \Mcal_1$. We define a flagged unitary zip as a tuple $(\underline{\Mcal},\iota, \Fcal_{\bullet})$ where $(\underline{\Mcal},\iota)$ is a unitary zip, and
\begin{equation}\label{filtrF}
    0=\Fcal_n\subset \Fcal_{n-1}\subset \dots \subset \Fcal_{r}=\Omega_1 \subset \dots \subset \Fcal_1\subset \Fcal_0= \Mcal_1
\end{equation}
is a filtration by locally free, locally direct factor $\Ocal_{S_k}$-submodules such that $\rank(\Fcal_{i})=n-i$ for $1\leq i \leq n$). Furthermore, put $\Lcal_i = \Fcal_{i-1}/ \Fcal_i$ (for $1\leq i \leq n$), and define
\begin{equation}
    \Lcal(\underline{a})= \bigotimes_{i=1}^n \Lcal_i^{a_i} \qquad \textrm{for } \underline{a}=(a_1, \dots , a_n)\in \ZZ^n.
\end{equation}
Our convention for the ordering of the flag \eqref{filtrF} is consistent with the choice of the Borel subgroup $B$. Indeed, $B$ stabilizes the filtration $\Span_k(e_n)\subset \Span_k(e_n,e_{n-1})\subset \dots \subset \Span_k(e_n, \dots , e_2) \subset k^n$ and the torus $T$ acts on the graded pieces respectively by $t_n, \dots , t_1$ in this order. We can extend the filtration \eqref{filtrF} uniquely to a symplectic flag of $\Mcal$ by taking the orthogonals
\begin{equation}
    \Mcal_1=\Mcal_1^\perp \subset \Fcal_1^\perp \subset \dots \subset \Fcal_r^\perp \subset \dots \subset \Fcal_{n-1}^\perp \subset \Fcal_n^\perp = \Mcal.
\end{equation}
Furthermore, we have $\Fcal_r^\perp = \Omega_1^\perp = \Omega_0\oplus \Mcal_1$. Intersecting the filtration with $\Mcal_0$, we obtain a full flag of $\Mcal_0$ as follows:
\begin{equation}
    0 \subset \Fcal_1^\perp\cap \Mcal_0 \subset \dots \subset \Fcal_r^\perp\cap \Mcal_0\subset \dots \subset \Fcal_{n-1}^\perp \cap \Mcal_0 \subset \Mcal_0
\end{equation}
and we have $\Fcal_r^\perp\cap \Mcal_0=\Omega_0$. 
For $d\leq s$, consider the composition 
\begin{equation}
\xymatrix@M=5pt{
    \Fcal_d^\perp\cap \Mcal_0 \ar[r] & \Mcal_0 \ar[r]^{V} & \Omega_1^{(p)} \ar[r] & \Omega_1^{(q)}/\Fcal^{(q)}_{r+d} . 
}
\end{equation}
It is a map of vector bundles of rank $d$. We have $\Fcal_d^\perp\cap \Mcal_0 = \frac{\Fcal_d^\perp}{\Mcal_1} =  \left(\frac{\Mcal_1}{\Fcal_d}\right)^\vee$ and hence $\det(\Fcal_d^\perp\cap \Mcal_0) = \det(\Mcal_1/\Fcal_d)^{-1}=(\Lcal_1\otimes \dots \otimes \Lcal_d)^{-1}$. Similarly, we find $\det(\Omega_1/\Fcal_{r+d})=\Lcal_{r+1}\otimes \dots \otimes \Lcal_{r+d}$. Thus, we obtain a map
\begin{equation}
    (\Lcal_1\otimes \dots \otimes \Lcal_d)^{-1} \to (\Lcal_{r+1}\otimes \dots \otimes \Lcal_{r+d})^q
\end{equation}
which is the same as a section of $H^0(S,\Lcal(\lambda_d))$. This section is the partial Hasse invariant $\Ha_d$ of weight $\lambda_d$ explained in the previous section. We now assume $s<d\leq n-1$. In this case, consider the composition map
\begin{equation}\label{partialH2}
   \frac{\Fcal_{d-s}^{(q)}}{\Omega_1^{(q)}}  \subset  \frac{\Mcal_1^{(q)}}{\Omega_1^{(q)}} \xrightarrow{ \ \ F \ \ } \Mcal_0 \longrightarrow  \frac{\Mcal_0}{\Fcal_d^\perp\cap \Mcal_0}.
\end{equation}
Taking determinants, we obtain a map $(\Lcal_{d-s+1}\otimes \dots \otimes \Lcal_r)^q \to (\Lcal_{d+1}\otimes \dots \otimes \Lcal_n)^{-1}$. Thus, we constructed a section of weight
\begin{equation}
(\underbrace{0, \dots ,0}_{\textrm{$d$ times}},\underbrace{-1, \dots , -1}_{\textrm{$n-d$ times}}) + (\underbrace{0, \dots , 0}_{\textrm{$d-s$ times}}, \underbrace{-q, \dots ,-q}_{\textrm{$n-d$ times}}, \underbrace{0, \dots ,0}_{\textrm{$s$ times}}) =  \lambda_d-\lambda_{\det} . 
\end{equation}
Up to multiplication by the determinant, the section given by \eqref{partialH2} corresponds to the partial Hasse invariant $\Ha_d$ (in particular, the vanishing loci coincide).

Lastly, we give a modular interpretation for the $\mu$-ordinary Hasse invariant $\Ha_\mu$. This is a section of a line bundle on $\GZip^\mu$ whose vanishing locus is exactly the closure of the unique codimension one stratum. The existence and properties of such sections were proved for general groups in \cite{Koskivirta-Wedhorn-Hasse}. In the general unitary case, Goldring--Nicole have constructed in \cite{Goldring-Nicole-mu-Hasse} such sections on the special fiber of the corresponding Shimura varieties at places of good reduction. Their construction is based on the crystalline cohomology $H^i_{\crys}(A)$ of an abelian variety $A$ (endowed with unitary-type additional structure). We give below another modular interpretation of $\Ha_\mu$ on the stack $\GZip^\mu$. Since the $G$-zip attached to an abelian variety $A$ is defined using $H^1_{\dR}(A)=H^1_{\crys}(A) / pH^1_{\crys}(A)$, it shows that $\Ha_\mu$ also admits an interpretation in terms of the mod $p$ reduction of the crystalline cohomology of $A$. Returning to the stack $\GZip^\mu$, consider the composition map
\begin{equation}\label{Hamu}
\xymatrix@M=5pt{
    \Omega_1 \ar[r]^V & \Omega_0^{(q)} \ar[r]^{V^{(q)}} & \Omega_1^{(q^2)}.
}
\end{equation}
Taking determinants, we obtain a section $\Ha_\mu$ of $\det(\Omega_1)^{q^2-1}=(\Lcal_{r+1}\otimes \dots \otimes \Lcal_n)^{q^2-1}$ over $\GZip^\mu$. It is easy to see that this section is non-zero, non-invertible. Hence its vanishing locus must be a codimension one closed substack of $\GZip^\mu$. Since there is a unique stratum of codimension $1$, it shows that the section $\Ha_\mu$ is a $\mu$-ordinary Hasse invariant. This would be difficult to prove directly from the modular interpretation \eqref{Hamu}. We see on this example that $\Ha_\mu$ is different from all flag partial Hasse invariants $\{\Ha_d\}_{1\leq d\leq n-1}$. This is in contrast to the case when $P$ is defined over $\FF_q$ (see Lemma \ref{Lemma-zipflag}).

In the case of $r=2$, $s=1$, we gave in \cite[Lemma 6.3.1]{Imai-Koskivirta-vector-bundles} a group-theoretical representation of $\Ha_\mu$. In \cite[Proposition 6.3.2]{Imai-Koskivirta-vector-bundles}, the section $\Ha_\mu$ has weight $(q+1,q+1,q^2+q)=\lambda_{\det}+(0,0,q^2-1)$. It coincides with the section constructed above up to multiplication by the determinant.

\subsection{The unitary split case}

We now consider the case of a unitary Shimura variety at a prime $p$ of good reduction which is split in the totally imaginary quadratic field $\mathbf{E}$. In this case, $G$ is isomorphic to $\GL_{n,\FF_p}$.

\subsubsection{Group theory}

Define $G=\GL_{n,\FF_q}$ (take $q=p$ in the context of Shimura varieties). Define a cocharacter $\mu \colon \GG_{\mathrm{m},k}\to G_k$ as in the previous section by $\mu(x)=\diag(xI_r,I_s)$. Write again $\Zcal_\mu=(G,P,L,Q,M,\varphi)$ for the attached zip datum. Let again $B$ denote the lower-triangular Borel and $T$ the diagonal torus. Again, we identify $X^*(T)=\ZZ^n$ as in \S \ref{sec-gp-inert} and define $V_{0,P}\subset V_0$ similarly. In this case, the Galois action is trivial, so the element $z=w_{0,I} w_0$ is given by the matrix
\begin{equation}
    z=\left( 
    \begin{matrix}
    0 & I_r \\ I_s & 0    
    \end{matrix} \right).
\end{equation}
Consider again the function $\Delta_d$ defined in \eqref{hchi}. The function $G_k\to \AA^1_k$, $g\mapsto \Delta_d(gz)$ corresponds to the pull-back $\Ha_d=\psi^*(\Delta_d)$. The sections $\{\Ha_d\}_{1\leq d \leq n-1}$ form a complete set of partial Hasse invariants for $\GF^\mu$. If $d\leq r$, the section $\Ha_d$ lies in $H^0(\GF^\mu,\Vcal_{\flag}(\lambda_d))$, where 
\begin{equation}
    \lambda_d =\chi_d-q w_{0,I}\chi_d = (\underbrace{1, \dots ,1}_{\textrm{$d$ times}},\underbrace{0, \dots ,0}_{\textrm{$n-d$ times}}) + (\underbrace{0, \dots , 0}_{\textrm{$r-d$ times}}, \underbrace{-q, \dots ,-q}_{\textrm{$d$ times}}, \underbrace{0, \dots ,0}_{\textrm{$s$ times}}). 
\end{equation}
 Similarly, in the case $r<d\leq n-1$, the weight of $\Ha_d$ is 
\begin{equation}
    \lambda_d =\chi_d-q w_{0,I} \chi_d = (\underbrace{1, \dots ,1}_{\textrm{$d$ times}},\underbrace{0, \dots ,0}_{\textrm{$n-d$ times}}) + (\underbrace{-q, \dots , -q}_{\textrm{$r$ times}}, \underbrace{0, \dots ,0}_{\textrm{$n-d$ times}}, \underbrace{-q, \dots ,-q}_{\textrm{$d-r$ times}}) . 
\end{equation}

The determinant function $\det  \colon  G_k\to \GG_{\mathrm{m}}$ defines a nowhere vanishing section $\det\in H^0(\GZip^\mu,\Vcal_I(\lambda_{\det}))$ where $\lambda_{\det}=(-(q-1), \dots , -(q-1))$. In particular, the line bundles $\Vcal_I(\lambda_{\det})$ and $\Vcal_{\flag}(\lambda_{\det})$ are trivial. Hence, the line bundles  $\{\Vcal_{\flag}(\lambda)\}_\lambda$ 
on $\GF^\mu$ are parametrized up to isomorphism by the group $\ZZ^n/\ZZ \lambda_{\det}$. The section $\Ha_r$ is the ordinary Hasse invariant of $\GZip^\mu$, which illustrates again Lemma \ref{Lemma-zipflag}.

\subsubsection{Partial Hasse invariants}

We give a modular interpretation for partial Hasse invariants. In this case, the corresponding notion of 
a unitary zip $(\underline{\Mcal},\iota)$ is similar to Definition \ref{def-unit-zip} except that $\iota$ is now an action of $\FF_q\times \FF_q$ on $\underline{\Mcal}$. This implies that $\Mcal$ splits again as $\Mcal=\Mcal_0\oplus \Mcal_1$, but this time the two components are stable under the Frobenius and Verschiebung maps. As a consequence, $(\underline{\Mcal},\iota)$ is completely and uniquely determined up to isomorphism by the datum $(\Mcal_1,\Omega_1, F, V)$ where $F \colon \Mcal_1^{(q)}\to \Mcal_1$ and $V \colon \Mcal_1\to \Mcal_1^{(q)}$. Define the flag space as the stack of triples $(\underline{\Mcal},\iota, \Fcal_\bullet)$ where $\Fcal_\bullet$ is a flag of $\Mcal_1$ as in \eqref{filtrF}. Define also $\Lcal_i$ for $1\leq i \leq n$ and $\Lcal(\underline{a})$ for $\underline{a}\in \ZZ^n$ as in \S\ref{sec-mod-inert}. For $d\leq r$, consider the composition
\begin{equation}
\xymatrix@1@M=5pt{
\Fcal_{r-d}^{(q)}/\Omega^{(q)}_1  \subset  \Mcal^{(q)}_1/\Omega^{(q)}_1  \ar[r]^-{F} & \Mcal_1 \ar[r] & \Mcal_1/\Fcal_{d}.
}
\end{equation}
We have $\det\left(\Fcal_{r-d}/\Omega_1\right)=\Lcal_r\otimes \dots \otimes \Lcal_{r-d+1}$ and $\det \left( \Mcal_1/\Fcal_{d}\right) = \Lcal_1\otimes \dots \otimes \Lcal_{d}$. Taking the determinant, we obtain a section of weight $\lambda_d$, and it corresponds to the Hasse invariant $\Ha_d$. Similarly, for $r<d\leq n-1$, consider the composition
\begin{equation}
\xymatrix@1@M=5pt{
\Fcal_{d}  \subset  \Mcal_1  \ar[r]^-{V} & \Omega^{(q)}_1 \ar[r] & \Omega^{(q)}_1 / \Fcal^{(q)}_{r+n-d}.
}
\end{equation}
We have $\det\left(\Fcal_{d}\right)=\Lcal_{d+1}\otimes \dots \otimes \Lcal_{n}$ and $\det \left(  \Omega_1 / \Fcal_{r+n-d} \right) = \Lcal_{r+1}\otimes \dots \otimes \Lcal_{r+n-d}$. Taking the determinant, we obtain a section of weight
\begin{equation}
  (\underbrace{0, \dots ,0}_{\textrm{$d$ times}},\underbrace{-1, \dots ,-1}_{\textrm{$n-d$ times}}) + (\underbrace{0, \dots , 0}_{\textrm{$r$ times}}, \underbrace{q, \dots ,q}_{\textrm{$n-d$ times}}, \underbrace{0, \dots ,0}_{\textrm{$d-r$ times}}) = \lambda_d - \lambda_{\det}.
\end{equation}
This section coincides with the partial Hasse invariant $\Ha_d$ up to multiplication by the determinant.

\newcommand{\etalchar}[1]{$^{#1}$}

\noindent
Naoki Imai\\
Graduate School of Mathematical Sciences, The University of Tokyo, 
3-8-1 Komaba, Meguro-ku, Tokyo, 153-8914, Japan \\
naoki@ms.u-tokyo.ac.jp\\ 

\noindent
Jean-Stefan Koskivirta\\
Department of Mathematics, Faculty of Science, Saitama University, 
255 Shimo-Okubo, Sakura-ku, Saitama City, Saitama 338-8570, Japan \\
jeanstefan.koskivirta@gmail.com


\begin{thebibliography}{ABD{\etalchar{+}}66}
	\providecommand{\url}[1]{\texttt{#1}}
	\providecommand{\urlprefix}{URL }
	\providecommand{\eprint}[2][]{\url{#2}}
	
	\bibitem[ABD{\etalchar{+}}66]{SGA3}
	M.~Artin, J.~E. Bertin, M.~Demazure, P.~Gabriel, A.~Grothendieck, M.~Raynaud
	and J.-P. Serre, S{G}{A}3: Sch\'emas en groupes, vol. 1963/64, Institut des
	Hautes \'Etudes Scientifiques, Paris, 1965/1966.
	
	\bibitem[AG05]{Andreatta-Goren-book}
	F.~Andreatta and E.~Goren, Hilbert modular forms: mod $p$ and $p$-adic aspects,
	vol. 173 of Mem. Amer. Math Soc., Amer. Math. Soc., 2005.
	
	\bibitem[BH17]{Bijakowski-Hernandez-Groupes-p-divisibles}
	S.~Bijakowski and V.~Hernandez, Groupes $p$-divisibles avec condition
	de~Pappas-Rapoport et invariants de Hasse, Journal de l'\'Ecole polytechnique
	- Math\'ematiques 4 (2017), 935--972.
	
	\bibitem[Bij18]{bijakowski-partial-Hasse}
	S.~Bijakowski, Partial Hasse Invariants, Partial Degrees, and the Canonical
	Subgroup, Canadian Journal of Mathematics 70 (2018), no.~4, 742--772.
	
	\bibitem[Del79]{Deligne-Shimura-varieties}
	P.~Deligne, Vari\'et\'es de {S}himura: {I}nterpr\'etation modulaire, et
	techniques de construction de mod\`eles canoniques, in Automorphic forms,
	representations and {$L$}-functions, {P}art 2, edited by A.~Borel and
	W.~Casselman, vol.~33 of Proc. Symp. Pure Math., Amer. Math. Soc.,
	Providence, RI, 1979 pp. 247--289, {P}roc. {S}ympos. {P}ure {M}ath., {O}regon
	{S}tate {U}niv., {C}orvallis, {OR}., 1977.
	
	\bibitem[DK17]{Diamond-Kassaei}
	F.~Diamond and P.~Kassaei, Minimal weights of {H}ilbert modular forms in
	characteristic {$p$}, Compos. Math. 153 (2017), no.~9, 1769--1778.
	
	\bibitem[DK23]{Diamond-Kassaei-cone-minimal}
	F.~Diamond and P.~L. Kassaei, The {C}one of {M}inimal {W}eights for {M}od $p$
	{H}ilbert {M}odular {F}orms, Int. Math. Res. Not. IMRN  (2023), no.~14,
	12148--12171.
	
	\bibitem[Don85]{Donkin-good-filtrations-LNM}
	S.~Donkin, Rational representations of algebraic groups: Tensor products and
	filtration, vol. 1140 of Lect. Notes in Math., Springer-Verlag, Berlin, 1985.
	
	\bibitem[GGN17]{Garibaldi-Guralnick-nakano-globally-irreducible}
	S.~Garibaldi, R.~M. Guralnick and D.~K. Nakano, Globally irreducible Weyl
	modules, Journal of Algebra 477 (2017), 69--87.
	
	\bibitem[GK18]{Goldring-Koskivirta-global-sections-compositio}
	W.~Goldring and J.-S. Koskivirta, Automorphic vector bundles with global
	sections on {$G$}-{Z}ip$^{\Zcal}$-schemes, Compositio Math. 154 (2018),
	2586--2605, \href{https://doi.org/10.2140/ant.2023.17.923}{DOI},
	\href{http://msp.org/idx/mr/4582533}{MR 4582533}.
	
	\bibitem[GK19a]{Goldring-Koskivirta-Strata-Hasse}
	W.~Goldring and J.-S. Koskivirta, Strata {H}asse invariants, {H}ecke algebras
	and {G}alois representations, Invent. Math. 217 (2019), no.~3, 887--984,
	\href{https://doi.org/10.1007/s00222-019-00882-5}{DOI}.
	
	\bibitem[GK19b]{Goldring-Koskivirta-zip-flags}
	W.~Goldring and J.-S. Koskivirta, Stratifications of flag spaces and
	functoriality, IMRN 2019 (2019), no.~12, 3646--3682.
	
	\bibitem[GN17]{Goldring-Nicole-mu-Hasse}
	W.~Goldring and M.-H. Nicole, The $\mu$-ordinary {H}asse invariant of unitary
	{S}himura varieties, J. Reine Angew. Math. 728 (2017), 137--151.
	
	\bibitem[Gor01]{Goren-partial-hasse}
	E.~Goren, Hasse invariants for {H}ilbert modular varieties, Israel J. Math. 122
	(2001), 157--174.
	
	\bibitem[GS69]{Griffiths-Schmid-homogeneous-complex-manifolds}
	P.~Griffiths and W.~Schmid, Locally homogeneous complex manifolds, Acta. Math.
	123 (1969), 253--302.
	
	\bibitem[Her18]{Hernandez-invariants-de-Hasse}
	V.~Hernandez, Invariants de Hasse $\mu $-ordinaires, Annales de l'Institut
	Fourier 68 (2018), no.~4, 1519--1607.
	
	\bibitem[HN17]{He-Nie-mu-ordinary}
	X.~He and S.~Nie, On the $\mu$-ordinary locus of a Shimura variety, Advances in
	Mathematics 321 (2017), 513--528.
	
	\bibitem[IK21]{Imai-Koskivirta-vector-bundles}
	N.~Imai and J.-S. Koskivirta, Automorphic vector bundles on the stack of
	$G$-zips, Forum Math. Sigma 9 (2021), Paper No. e37, 31 pp.
	
	\bibitem[IK22]{Goldring-Imai-Koskivirta-weights}
	N.~Imai and J.-S. Koskivirta, Weights of mod $p$ automorphic forms and partial
	Hasse invariants, 2022, with an appendix by W. Goldring. Preprint,
	arXiv:2211.16207.
	
	\bibitem[Jan03]{jantzen-representations}
	J.~Jantzen, Representations of algebraic groups, vol. 107 of Math. Surveys and
	Monographs, American Mathematical Society, Providence, RI, 2nd edn., 2003.
	
	\bibitem[Kis10]{Kisin-Hodge-Type-Shimura}
	M.~Kisin, Integral models for {S}himura varieties of abelian type, J. Amer.
	Math. Soc. 23 (2010), no.~4, 967--1012.
	
	\bibitem[KKLV89]{Knop-Kraft-Luna-Vust-Local-properties}
	F.~Knop, H.~Kraft, D.~Luna and T.~Vust, Local {P}roperties of {A}lgebraic
	{G}roup {A}ctions, in Transformationsgruppen und Invariantentheorie, vol.~13
	of DMV Sem., Birkhauser, 1989 pp. 63--75.
	
	\bibitem[Kos19]{Koskivirta-automforms-GZip}
	J.-S. Koskivirta, Automorphic forms on the stack of {$G$}-zips, Results Math.
	74 (2019), no.~3, Paper No. 91, 52 pp.
	
	\bibitem[KW18]{Koskivirta-Wedhorn-Hasse}
	J.-S. Koskivirta and T.~Wedhorn, Generalized {$\mu$}-ordinary {H}asse
	invariants, J. Algebra 502 (2018), 98--119.
	
	\bibitem[LS18]{Lan-Stroh-stratifications-compactifications}
	K.-W. Lan and B.~Stroh, Compactifications of subschemes of integral models of
	{S}himura varieties, Forum Math. Sigma 6 (2018), e18, 105pp.
	
	\bibitem[Mil90]{Milne-ann-arbor}
	J.~Milne, Canonical Models of (Mixed) {S}himura varieties and automorphic
	vector bundles, in Automorphic forms, {S}himura varieties, and
	{$L$}-functions, {V}ol.\ {I}, edited by L.~Clozel and J.~Milne, vol.~11 of
	Perspect. Math., Academic Press, Boston, MA, 1990 pp. 283--414, proc. Conf.
	held in {A}nn {A}rbor, {MI}, 1988.
	
	\bibitem[Moo04]{Moonen-Serre-Tate}
	B.~Moonen, Serre-{T}ate theory for moduli spaces of {PEL}-type, Ann. Sci. ENS
	37 (2004), no.~2, 223--269.
	
	\bibitem[PWZ11]{Pink-Wedhorn-Ziegler-zip-data}
	R.~Pink, T.~Wedhorn and P.~Ziegler, Algebraic zip data, Doc. Math. 16 (2011),
	253--300.
	
	\bibitem[PWZ15]{Pink-Wedhorn-Ziegler-F-Zips-additional-structure}
	R.~Pink, T.~Wedhorn and P.~Ziegler, ${F}$-zips with additional structure,
	Pacific J. Math. 274 (2015), no.~1, 183--236.
	
	\bibitem[Ros61]{Rosenlicht-toroidal}
	M.~Rosenlicht, Toroidal algebraic groups, Proc. Amer. Math. Soc. 12 (1961),
	984--988.
	
	\bibitem[RR85]{Ramanan-Ramanathan-projective-normality}
	S.~Ramanan and A.~Ramanathan, Projective normality of flag varieties and
	{S}chubert varieties, Invent. Math. 79 (1985), 217--224.
	
	\bibitem[Spr98]{Springer-Linear-Algebraic-Groups-book}
	T.~Springer, Linear Algebraic Groups, vol.~9 of Progress in Math., Birkhauser,
	2nd edn., 1998.
	
	\bibitem[Ste63]{steinberg-reps-alg-gps-nagoya}
	R.~Steinberg, Representations of Algebraic Groups, Nagoya Mathematical Journal
	22 (1963), 33-56.
	
	\bibitem[SYZ21]{Shen-Yu-Zhang-EKOR}
	X.~Shen, C.-F. Yu and C.~Zhang, E{KOR} strata for {S}himura varieties with
	parahoric level structure, Duke Math. J. 170 (2021), no.~14, 3111--3236.
	
	\bibitem[Vas99]{Vasiu-Preabelian-integral-canonical-models}
	A.~Vasiu, Integral canonical models of {S}himura varieties of preabelian type,
	Asian J. Math. 3 (1999), 401--518.
	
	\bibitem[Wed99]{Wedhorn-ordinariness-Shimura-varieties}
	T.~Wedhorn, Ordinariness in good reductions of {S}himura varieties of
	{PEL}-type, Ann. Sci. ENS 32 (1999), no.~5, 575--618.
	
	\bibitem[Wed14]{Wedhorn-bruhat}
	T.~Wedhorn, Bruhat strata and ${F}$-zips with additional structure, M\"unster
	J. Math. 7 (2014), no.~2, 529--556.
	
	\bibitem[Wor13]{Wortmann-mu-ordinary}
	D.~Wortmann, The $\mu$-ordinary locus for {S}himura varieties of {H}odge type,
	2013, preprint, arXiv:1310.6444.
	
	\bibitem[Zha18]{Zhang-EO-Hodge}
	C.~Zhang, Ekedahl-{O}ort strata for good reductions of {S}himura varieties of
	{H}odge type, Canad. J. Math. 70 (2018), no.~2, 451--480.
	
\end{thebibliography}
\end{document}